\documentclass[reqno]{amsart}

\usepackage{mathabx}

\usepackage{tikz-cd}
\usepackage{verbatim}

\usepackage{fullpage,dsfont}

\usepackage{hyperref} 

\usepackage{parskip}
\makeatletter 
\def\thm@space@setup{%
 \thm@preskip=\parskip \thm@postskip=0pt
}
\def\th@remark{%
  \thm@headfont{\itshape}%
  \normalfont 
  \thm@preskip\parskip \thm@postskip=0pt
}
\makeatother

\usepackage[nobysame,alphabetic,initials,msc-links]{amsrefs}

\usepackage[final]{pdfpages}

\usepackage{multirow}

\usepackage{array}

\usetikzlibrary{braids}

\usepackage{mystix2}

\DefineSimpleKey{bib}{how}
\DefineSimpleKey{bib}{mrclass}
\DefineSimpleKey{bib}{mrnumber}
\DefineSimpleKey{bib}{fjournal}
\DefineSimpleKey{bib}{mrreviewer}

\renewcommand{\PrintDOI}[1]{%
  \href{http://dx.doi.org/#1}{{\tt DOI:#1}}%
}
\renewcommand{\eprint}[1]{#1}
\BibSpec{book}{%
    +{}  {\PrintPrimary}                {transition}
    +{.} { \PrintDate}                  {date}
    +{.} { \textit}                     {title}
    +{.} { }                            {part}
    +{:} { \textit}                     {subtitle}
    +{,} { \PrintEdition}               {edition}
    +{}  { \PrintEditorsB}              {editor}
    +{,} { \PrintTranslatorsC}          {translator}
    +{,} { \PrintContributions}         {contribution}
    +{,} { }                            {series}
    +{,} { \voltext}                    {volume}
    +{,} { }                            {publisher}
    +{,} { }                            {organization}
    +{,} { }                            {address}
    +{,} { }                            {status}
    +{,} { \PrintDOI}                   {doi}
    +{,} { \PrintISBNs}                 {isbn}
    +{}  { \parenthesize}               {language}
    +{}  { \PrintTranslation}           {translation}
    +{;} { \PrintReprint}               {reprint}
    +{.} { }                            {note}
    +{.} {}                             {transition}
    +{}  {\SentenceSpace \PrintReviews} {review}
}
\BibSpec{article}{%
    +{}  {\PrintAuthors}                {author}
    +{,} { \textit}                     {title}
    +{.} { }                            {part}
    +{:} { \textit}                     {subtitle}
    +{,} { \PrintContributions}         {contribution}
    +{.} { \PrintPartials}              {partial}
    +{,} { }                            {journal}
    +{}  { \textbf}                     {volume}
    +{}  { \PrintDatePV}                {date}
    +{,} { \issuetext}                  {number}
    +{,} { \eprintpages}                {pages}
    +{,} { }                            {status}
    +{,} { \PrintDOI}                   {doi}
    +{,} { \eprint}        {eprint}
    +{}  { \parenthesize}               {language}
    +{}  { \PrintTranslation}           {translation}
    +{;} { \PrintReprint}               {reprint}
    +{.} { }                            {note}
    +{.} {}                             {transition}
    +{}  {\SentenceSpace \PrintReviews} {review}
}
\BibSpec{collection.article}{%
    +{}  {\PrintAuthors}                {author}
    +{,} { \textit}                     {title}
    +{.} { }                            {part}
    +{:} { \textit}                     {subtitle}
    +{,} { \PrintContributions}         {contribution}
    +{,} { \PrintConference}            {conference}
    +{}  {\PrintBook}                   {book}
    +{,} { }                            {booktitle}
    +{,} { \PrintDateB}                 {date}
    +{,} { pp.~}                        {pages}
    +{,} { }                            {publisher}
    +{,} { }                            {organization}
    +{,} { }                            {address}
    +{,} { }                            {status}
    +{,} { \PrintDOI}                   {doi}
    +{,} { \eprint}        {eprint}
    +{}  { \parenthesize}               {language}
    +{}  { \PrintTranslation}           {translation}
    +{;} { \PrintReprint}               {reprint}
    +{.} { }                            {note}
    +{.} {}                             {transition}
    +{}  {\SentenceSpace \PrintReviews} {review}
}
\BibSpec{misc}{%
  +{}{\PrintAuthors}  {author}
  +{,}{ \textit}      {title}
  +{.}{ }             {how}
  +{}{ \parenthesize} {date}
  +{,} { available at \eprint}        {eprint}
  +{,}{ available at \url}{url}
  +{,}{ }             {note}
  +{.}{}              {transition}
}

\usepackage{amssymb, amsfonts, amsxtra, amsmath}

\usepackage{mathrsfs}
\usepackage{mathdots}
\usepackage{wasysym}
\usepackage[all]{xy}
\usepackage{bbm}
\usepackage{calc}
\usepackage{accents}
\usepackage{relsize}

\usepackage[bbgreekl]{mathbbol}

\usepackage{stackengine}

\numberwithin{equation}{section}

\usepackage{caption} 
\usepackage{subcaption} 
\usepackage{enumerate}

\usepackage{dsfont}

\newtheorem{Theorem}{Theorem}[section]
\newtheorem*{Theorem*}{Theorem}
\newtheorem{Def}[Theorem]{Definition}
\newtheorem{Lem}[Theorem]{Lemma}
\newtheorem{Prop}[Theorem]{Proposition}
\newtheorem{Cor}[Theorem]{Corollary}

\newtheorem{Rem}[Theorem]{Remark}

\newcommand{\C}{\mathbb{C}}
\newcommand{\R}{\mathbb{R}}

\newcommand{\Ker}{\mathrm{Ker}}

\newcommand{\mcK}{\mathcal{K}}
\newcommand{\mcB}{\mathcal{B}}
\newcommand{\id}{\mathrm{id}}

\newcommand{\Hsp}{\mathcal{H}}
\newcommand{\Gsp}{\mathcal{G}}

\newcommand{\G}{\mathbb{G}}

\newcommand{\Ll}{\mathbb{L}}
\newcommand{\Hh}{\mathbb{H}}
\newcommand{\X}{\mathbb{X}}
\newcommand{\Y}{\mathbb{Y}}

\newcommand{\opp}{\mathrm{op}}

\newcommand{\msD}{\mathscr{D}}
\newcommand{\msN}{\mathscr{N}}
\newcommand{\msG}{\mathscr{G}}
\newcommand{\Rep}{\mathrm{Rep}}
\newcommand{\ext}{\mathrm{ext}}
\newcommand{\Ind}{\mathrm{Ind}}

\newcommand{\barboxtimes}{\overline{\boxtimes}}

\newcommand{\mcX}{\mathcal{X}}

\newcommand{\red}{\mathrm{red}}
\newcommand{\triv}{\mathrm{triv}}
\newcommand{\Res}{\mathrm{Res}}

\newcommand{\vp}{\varphi}
\newcommand{\ov}{\overline}
\newcommand{\oon}{\operatorname}
\newcommand{\mc}{\mathcal}
\newcommand{\mscr}{\mathscr}
\newcommand{\eps}{\varepsilon}
\newcommand{\wh}{\widehat}

\newcommand{\Mult}{\operatorname{M}}

\newcommand{\TensRep}{\otop}

\newcommand{\braidnew}[2]{{}^{#1}{\fdiagovrdiag}^{#2}}

\newcommand{\HH}{\mathbb{H}}
\newcommand{\GG}{\mathbb{G}}

\newcommand{\mf}{\mathfrak}

\newcommand{\ww}{\mathrm{W}}
\newcommand{\WW}{{\mathds{V}\!\!\text{\reflectbox{$\mathds{V}$}}}}
\newcommand{\Ww}{\mathds{W}}
\newcommand{\wW}{\text{\reflectbox{$\Ww$}}\:\!}

\title{The standard construction for cocycle twisted and braided tensor product W$^*$-algebras}
\author{K. De Commer and J. Krajczok}
\address{Vrije Universiteit Brussel}
\email{kenny.de.commer@vub.be}
\email{jacek.krajczok@vub.be}



\begin{document}

\begin{abstract}
Given a locally compact quantum group $\G$ and a (generalized) dual unitary $2$-cocycle $\hat{\Omega}$, any W$^*$-algebra $A$ with a $\G$-action can be twisted into a new W$^*$-algebra $A_{\hat{\Omega}}$ with an action by the cocycle twist $\G_{\hat{\Omega}}$ of $\G$. We show how, in general, the standard space $L^2(A_{\hat{\Omega}})$, with its standard $\G_{\hat{\Omega}}$-representation, can be seen as a twist  of $L^2(A)$ with its standard $\G$-representation. We then apply this general result in the special case of (generalized) Drinfeld doubles.
\end{abstract}

\maketitle

\section*{Introduction}

A \emph{locally compact quantum group} $\G = (M,\Delta)$  \cites{KV00,KV03} consists of a W$^*$-algebra $M$, together with a faithful, normal, unital $*$-homomorphism $\Delta\colon M \rightarrow M\bar{\otimes} M$, the \emph{coproduct}, satisfying \emph{coassociativity},
\begin{equation}\label{EqCoasso}
(\Delta\otimes \id)\Delta = (\id\otimes \Delta)\Delta,
\end{equation}
and admitting normal, semi-finite, faithful (nsf) weights $\varphi,\psi\colon M^+ \rightarrow [0,\infty]$ satisfying the left, resp.\ right invariance condition
\[
\vp((\omega\otimes \id)\Delta(x)) = \vp(x),\quad \psi((\id\otimes \omega)\Delta(x)) = \psi(x),\qquad \forall x \in M^+,\, \textrm{normal states }\omega\textrm{ on }M.
\]
We write $M = L^{\infty}(\G)$ to guide intuition, and abbreviate LCQG = locally compact quantum group. 

Such a  LCQG $\G$ admits a W$^*$-category $\Rep(\G)$ of \emph{unitary (left) $\G$-representations}, i.e.~couples $(\Hsp,U)$ with
\begin{equation}\label{EqUnitRepDef}
\Hsp \textrm{ Hilbert space},\quad U \in M\bar{\otimes} \mcB(\Hsp) \textrm{ unitary \emph{corepresentation}:}\qquad(\Delta\otimes \id)U = U_{[13]}U_{[23]},
\end{equation}
where we use the leg numbering notation $U_{[23]} = 1\otimes U$ etc.\ (the brackets are used to shield off the leg numbering from other indices). Morphisms are given by bounded intertwiners, 
\[
\mathrm{Mor}((\Hsp,U),(\Hsp',U')) = \{T \in \mathcal{B}(\Hsp,\Hsp')\mid (1\otimes T)U = U'(1\otimes T)\}.
\]
Alternatively,  $\G$-representations may be viewed as non-degenerate contractive Hilbert space representations of the universal Banach algebra envelope $\hat{M}_u$ of the predual $M_*$ (see \cite{Kus01}), where $M_*$ is endowed with the convolution product $\omega*\chi = (\omega\otimes \chi)\Delta$ and the smallest norm  making
\begin{equation}\label{EqPredualRepIntro}
\hat{\pi}_U(\omega) := (\omega \otimes \id)(U)
\end{equation}
contractive for all unitary $\G$-representations. One shows that $\hat{M}_u$ is in fact a C$^*$-algebra (recall that a Banach algebra can have at most one C$^*$-algebra structure), and the unique bounded extension of \eqref{EqPredualRepIntro} to $\hat{M}_u$ gives a (functorial) one-to-one correspondence $(\Hsp,U) \mapsto (\Hsp,\hat{\pi}_U)$ with non-degenerate representations of $\hat{M}_u$ on $\Hsp$ (recall that non-degenerate contractive Hilbert space representations of a C$^*$-algebra are automatically $*$-preserving). This (alternative) construction gives the same C$^*$-algebra $\hat{M}_u$ as in \cite{Kus01}. Again, we write $\hat{M}_u = C^*(\G)$ (`universal group C$^*$-algebra') when this is useful for intuition.

One has that $\Rep(\G)$ is a tensor W$^*$-category by the tensor structure
\begin{equation}\label{EqTensProd1}
(\Hsp,U) \otimes_{\G} (\Hsp',U') := (\Hsp\otimes \Hsp',U \TensRep U'),\qquad U \TensRep U' := U_{[12]}U_{[13]}'. 
\end{equation}
Alternatively, this tensor product is implemented by the coproduct $\hat{\Delta}_u\colon \hat{M}_u\rightarrow \Mult(\hat{M}_u\otimes \hat{M}_u)$ into the multiplier C$^*$-algebra of the minimal tensor product, which is uniquely determined by the fact that
\[
\hat{\pi}_{U\TensRep U'} = (\hat{\pi}_U \otimes \hat{\pi}_{U'})\hat{\Delta}_u^{\opp} =: (\hat{\pi}_U*\hat{\pi}_{U'}), \qquad \forall (\Hsp,U),(\Hsp',U') \in \Rep(\G),
\]
where $\hat{\Delta}_u^{\opp}$ is the coproduct with the legs flipped.\footnote{The definitions above depend on conventions, which are unfortunately not uniform throughout the literature. We follow here the conventions as in \cite{DCKr24} for consistency.}

There is also a natural notion of a \emph{(left) $\G$-action} on a W$^*$-algebra $A$, defined as a (left) coaction by $(M,\Delta)$: i.e.\ a faithful, unital, normal $*$-homomorphism $\gamma\colon A \rightarrow M\bar{\otimes}A$ satisfying the coaction property 
\[
(\Delta\otimes \id)\gamma = (\id\otimes \gamma)\gamma.
\]
We refer to $(A,\gamma)$ as a \emph{$\G$-W$^*$-algebra}. By \cite{Vae01}*{Definition 3.5}, any such $\G$-W$^*$-algebra $(A,\gamma)$ admits a canonical unitary representation $U_{\gamma}$ on its standard Hilbert space $L^2(A)$ implementing the $\G$-action,
\[
\gamma(a) = U_{\gamma}^*(1\otimes a)U_{\gamma},\qquad \forall a \in A. 
\]

In this paper, we prove a structural result for standard implementations in the presence of a \emph{quasi-triangular structure} on $\G$. This consists of a unitary $\wh{\mcX}_u\in \Mult(\hat{M}_u\otimes \hat{M}_u)$ which is a \emph{skew bicharacter}, i.e.
\begin{equation}\label{EqSkewBich}
(\hat{\Delta}_u\otimes \id)\wh{\mcX}_u = \wh{\mcX}_{u[13]}\wh{\mcX}_{u[23]},\quad (\id\otimes \hat{\Delta}_u)\wh{\mcX}_u = \wh{\mcX}_{u [13]}\wh{\mcX}_{u [12]},
\end{equation}
which moreover satisfies the compatibility property
\begin{equation}\label{EqCocTwist}
\wh{\mcX}_u\hat{\Delta}_u(x)\wh{\mcX}_u^* = \hat{\Delta}^{\opp}_u(x),\qquad \forall x\in \hat{M}_u.
\end{equation}
For example, the Drinfeld double of a locally compact quantum group \cite{BV05} has a natural quasi-triangular structure. A quasi-triangular structure turns $\Rep(\G)$ into a \emph{unitary braided} tensor W$^*$-category, with braiding 
\[
\braidnew{\Hsp}{\Hsp'}\colon \Hsp \otimes \Hsp'\rightarrow \Hsp'\otimes \Hsp,\qquad  \braidnew{\Hsp}{\Hsp'}:= \wh{\mcX}_{U',U}\circ \Sigma_{\Hsp,\Hsp'},\qquad \wh{\mcX}_{U',U} := (\hat{\pi}_{U'} \otimes \hat{\pi}_{U})(\wh{\mcX}_u),
\]
where $\Sigma_{\Hsp,\Hsp'}$ is the usual flip map $\Hsp \otimes \Hsp'\rightarrow \Hsp'\otimes \Hsp$ (\cite{DCKr24}*{Proposition 3.4}). When it is convenient and $U,U'$ are clear from the context, we will also write $\wh{\mcX}_{\Hsp',\Hsp}=\wh{\mcX}_{U',U}$.

On the other hand, as shown in \cite{DCKr24}, a quasi-triangular structure allows to take the \emph{braided tensor product} $(A_1\barboxtimes A_2,\gamma_1\bowtie \gamma_2)$ of two $\G$-W$^*$-algebras $(A_1,\gamma_1)$ and $(A_2,\gamma_2)$. This is a W$^*$-algebraic version of the C$^*$-algebraic results in  \cites{NV10,MRW14,MRW16} (see also \cites{Kas15,BMRS19}, the W$^*$-algebraic results in \cite{Hou07} for the case $\G = \R$ and \cite{Moa18} for the case $\G$ \emph{discrete}, and e.g.\  \cite{Maj95} for an exposition in the purely algebraic realm). This braided tensor product can be concretely implemented as follows: viewing $A_i \subseteq \mcB(L^2(A_i))$ via the standard representation and putting
\begin{equation}\begin{split}\label{EqStandardRepBraid}
\pi_{\boxtimes,1}(a_1) &:=(\braidnew{L^2(A_2)}{L^2(A_1)}) (1\otimes a_1)(\braidnew{L^2(A_2)}{L^2(A_1)})^*=
\wh{\mc{X}}_{L^2(A_1),L^2(A_2)} (a_1\otimes 1)
\wh{\mc{X}}_{L^2(A_1),L^2(A_2)}^*
,\\
\pi_{\boxtimes,2}(a_2) &:= 
1\otimes a_2
\end{split}\end{equation}
for $a_i \in A_i$, define $A_1\barboxtimes A_2 \subseteq \mcB(L^2(A_1)\otimes L^2(A_2))$ via
\begin{equation}\label{EqIdentification}
A_1\barboxtimes A_2 = \overline{\operatorname{span}}^{\sigma\textrm{-weak}}\{\pi_{\boxtimes,1}(a_1)\pi_{\boxtimes,2}(a_2)\mid a_i\in A_i\}.
\end{equation}
Then \cite{DCKr24}*{Theorem 4.9} shows that $A_1\ov\boxtimes A_2$ is a von Neumann algebra,  and by \cite{DCKr24}*{Proposition 6.3} there is a well-defined action $\gamma_1\bowtie\gamma_2$ of $\G$ on $A_1\barboxtimes A_2$, uniquely defined via 
\begin{equation}\label{EqIdentification2}
(\gamma_1 \bowtie \gamma_2)(\pi_{\boxtimes,i}(a_i)) = (\id\otimes \pi_{\boxtimes,i})\gamma_i(a_i),\qquad \forall a_i \in A_i. 
\end{equation}
The main result of this paper is the following. 

\begin{Theorem*}
Let $(\G,\wh{\mcX}_u)$ be a quasi-triangular quantum group, and let $(A_1,\gamma_1),(A_2,\gamma_2)$ be two $\G$-W$^*$-algebras. Then there is a canonical unitary intertwiner of unitary $\G$-representations 
\[
I_\boxtimes \colon L^2(A_1 \barboxtimes A_2) \cong L^2(A_1) \otimes L^2(A_2), \qquad U_{\gamma_1\bowtie \gamma_2}\cong U_{\gamma_1} \TensRep U_{\gamma_2}.
\]
This unitary intertwines the standard representation of $A_1\ov{\boxtimes} A_2$ with the identity representation \eqref{EqIdentification}, and satisfies 
\[
I_\boxtimes J_{A_1\barboxtimes A_2}I^*_\boxtimes  =\wh{\mcX}_{L^2(A_1),L^2(A_2)}(J_{A_1}\otimes J_{A_2}),
\]
where $J_A\colon L^2(A) \rightarrow L^2(A)$ denotes the modular conjugation of a W$^*$-algebra $A$.
\end{Theorem*} 
At the moment, it is not clear for us how to characterise the image of the standard positive cone in $L^2(A_1\barboxtimes A_2)$ under $I_\boxtimes$ (for a description of the positive cone in the case of ordinary tensor product, see \cite{MT84}).

The above main theorem will be a consequence of quite general results on \emph{(generalized) cocycle twists}, which are dealt with in the first four sections:
\begin{itemize}
\item In the \emph{first section}, we introduce the notion of a \emph{linking quantum groupoid}, which is a convenient technical device to deal with generalized  unitary $2$-cocycles (i.e.~ \emph{Galois objects}).
\item  In the \emph{second section}, we study representations and actions of linking quantum groupoids and their duals. 
\end{itemize}
The details for the results in these two sections are not spelled out - many of them can be found in \cite{DC09} (see also \cite{BC17}), and in general the techniques required in this setting offer nothing new except for a slight adaptation of the results in  e.g.\ \cites{KV00,Kus01,KV03,Vae01}.
\begin{itemize}
\item In the \emph{third section}, we use the formalism of the first two sections to show how to induce actions on W$^*$-algebras and representations on Hilbert spaces across cocycle twists, and we then show our main abstract result which roughly states that `Induction commutes with the standard construction', see Theorem \ref{TheoIndL}. 
\item In the \emph{fourth section}, we specialize these results to the case of (genuine) $2$-cocycle twists, and single out the case of \emph{universally continuous $2$-cocycles} where one does not need to change Hilbert spaces when inducing. 
\end{itemize}
All of the above results depend on robust techniques, with results guided by a very general abstract formalism. However: 
\begin{itemize}
\item In the \emph{fifth section} we specialize even further to the setting of generalized Drinfeld double, and prove our main concrete result on the standard representation of braided tensor products, see Theorem \ref{TheoInductionBraided}. Here we need some more delicate machinery to lift certain dual unitary $2$-cocycles to the appropriate universal level. In particular, we will show that the universal C$^*$-algebra of $C_0$-functions on a generalized Drinfeld double is given by the maximal tensor product of the universal C$^*$-algebras of $C_0$-functions on its components, which as far as we know is a new result -- see Proposition \ref{prop3}. We then prove the theorem on quasi-triangular quantum group actions alluded to above.
\end{itemize}

\emph{Conventions}

We denote by $[S]$ the norm closure of the linear span of a subset $S$ of a normed space. We write $[S]^{\sigma\textrm{-weak}}$ for the $\sigma$-weak closure of the linear span of $S$, if $S$ is a subset of a W$^*$-algebra. The symbol $\odot$ stands for the algebraic tensor product of vector spaces. We will use leg numbering notation with flips, so that e.g.~$(a\otimes b\otimes c)_{[412]}=b\otimes c \otimes 1\otimes a$.

We denote by $\Sigma\colon \Hsp \otimes \Gsp \rightarrow \Gsp \otimes \Hsp$ the flip map between tensor products of Hilbert spaces. We denote by $\mcB(\Hsp)$ the W$^*$-algebra of all bounded operators on $\Hsp$, and by $\mcK(\Hsp)$ the C$^*$-algebra of all compact operators. 

For dual operator spaces $X\subseteq \mc{B}(\mc{H}),Y\subseteq\mc{B}(\mc{G})$, we denote by $X\bar\otimes Y$ the normal spatial tensor product and by $X\bar{\otimes}_{\mathcal{F}}Y$ the normal Fubini tensor product \cite{EffrosRuan}*{Section 7.2}. It is straightforward to check, using \cite{EffrosRuan}*{Theorem 7.2.4}, that if  $X\subseteq A$, $Y\subseteq B$ are (possibly off-diagonal) corners in von Neumann algebras $A\subseteq \mc{B}(\mc{H}),B\subseteq \mc{B}(\mc{G})$, then $X\bar\otimes Y=X\bar\otimes_{\mathcal{F}} Y\subseteq A\bar\otimes B$. Furthermore, $X\bar\otimes Y$ is a corner in $A\bar\otimes B$. Every dual operator space we use is of this form.

In particular, any Hilbert space equipped with the column or row operator space structure is a corner in a von Neumann algebra \cite{EffrosRuan}*{Section 3.4}. By default, we equip a Hilbert space $\mc{H}$ with its column operator space structure $\mc{H}=\mc{B}(\C,\mc{H})$, and its dual $\mc{H}^*$ with the row operator space structure $\mc{H}^*=\mc{B}(\mc{H},\C)$. We also identify $\mc{H}^*$ with its complex conjugate $\ov{\mc{H}}$ via $\xi^*\mapsto \ov{\xi}$. For any von Neumann algebra $A$ we have
\[
A\bar\otimes \mc{H}=A\bar\otimes \mc{B}(\C,\mc{H})\subseteq \mc{B}(L^2(A),L^2(A)\otimes \mc{H}),\quad
A\bar\otimes \ov{\mc{H}}=A\bar\otimes \mc{B}(\mc{H},\C)\subseteq \mc{B}(L^2(A)\otimes \mc{H},L^2(A))
\]
and we identify $\xi\in A\bar\otimes \mc{H}$, $\eta\in A\bar\otimes \ov{\mc{H}}$ with maps $\xi\colon L^2(A)\rightarrow L^2(A)\otimes \mc{H}$, $\eta\colon L^2(A)\otimes\mc{H}\rightarrow L^2(A)$.

If $A$ is a C$^*$-algebra, we denote by $\Mult(A)$ its multiplier C$^*$-algebra. If $A,B$ are two C$^*$-algebras, we denote by $A\otimes B$ their minimal tensor product, and by $A\otimes_{\max} B$ their maximal tensor product. If $\pi_A$ is a non-degenerate $*$-representation of $A$ on $\Hsp$, resp.\ $\pi_B$ of $B$ on $\Gsp$, we write 
\[
\pi_A\otimes \pi_B\colon A \otimes B \rightarrow\mcB(\Hsp \otimes \Gsp),
\]
while if $\pi_A$ and $\pi_B$ are non-degenerate $*$-representations of $A,B$ on the \emph{same} Hilbert space $\Hsp$ such that $\pi_A$ and $\pi_B$ elementwise commute, we put
\[
\pi_A\times \pi_B\colon A\otimes_{\max}B \rightarrow \mcB(\Hsp),\qquad a\otimes b \mapsto \pi_A(a)\pi_B(b). 
\]
When $A,B,C,D$ are C$^*$-algebras equipped with non-degenerate $*$-homomorphisms $\pi_A\colon A \rightarrow \Mult(C), \pi_B\colon B \rightarrow \Mult(D)$, we put $\pi\otimes_{\max} \pi'$ for the non-degenerate $*$-homomorphism 
\[
\pi\otimes_{\max} \pi'\colon A\otimes_{\max} B \rightarrow \Mult(C\otimes_{\max}D),\qquad a\otimes b \mapsto \pi_A(a)\otimes \pi_B(b),  
\]
the latter interpreted in the obvious way. When $\varepsilon_A\colon A \rightarrow \C$, resp. $\varepsilon_B\colon  B \rightarrow \C$ is a non-zero character, we usually abuse notation and write 
\[
\varepsilon_A \times \pi_B = \varepsilon_A \otimes_{\max} \pi_B\colon A\otimes_{\max}B \rightarrow \Mult(D),\qquad \pi_A \times \varepsilon_B = \pi_A \otimes_{\max} \varepsilon_B\colon A\otimes_{\max}B \rightarrow \Mult(C). 
\]
If $A$ is a C$^*$-algebra, we denote by $\Rep(A)$ the W$^*$-category of non-degenerate $*$-representations of $A$ on Hilbert spaces, together with bounded intertwiners as morphisms.

We refer to \cites{Tak02,Tak03,Haa75} for the basics of W$^*$-algebra (= von Neumann algebra) theory. If $A$ is a W$^*$-algebra, we denote by $A^+$ its positive cone, by $A_*$ its predual and by $A_*^+$ the positive cone of its predual. We write $L^2(A)$ for its standard Hilbert space with modular conjugation $J_A$. We write $\pi_A$ for the standard representation of $A$ on $L^2(A)$, and $\rho_A$ for the standard right representation, so 
\[
\rho_A(a) = J_A\pi_A(a)^*J_A,\qquad \forall a\in A. 
\]
We drop the notation $\pi_A$ when there is no ambiguity. If $\vp$ is an nsf weight for $A$, we denote by $\Lambda_{\vp}$ the associated GNS-map into $L^2(A)$ with domain
\[
\msN_{\vp} = \msD(\Lambda_{\vp}) = \{a\in A \mid \vp(a^*a)<\infty\}.
\]

We mainly follow the conventions of \cites{KV00,KV03} for locally compact quantum groups. In the following, we fix a locally compact quantum group $\G = (M,\Delta)$. We index notation with $M$ or $\G$ whenever this is useful. We further fix a left invariant nsf weight $\varphi$ on $M$, and then normalize the right Haar weight $\psi$ on $M$ so that $\psi = \vp \circ R$, where $R\colon M \rightarrow M$ is the \emph{unitary antipode} of $\G$. We write the standard Hilbert space of $M$ as $L^2(\G) = L^2(M)$. 

The \emph{left regular unitary $\G$-representation} is given by the unitary $\ww$, uniquely determined by 
\[
(\omega\otimes \id)(\ww^*)\Lambda_{\vp}(x) = \Lambda_{\vp}((\omega\otimes \id)\Delta(x)),\qquad \forall x\in \msN_{\vp},\, \omega\in \mcB(L^2(\G))_*.
\]
Whenever there is a risk of confusion, we will write $\ww=\ww_{\GG}$. We write $\hat{\pi}_{\red}$ for the associated non-degenerate representation of $\hat{M}_u$ on $L^2(\G)$ (the \emph{reducing map}) and put 
\[
\hat{M}_u = C^*(\G) = C_0^u(\hat{\G}),\qquad \hat{M}_{\red} = C^*_{\red}(\G) = C_0(\hat{\G}) := \hat{\pi}_{\red}(\hat{M}_u),\qquad \hat{M} = W^*(\G) = L^{\infty}(\hat{\G}) = \hat{M}_{\red}''.
\]
Then $\hat{M}$ defines the \emph{Pontryagin dual} LCQG 
\[
\hat{\G} = (\hat{M},\hat{\Delta}),\qquad \hat{\Delta}(x) = \Sigma \ww(x\otimes 1)\ww^*\Sigma,\qquad \forall x\in \hat{M}.
\]
Having fixed $\vp = \vp_{M}$, one can canonically construct a left invariant nsf weight $\hat{\vp}= \vp_{\hat{M}}$ on $\hat{M}$, together with a GNS-map $\Lambda_{\hat{\vp}}\colon \msN_{\hat{\vp}} \rightarrow L^2(\G)$, see \cite{KV03}. This construction is involutive, and we may hence without ambiguity identify $L^2(\G) \cong L^2(\hat{\G})$ in this way. We then have that 
\[
\ww_{\hat{\G}}=\hat{\ww} =\Sigma \ww^*\Sigma,
\]
and the unitary antipodes of $\G$ and $\hat{\G}$ are implemented by 
\begin{equation}\label{EqFormUnitaryAntipo}
R(x) = J_{\hat{M}}x^*J_{\hat{M}},\qquad \hat{R}(y) = J_My^*J_M,\qquad \forall x\in M,y\in \hat{M}.
\end{equation}
We also recall the formulas
\begin{equation}\label{EqModConjMultUn}
(J_{\hat{M}}\otimes J_M)\ww (J_{\hat{M}}\otimes J_M) = \ww^*,\qquad (J_M\otimes J_{\hat{M}})\hat{\ww}(J_{M}\otimes J_{\hat{M}}) = \hat{\ww}^*.
\end{equation}

The \emph{scaling constant} $\nu_{\G}$ is the unique positive number governing the skew-commutation between the unitary modular one-parametergroups associated to $\vp$ and $\hat{\vp}$, i.e.\
\[
\nabla_{\hat{\vp}}^{it}\nabla_{\vp}^{is} = \nu_{\G}^{ist} \nabla_{\vp}^{is}\nabla_{\hat{\vp}}^{it},\qquad \forall s,t\in \R. 
\]
The self-adjoint unitary
\begin{equation}\label{EqFundSymm}
u_{\G} = \nu_{\G}^{i/8}J_MJ_{\hat{M}} = \nu_{\G}^{-i/8}J_{\hat{M}}J_M
\end{equation}
is referred to as the \emph{fundamental symmetry}. 

\section{Monoidal equivalence and cocycle twists}

\subsection{W\texorpdfstring{$^*$}{*}-Morita equivalence of W\texorpdfstring{$^*$}{*}-algebras}

We briefly recall Morita equivalence theory for W$^*$-algebras. As to have our notation compatible already with structures as they will appear later on, we will use `dual' notation, so e.g.\ $\hat{M}$ will at the moment denote  just an arbitrary W$^*$-algebra.

\begin{Def}[\cite{Rie74}]
Let $\hat{P},\hat{M}$ be two W$^*$-algebras. We call $\hat{P}$ and $\hat{M}$ \emph{W$^*$-Morita equivalent} if there exists a W$^*$-algebra $\hat{Q}$ with self-adjoint projection $e\in\hat{Q}$ such that 
\[
\hat{P} \cong e\hat{Q}e,\qquad \hat{M} \cong e^{\perp}\hat{Q}e^{\perp},
\]
and such that both $e$ and $e^{\perp}$ are \emph{full}, i.e.\ $\hat{Q}e\hat{Q}$ and $\hat{Q}e^{\perp}\hat{Q}$ are $\sigma$-weakly dense in $\hat{Q}$. 

We  call $(\hat{Q},e)$ a \emph{linking W$^*$-algebra}, or more specifically a linking W$^*$-algebra between $\hat{M}$ and $\hat{P}$. 
\end{Def}

We often simply identify $\hat{P} = e\hat{Q}e$ and $\hat{M} = e^{\perp}\hat{Q}e^{\perp}$, and then identify $\hat{Q}$ with the space of $2$-by-$2$-matrices
\[
\hat{Q} = \begin{pmatrix} \hat{P} & \hat{N} \\ \hat{O} & \hat{M} \end{pmatrix},\qquad \hat{N} = e\hat{Q}e^{\perp},\quad \hat{O} = e^{\perp}\hat{Q}e. 
\]
The fullness condition is  equivalent to the condition that 
\[
[\hat{N}\hat{O}]^{\sigma\textrm{-weak}} = \hat{P},\qquad [\hat{O}\hat{N}]^{\sigma\textrm{-weak}} = \hat{M}. 
\]
It is also convenient to label $e_1 := e, e_2:=e^{\perp}$ and to put 
\[
\hat{Q} = \begin{pmatrix} \hat{Q}_{11} & \hat{Q}_{12} \\ \hat{Q}_{21} & \hat{Q}_{22}\end{pmatrix}.
\]
Then we can identify $L^2(\hat{Q}) = \bigoplus_{i,j=1}^2 L^2(\hat{Q}_{ij})$ (orthogonal direct sum), where we write
\[
L^2(\hat{Q}_{ij}) = \pi_{\hat{Q}}(e_i)\rho_{\hat{Q}}(e_j) L^2(\hat{Q}), 
\]
and the standard representation $\pi_{\hat{Q}}$ and anti-representation $\rho_{\hat{Q}}$ split as
\[\begin{split}
\hat{\pi}_{ij}^k\colon \hat{Q}_{ij} \rightarrow \mcB(L^2(\hat{Q}_{jk}),L^2(\hat{Q}_{ik})),&\quad x\mapsto \pi_{\hat{Q}}(x)_{\mid L^2(\hat{Q}_{jk})}, \\
\hat{\rho}_{ij}^k\colon \hat{Q}_{ij} \rightarrow \mcB(L^2(\hat{Q}_{ki}),L^2(\hat{Q}_{kj})),&\quad x\mapsto \rho_{\hat{Q}}(x)_{\mid L^2(\hat{Q}_{ik})}. 
\end{split}\]
We also write
\[
\hat{J}_{ij}\colon L^2(\hat{Q}_{ij})\rightarrow L^2(\hat{Q}_{ji}),\qquad \xi \mapsto J_{\hat{Q}}\xi.
\]
Then $(L^2(\hat{Q}_{ii}),\hat{\pi}_{ii}^i,\hat{J}_{ii})$ can be identified with the standard representation and modular conjugation of $\hat{Q}_{ii}$ (and with positive cone the projection of the positive cone of $L^2(\hat{Q})$).

An alternative viewpoint for W$^*$-Morita equivalence is to have a \emph{W$^*$-Morita correspondence}, that is, a Hilbert space $\Hsp$ with a faithful, normal, unital $*$-representation $\hat{\pi}$ of $\hat{P}$, resp.~anti-$*$-representation $\hat{\rho}$ of $\hat{M}$, such that 
\[
\hat{\pi}(\hat{P}) = \hat{\rho}(\hat{M})'. 
\]
For example, if $\hat{Q}$ is a linking W$^*$-algebra, then fullness of $e_1,e_2$ implies that $\Hsp = L^2(\hat{Q}_{12})$ is a W$^*$-Morita equivalence between $\hat{M}$ and $\hat{P}$ via $\hat{\pi}_{11}^2$ and $\hat{\rho}_{22}^1$. Conversely, from a W$^*$-Morita correspondence $\Hsp$ we construct a linking W$^*$-algebra 
\[
\hat{Q} :=  \begin{pmatrix} \hat{P} & \hat{N} \\ \hat{O} & \hat{M} \end{pmatrix} \subseteq \mcB\begin{pmatrix} \Hsp \\L^2(\hat{M})\end{pmatrix},
\]
where we identify $\hat{P}\cong \hat{\pi}(\hat{P})$ and
\[
\hat{N} = \{x\colon  L^2(\hat{M}) \rightarrow \Hsp\mid x\rho_{\hat{M}}(y) = \hat{\rho}(y)x\textrm{ for all }y\in \hat{M}\},\qquad \hat{O} = \{x^*\mid x\in \hat{N}\}.
\]

Moreover, we have a convenient model for $L^2(\hat{Q})$ through the unique unitary such that
\begin{equation}\label{EqStandardMorita}
 \begin{pmatrix} L^2(\hat{P}) & L^2(\hat{N}) \\ L^2(\hat{O}) & L^2(\hat{M}) \end{pmatrix}
 \cong 
\begin{pmatrix} \hat{N} \bar{\otimes}_{\hat{M}}\overline{\Hsp}  & \Hsp \\ \overline{\Hsp} & L^2(\hat{M}) \end{pmatrix},
\qquad  
 \begin{pmatrix} \pi_{\hat{Q}}(x)\hat{J}_{12}\pi_{\hat{Q}}(w)\xi & \pi_{\hat{Q}}(y)\eta \\ \hat{J}_{12} \pi_{\hat{Q}}(z) \zeta & \vartheta\end{pmatrix}
 \mapsto
 \begin{pmatrix} x\otimes \overline{w\xi} & y\eta \\ \overline{z\zeta} & \vartheta\end{pmatrix},
\end{equation}
where $x,y,w,z\in \hat{N}$ and $\xi,\eta,\zeta,\vartheta\in L^2(\hat{M})$, and where $\hat{N} \bar{\otimes}_{\hat{M}}\overline{\Hsp}$ is the separation-completion of the algebraic tensor product $\hat{N} \odot \overline{\Hsp}$ under the scalar product 
\[
\langle x\otimes \overline{\xi} , y\otimes \overline{\eta}\rangle = \langle \eta , \hat{\rho}(x^*y)\xi\rangle,\qquad \forall \xi,\eta\in \Hsp,x,y\in \hat{N}. 
\] 
In this model, the left regular representation of $\hat{Q}$ and its modular conjugation are determined by 
\begin{equation}\label{EqMoritaLeftAction}
\pi_{\hat{Q}}\begin{pmatrix} a & b \\ c & d \end{pmatrix} \begin{pmatrix} x\otimes \overline{\xi} & \eta \\ \overline{\zeta} & \vartheta\end{pmatrix} = \begin{pmatrix} ax\otimes \overline{\xi} + b\otimes \overline{\zeta} & a\eta + b\vartheta \\ \overline{\hat{\rho}(cx)^*\xi} + \overline{\hat{\rho}(d)^*\zeta} & c\eta + d\vartheta\end{pmatrix}
\end{equation}
when $a\in \hat{P},b,x\in\hat{N},c\in\hat{O},d\in \hat{M}, \xi,\eta,\zeta\in \Hsp, \vartheta\in L^2(\hat{M})$ and
\begin{equation}\label{EqMoritaModCon}
J_{\hat{Q}}\begin{pmatrix} x\otimes \overline{y\xi} & \eta \\ \overline{\zeta} & \vartheta\end{pmatrix} = \begin{pmatrix} y\otimes \overline{xJ_{\hat{M}}\xi} & \zeta \\ \overline{\eta} & J_{\hat{M}}\vartheta\end{pmatrix}
\end{equation}
when $x,y\in\hat{N}, \eta,\zeta\in \Hsp, \xi,\vartheta\in L^2(\hat{M})$.

\subsection{W\texorpdfstring{$^*$}{*}-Morita equivalence of W\texorpdfstring{$^*$}{*}-bialgebras and linking quantum groupoids}

A \emph{W$^*$-bialgebra} is a W$^*$-algebra $\hat{M}$, equipped with a coassociative, faithful, unital, normal $*$-homomorphism $\hat{\Delta}\colon \hat{M} \rightarrow \hat{M}\bar{\otimes}\hat{M}$.

The following notion is introduced in \cite{DC11a}. 

\begin{Def}\label{def2}
Let $(\hat{M},\Delta_{\hat{M}})$ and $(\hat{P},\Delta_{\hat{P}})$ be W$^*$-bialgebras. We say that $(\hat{M},\Delta_{\hat{M}})$ and $(\hat{P},\Delta_{\hat{P}})$ are \emph{comonoidally W$^*$-Morita equivalent} if there exists a linking W$^*$-algebra $(\hat{Q},e)$ with 
\[
\hat{Q} = \begin{pmatrix} \hat{P} & \hat{N} \\ \hat{O} & \hat{M}\end{pmatrix} = \begin{pmatrix} \hat{Q}_{11} & \hat{Q}_{12} \\ \hat{Q}_{21} & \hat{Q}_{22}\end{pmatrix},
\] 
together with a coassociative, faithful, (non-unital) normal $*$-homomorphism $\Delta_{\hat{Q}}\colon \hat{Q} \rightarrow \hat{Q}\bar{\otimes}\hat{Q}$ such that 
\[
\Delta_{\hat{Q}}(e) = e\otimes e,\quad \Delta_{\hat{Q}}(e^{\perp}) = e^{\perp}\otimes e^{\perp},\qquad (\Delta_{\hat{Q}})_{\mid \hat{M}} = \Delta_{\hat{M}},\quad (\Delta_{\hat{Q}})_{\mid \hat{P}} = \Delta_{\hat{P}}. 
\]
We say that $(\hat{Q},\Delta_{\hat{Q}})$ is a concrete implementation of the comonoidal W$^*$-Morita equivalence.
\end{Def}

In the above setting, $\Delta_{\hat{Q}}$ restricts to the off-diagonal corners as  
\[
\Delta_{\hat{N}} = \hat{\Delta}_{12}\colon\hat{N}\rightarrow  \hat{N}\bar{\otimes}\hat{N},\qquad \Delta_{\hat{O}} = \hat{\Delta}_{21}\colon \hat{O}\rightarrow \hat{O}\bar{\otimes}\hat{O}.
\] 
As explained in the introduction, the above tensor products can be seen as appropriate off-diagonal corners of $\hat{Q}\bar{\otimes}\hat{Q}$ or as abstract tensor products of dual operator spaces.

\begin{Theorem}[{\cite{DC11a}*{Theorem 0.7}}]
Assume $(\hat{M},\Delta_{\hat{M}})$ and $(\hat{P},\Delta_{\hat{P}})$ are comonoidally W$^*$-Morita equivalent W$^*$-bialgebras.  Then $(\hat{M},\Delta_{\hat{M}})$ is a LCQG if and only if $(\hat{P},\Delta_{\hat{P}})$ is a LCQG.  
\end{Theorem} 

In the situation of the above theorem, we write $\hat{M} = W^*(\G)$ and $\hat{P} = W^*(\Hh)$, and we then call $\G$ and $\Hh$ \emph{monoidally equivalent} (see Theorem \ref{TheoIndRep} for the reason of this nomenclature). We then write  also
\[
\hat{Q} = W^*(\msG) =  \begin{pmatrix} W^*(\Hh) & W^*(\X) \\ W^*(\Y) & W^*(\G)\end{pmatrix},\qquad \hat{\msG} = (\hat{Q},e,\Delta_{\hat{Q}}). 
\]

\begin{Def}\label{DefLinkingQuantumGroupoid}
We refer to the symbol $\msG$ as above a \emph{linking quantum groupoid} between $\G$ and $\Hh$. 
\end{Def}

See Section \ref{SecDualLQG} for more information on the precise structure be attached to $\msG$ so that the notation is consistent with the one for LCQGs.

\begin{Rem}\label{rem1}
Let $\msG$ be a linking quantum groupoid corresponding to $\hat{\msG}=(\hat{Q},e,\Delta_{\hat{Q}})$. Then the triple $(\hat{Q},e^{\perp},\Delta_{\hat{Q}})$ (with the role of $e$ and $e^{\perp}$ swapped) also defines a linking quantum groupoid. It will be called the \emph{reflection} of $\msG$. Reflecting is a useful technical tool when we want to exchange the roles of $\G$ and $\Hh$.
\end{Rem}

One can interpret $\msG$ as a quantum groupoid with two (classical) objects having $\G$ and $\Hh$ as their isotropy. As discussed at the end of \cite{DC11a}*{Section 5}, one can view $(W^*(\msG),\Delta_{\hat{Q}})$ more concretely as an instance of a \emph{measured quantum groupoid} in the sense of \cite{Les07}. However, as they are quantum groupoids of a simple type, they can be handled in a manner similar to quantum groups themselves, with the necessary generalisation only concerning the manipulation of the algebraic structure, and less so the analytic structure.  

One such change concerns the multiplicative unitaries, which will in this setting be multiplicative \emph{partial isometries}: endowing $\hat{Q}$ with the weight\footnote{There is in general no canonical normalisation of $\vp_{\hat{P}}$ in terms of $\vp_{\hat{M}}$, so one simply makes a choice here.} $\vp_{\hat{Q}} = \vp_{\hat{P}}\oplus \vp_{\hat{M}}$, we construct the partial isometry
\begin{equation}\label{EqMultIsoLinkQGr}
\ww_{\hat{Q}}^*\colon L^2(\hat{Q}) \otimes L^2(\hat{Q}) \rightarrow L^2(\hat{Q}) \otimes L^2(\hat{Q}),\quad \Lambda_{\hat{Q}}(x) \otimes \Lambda_{\hat{Q}}(y) \mapsto (\Lambda_{\hat{Q}}\otimes \Lambda_{\hat{Q}})(\Delta_{\hat{Q}}(y)(x\otimes 1)),\quad\; \forall x,y \in \msN_{\vp_{\hat{Q}}}.  
\end{equation}
We can split $L^2(\hat{Q}) = L^2(\msG) = L^2(\Hh)\oplus L^2(\Y) \oplus L^2(\X) \oplus L^2(\G)$, where 
\[
L^2(\Hh) = L^2(\hat{P}),\quad L^2(\Y) = L^2(\hat{O}),\quad L^2(\X) = L^2(\hat{N}),\quad L^2(\G) = L^2(\hat{M}).
\]
If we abbreviate
\[
\G_{11} = \Hh,\quad \G_{12} = \X,\quad \G_{21} = \Y,\quad \G_{22}= \G,
\]
then we can state that there exist elements 
\[
\hat{\ww}_{ij} \in W^*(\G_{ji}) \bar{\otimes} \mcB(L^2(\G_{ij}))
\]
such that $\ww_{\hat{Q}}$ splits into unitaries
\begin{equation}\label{EqDecompDualW}
\hat{\ww}_{ij}^k = (\hat{\pi}_{ji}^k\otimes \id)\hat{W}_{ij}\colon L^2(\G_{ik})\otimes L^2(\G_{ij}) \rightarrow L^2(\G_{jk})\otimes L^2(\G_{ij}),\qquad \xi \mapsto \ww_{\hat{Q}}\xi. 
\end{equation}
See \cite{DC09}*{Chapter 11} for a detailed discussion, where also the contents of the next section are elaborated on. 

\subsection{Dual of a linking quantum groupoid}\label{SecDualLQG}

\begin{Def}\label{DefLQGr}
Let $\msG$ be a linking quantum groupoid. Its associated \emph{function W$^*$-algebra} is
\[
Q= L^{\infty}(\msG) := [(\omega\otimes \id)\ww_{\hat{Q}}\mid \omega \in \hat{Q}_*]^{\sigma\textrm{-weak}} \subseteq \mcB(L^2(\msG)). 
\]
\end{Def}
One shows that $Q$ is indeed a W$^*$-algebra. The decomposition of \eqref{EqDecompDualW} then shows that $L^{\infty}(\msG)$ is naturally a direct sum W$^*$-algebra 
\[
L^{\infty}(\msG) = L^{\infty}(\Hh)\oplus L^{\infty}(\Y)\oplus L^{\infty}(\X)\oplus L^{\infty}(\G) \subseteq \mcB(L^2(\Hh)\oplus L^2(\Y) \oplus L^2(\X) \oplus L^2(\G)).
\]
If we put 
\[
L^{\infty}(\G_{11}) = L^{\infty}(\Hh),\quad L^{\infty}(\G_{21}) = L^{\infty}(\Y),\quad L^{\infty}(\G_{12}) = L^{\infty}(\X),\quad L^{\infty}(\G_{22}) = L^{\infty}(\G),
\]
then $1_{ij} = \pi_{\hat{Q}}(e_i)\rho_{\hat{Q}}(e_j)$ is the unit of $L^{\infty}(\G_{ij})$.

These components of $L^{\infty}(\msG)$ are endowed with faithful, normal, unital $*$-homomorphisms 
\[
\Delta_{ij}^k\colon L^{\infty}(\G_{ij}) \rightarrow L^{\infty}(\G_{ik})\bar{\otimes} L^{\infty}(\G_{kj})
\]
such that
\[
\Sigma \ww_{\hat{Q}}(x\otimes 1)\ww_{\hat{Q}}^*\Sigma = \Delta_{ij}^1(x) + \Delta_{ij}^2(x),\qquad \forall x\in L^{\infty}(\G_{ij}).  
\]
These maps $\Delta_{ij}^k$ satisfy the joint coassociativity condition
\[
(\Delta_{i,k}^{\ell}\otimes \id)\Delta_{i,j}^k=
(\id \otimes \Delta^k_{\ell,j})\Delta_{i,j}^{\ell},\qquad \forall i,j,k,\ell \in \{1,2\}.
\]
If we endow $Q = L^{\infty}(\msG)$ with the coassociative, faithful, (non-unital) normal $*$-homomorphism
\[
\Delta_{Q}(x) = \Delta_{ij}^1(x) + \Delta_{ij}^2(x),\qquad \forall x\in L^{\infty}(\G_{ij}),
\]
then 
\begin{equation}\label{EqDefLQGr}
\msG = (L^{\infty}(\msG),\Delta_Q) = (L^{\infty}(\msG),\Delta_Q,\{1_{ij}\})
\end{equation} 
is again a measured quantum groupoid (we add the units $\{1_{ij}\}$ to the notation if we want to emphasize it as a \emph{linking} quantum groupoid). In fact, $(L^{\infty}(\msG),\Delta_Q)$ is the dual of $(\hat{Q},\Delta_{\hat{Q}})$ in the sense of measured quantum groupoids. We also have on each of the $L^{\infty}(\G_{ij})$ two nsf weights $\vp_{ij}$ and $\psi_{ij}$, satisfying 
\[\begin{split}
\vp_{kj}((\omega_{ik}\otimes \id)\Delta_{ij}^k(x)) &= \vp_{ij}(x),\qquad \forall x\in L^{\infty}(\G_{ij})^+,\, \textrm{normal states }\omega_{ik}\textrm{ on }L^{\infty}(\G_{ik}),\\
 \psi_{ik}((\id\otimes \omega_{kj})\Delta_{ij}^k(x)) &= \psi_{ij}(x),\qquad \forall x\in L^{\infty}(\G_{ij})^+,\, \textrm{normal states }\omega_{kj}\textrm{ on }L^{\infty}(\G_{kj}).
\end{split}\]
These weights are canonically determined once $\vp_{\hat{M}}$ and $\vp_{\hat{P}}$ are fixed, and one can then identify $L^2(\hat{Q})\cong L^2(Q)$ via a concrete GNS-map for the weight $\vp_Q = \bigoplus_{i,j=1}^{2}\vp_{ij}$, in a manner completely similar to the construction in \cite{KV03}. In particular, $\vp_{ii}$ and $\psi_{ii}$ can be made to coincide with the left and right invariant nsf weights on $(Q_{ii},\Delta_{ii}^i)$, and the identification $L^2(\hat{Q}_{ii})\cong L^2(Q_{ii})$ is the usual one from the theory of locally compact quantum groups.  

Clearly $\hat{\ww}^{k}_{ij}\in \mcB(L^2(\G_{ik}),L^2(\G_{jk}))\bar\otimes L^{\infty}(\G_{ij})$.  Using \cite{DC09}*{Lemma 11.1.5.(3)} one can check that
\begin{equation}\label{eq17}
(\id\otimes \Delta^{\ell}_{ij})(\hat{\ww}^{k}_{ij})=
\hat{\ww}^{k}_{\ell j [13]} \hat{\ww}^{k}_{i\ell [12]}
\end{equation}
for all indices $i,j,k,\ell \in \{1,2\}$.

As for locally compact quantum groups, the modular conjugation of $Q$  allows to construct a \emph{unitary antipode} for the quantum groupoid W$^*$-algebra $W^*(\msG)$, via 
\begin{equation}\label{EqUnitaryAntipodeQgr}
\hat{R}\colon W^*(\msG)\rightarrow W^*(\msG),\qquad \hat{R}(x) = J_Qx^*J_Q,\qquad \forall x\in W^*(\msG). 
\end{equation}
With $J_{ij}$ denoting the restriction of $J_Q$ to $L^2(\G_{ij})$, we then have that $\hat{R}$ restricts to anti-multiplicative $*$-compatible maps
\begin{equation}\label{EqUnitAntipodQGr}
\hat{R}_{ij}\colon W^*(\G_{ij})\rightarrow W^*(\G_{ji}),\qquad \hat{\pi}_{ji}^k(\hat{R}_{ij}(x)) = J_{jk}\hat{\pi}_{ij}^k(x)^*J_{ik},
\end{equation}
meaning that for all $i,j,k\in\{1,2\}$ we have
\begin{equation}\label{EqUnitAntipodQGrProp}
\hat{R}_{ik}(xy) = \hat{R}_{jk}(y)\hat{R}_{ij}(x),\quad \hat{R}_{ij}(x)^* = \hat{R}_{ji}(x^*),\qquad\forall x\in  W^*(\G_{ij}),y\in W^*(\G_{jk}). 
\end{equation}
We also have the expected anti-comultiplicativity property
\[
\hat{\Delta}_{ji}\circ \hat{R}_{ij}=  (\hat{R}_{ij}\otimes \hat{R}_{ij})\hat{\Delta}_{ij}^{\opp}. 
\]

Finally, we also have C$^*$-algebras associated to $\msG$, determined via 
\[
C_0(\msG) =  [(\omega\otimes \id)\ww_{\hat{Q}}\mid \omega \in \hat{Q}_*] \subseteq L^{\infty}(\G),
\]
which breaks up into components 
\[
C_0(\G_{ij}) = C_0(\Hh)\oplus C_0(\Y)\oplus C_0(\X)\oplus C_0(\G) \subseteq \mcB(L^2(\Hh)\oplus L^2(\Y) \oplus L^2(\X) \oplus L^2(\G)).
\]

\subsection{Bi-Galois objects}\label{SecBiGal}

Let $\msG = \begin{pmatrix} \Hh & \X \\ \Y & \G\end{pmatrix}$ be a linking quantum groupoid. Then the particular component 
\begin{equation}\label{EqFormAlph}
(L^{\infty}(\X),\alpha_{\X}),\qquad \alpha_{\X}= \Delta_{12}^2\colon L^{\infty}(\X)\rightarrow L^{\infty}(\X)\bar{\otimes} L^{\infty}(\G)
\end{equation}
defines a (right) action of $\G$ on $L^{\infty}(\X)$. It is called the \emph{Galois object} for $\G$ associated to the above quantum groupoid. From $\G$ and $\X$ one can reconstruct the whole quantum groupoid (up to isomorphism) \cite{DC11b}, but we will not need this result here. We only observe that $(N,\alpha_{\X}) = (L^{\infty}(\X),\alpha_{\X})$ is an \emph{ergodic}  right $\G$-action:  with
\[
N^{\alpha_{\X}} = \{x \in N \mid \alpha_{\X}(x) = x\otimes1\},
\]
we have $N^{\alpha_{\X}} = \C$. Moreover, $\alpha_{\X}$ is \emph{integrable}, in that we have semi-finiteness of the normal, faithful weight
\[
\vp_{\X} = \vp_{12} = (\id\otimes \vp_M)\alpha_{\X}\colon N^+ \rightarrow (N^{\alpha_{\X}})_{\ext}^+ = [0,\infty].
\]

If we also consider on $L^{\infty}(\X)$ its associated left $\Hh$-action 
\begin{equation}\label{EqFormGamm}
\gamma_{\X} = \Delta_{12}^1\colon L^{\infty}(\X) \rightarrow L^{\infty}(\Hh)\bar{\otimes} L^{\infty}(\X),
\end{equation}
then $\alpha_{\X}$ and $\gamma_{\X}$ commute:
\[
(\gamma_{\X}\otimes \id)\alpha_{\X} = (\id\otimes \alpha_{\X})\gamma_{\X}.
\] 
We can similarly define on $L^{\infty}(\Y)$ the commuting left and right actions 
\[
\gamma_{\Y} = \Delta_{21}^2\colon L^{\infty}(\Y)\rightarrow L^{\infty}(\G)\bar{\otimes}L^{\infty}(\Y),\quad \alpha_{\Y} = \Delta_{21}^1\colon L^{\infty}(\Y) \rightarrow L^{\infty}(\Y)\bar{\otimes} L^{\infty}(\Hh). 
\]
\begin{Def}
The triple $(L^{\infty}(\X),\alpha_{\X},\gamma_{\X})$ is referred to as the \emph{biGalois object} implementing the monoidal equivalence between $\G$ and $\Hh$. We call $(L^{\infty}(\Y),\alpha_{\Y},\gamma_{\Y})$ the \emph{opposite} biGalois object, implementing the monoidal equivalence between $\Hh$ and $\G$.
\end{Def}

\subsection{\texorpdfstring{$2$}{2}-cocycles and skew bicharacters}\label{SubSecCocBich}

As a particular case, consider a W$^*$-bialgebra $(\hat{M},\Delta_{\hat{M}})$ and let $\hat{\Omega} \in \hat{M} \bar{\otimes} \hat{M}$ be a unitary $2$-cocycle: 
\[
\hat{\Omega}_{[12]}(\Delta_{\hat{M}}\otimes \id)\hat{\Omega}= \hat{\Omega}_{[23]}(\id\otimes \Delta_{\hat{M}})\hat{\Omega}. 
\]
Then we can make on $\hat{P}:= \hat{M}$ a new coassociative coproduct 
\[
\Delta_{\hat{P}}(x) := \hat{\Omega} \Delta_{\hat{M}}(x)\hat{\Omega}^*,\qquad \forall x\in \hat{M}.
\]
The W$^*$-bialgebras $(\hat{M},\Delta_{\hat{M}})$ and $(\hat{P},\Delta_{\hat{P}})$ are comonoidally W$^*$-Morita equivalent via  the trivial linking W$^*$-algebra
$\hat{Q} = \begin{pmatrix} \hat{M} & \hat{M} \\ \hat{M} & \hat{M} \end{pmatrix}$ with coproducts given by 
\begin{equation}\label{EqCoprodCocTwist}
\hat{\Delta}_{11}=\Delta_{\hat{P}},\quad \hat{\Delta}_{22}=\Delta_{\hat{M}},\quad
\hat{\Delta}_{12} = \hat{\Omega}\Delta_{\hat{M}}(\cdot),\quad \hat{\Delta}_{21} = \Delta_{\hat{M}}(\cdot)\hat{\Omega}^*.
\end{equation}
If $\G= (M,\Delta_M)$ is a LCQG, we refer to $\hat{\Omega}$ as a \emph{dual unitary $2$-cocycle} for $\G$, and we write
\begin{equation}
\G_{\hat{\Omega}} = (P,\Delta_P).
\end{equation}

Specializing further, consider W$^*$-bialgebras $(\hat{M}_1,\hat{\Delta}_1),(\hat{M}_2,\hat{\Delta}_2)$ together with a \emph{unitary skew bicharacter}
\begin{equation}\label{EqSkewBichvN}
\wh{\mcX} \in \hat{M}_1\bar{\otimes}\hat{M}_2,\qquad  (\hat{\Delta}_1\otimes \id)\wh{\mcX} = \wh{\mcX}_{[13]}\wh{\mcX}_{[23]},\quad (\id\otimes \hat{\Delta}_2)\wh{\mcX} = \wh{\mcX}_{[13]}\wh{\mcX}_{[12]}.
\end{equation}
Then the tensor product W$^*$-bialgebra $(\hat{M},\Delta_{\hat{M}}) := (\hat{M}_1\bar{\otimes}\hat{M}_2,\hat{\Delta}_{\otimes})$, where 
\[
\hat{\Delta}_{\otimes}(x\otimes y)= \hat{\Delta}_1(x)_{[13]}\hat{\Delta}_2(y)_{[24]},\qquad \forall x\in \hat{M}_1,y\in \hat{M}_2,
\] 
obtains the unitary $2$-cocycle
\begin{equation}\label{EqCocycBichar}
\hat{\Omega} = \wh{\mcX}_{[32]} \in (\hat{M}_1\bar{\otimes}\hat{M}_2)\bar{\otimes}(\hat{M}_1\bar{\otimes}\hat{M}_2). 
\end{equation}
If $\hat{\G}_i = (\hat{M}_i,\hat{\Delta}_i)$ for LCQG $\hat{\G}_i$, one calls 
\begin{equation}
D_{\wh{\mcX}}(\hat{\G}_1,\hat{\G}_2):=(\hat{M},\hat{\Omega}\Delta_{\hat{M}}(-)\hat{\Omega}^*)
\end{equation}
a \emph{generalized Drinfeld double} \cite{BV05}*{Section 8}.  When $\hat{\G}_1 = \G$ and $\hat{\G}_2= \hat{\G}$ with $\wh{\mcX} = \ww$,
\begin{equation}
D(\G) = D_{\ww}(\hat{\G}_1,\hat{\G}_2)
\end{equation}
is the \emph{Drinfeld double} of $\G$. 

\section{Representations and actions of linking quantum groupoids}

Many of the results on locally compact quantum groups extend to linking quantum groupoids. In this section, we give some information on how such extensions look like.

\subsection{Representations of linking quantum groupoids}

We refer to \cite{DC09}*{Section 11 and Section 7.6} for details on the next results.

\begin{Def}\label{def1}
Let $\msG = (Q,\Delta_Q,\{1_{ij}\})$ be a linking quantum groupoid. By a \emph{unitary $\msG$-representation}, we mean a pair of Hilbert spaces $\Hsp_1,\Hsp_2$ together with four unitaries $U_{ij} \in L^{\infty}(\G_{ij})\bar\otimes \mcB(\Hsp_j,\Hsp_i)$
such that 
\[ 
(\Delta_{ij}^k \otimes \id)U_{ij} = U_{ik[13]}U_{kj[23]},\qquad \forall i,j,k\in \{1,2\}.
\]
\end{Def} 
Putting $\Hsp_1,\Hsp_2$ together into their direct sum $\Hsp = \begin{pmatrix}\Hsp_1\\ \Hsp_2\end{pmatrix}$, and denoting by $p_1,p_2$ the projections on the subspaces $\Hsp_1,\Hsp_2$, the above condition is equivalent with having a partial isometry $U \in L^{\infty}(\msG)\bar{\otimes} \mcB(\Hsp)$ with
\[
(\Delta_Q\otimes \id)U = U_{[13]}U_{[23]},\qquad UU^* = \sum_{i,j=1}^{2} 1_{ij}\otimes p_i,\quad U^*U = \sum_{i,j=1}^{2} 1_{ij}\otimes p_j. 
\]
For the convenience of language, we still refer to $U$ as a `unitary $\msG$-representation'. Modifying the usual argument for LCQG (see e.g.\ the proof of \cite{Wor96}*{Theorem 1.6.2}, making use of the Baaj-Skandalis argument of \cite{BS93}*{Proposition 3.6}),
 one shows that such unitary $\msG$-representations $U$ are automatically continuous, i.e.\
\[
U \in \Mult(C_0(\msG) \otimes \mcK(\Hsp)).
\]

We denote by $\Rep(\msG)$ the W$^*$-category of unitary $\msG$-representations. A morphism $T\in \oon{Mor}(U,U')$ in $\Rep(\msG)$ is a bounded linear map satisfying $(1\otimes T)U=U'(1\otimes T)$. Note that $T$ automatically respects the decomposition of $\Hsp,\Hsp'$ into direct sums. The W$^*$-category $\Rep(\msG)$ is monoidal through 
\begin{equation}\label{eq5}
\left(\begin{pmatrix} \Hsp_1 \\ \Hsp_2\end{pmatrix},U\right)\otimes_{\msG}  \left(\begin{pmatrix} \Hsp_1' \\ \Hsp_2'\end{pmatrix},U'\right) = \left(\begin{pmatrix} \Hsp_1 \otimes \Hsp_1' \\ \Hsp_2\otimes \Hsp_2'\end{pmatrix},U \TensRep U'\right), \qquad U \TensRep U' := U_{[12]}U_{[13]}'. 
\end{equation}

Now $Q_*$ still has a convolution product $\omega *\chi = (\omega \otimes \chi)\Delta_Q$, with any unitary $\msG$-representation $U$ leading to a contractive representation 
\[
\hat{\pi}_U\colon Q_*\rightarrow \mcB(\Hsp),\quad \omega \mapsto (\omega\otimes \id)(U). 
\]
We then have the following extension of the results in \cite{Kus01}*{Section 5}: 
\begin{Theorem}
The completion of $Q_*$ with respect to
\[
\|\omega\|_u := \sup\{\|(\omega\otimes \id)(U)\|\mid U \in \Rep(\msG)\}
\]
is a C$^*$-algebra $\hat{Q}_u = C^*(\msG)$. Moreover, any representation of the form $\hat{\pi}_U$ extends uniquely to a non-degenerate $*$-representation of $C^*(\msG)$,  leading to an isomorphism of W$^*$-categories 
\begin{equation}\label{EqIsoCat}
\Rep(\msG) \leftrightarrow \Rep(\hat{Q}_u),\qquad U \mapsto  \hat{\pi}_U.
\end{equation}
\end{Theorem} 
The C$^*$-algebra $\hat{Q}_u = C^*(\msG)$ splits again as 
\begin{equation}\label{EqStrongMorita}
\hat{Q}_u = C^*(\msG) = \begin{pmatrix} \hat{P}_u & \hat{N}_u \\
\hat{O}_u & \hat{M}_u\end{pmatrix} = \begin{pmatrix} C^*(\Hh) & C^*(\X) \\ C^*(\Y) & C^*(\G)\end{pmatrix} =  \begin{pmatrix} C^*(\G_{11}) & C^*(\G_{12}) \\ C^*(\G_{21}) & C^*(\G_{22})\end{pmatrix},
\end{equation}
where $C^*(\G)$ and $C^*(\Hh)$ are indeed the universal group C$^*$-algebras for $\G$ and $\Hh$, and with \eqref{EqStrongMorita} establishing a  C$^*$-Morita equivalence between $C^*(\G)$ and $C^*(\Hh)$ \cite{DC09}*{Theorem 7.6.4}. 

Note that if $C^*(\msG)$ is $*$-represented on a Hilbert space $\mc{H}$, then the images of projections $e_i$ establish a decomposition of $\mc{H}$ into a direct sum $\mc{H}=\mc{H}_1\oplus \mc{H}_2$. Furthermore, such a decomposition is preserved by morphisms in $\Rep(C^*(\msG))$.

There is now moreover a unique (degenerate) $*$-homomorphism 
\[
\Delta_{\hat{Q}_u}\colon \hat{Q}_u \rightarrow \Mult(\hat{Q}_u \otimes \hat{Q}_u),\qquad \Delta_{\hat{Q}_u} = \begin{pmatrix} \hat{\Delta}_{u,11} & \hat{\Delta}_{u,12}\\ \hat{\Delta}_{u,21} & \hat{\Delta}_{u,22}\end{pmatrix}
\] 
such that for all $U,U' \in \Rep(\msG)$ we have
\begin{equation}\label{EqTensProd2}
(\omega \otimes \id)(U_{[12]}U_{[13]}') =  (\hat{\pi}_U \otimes \hat{\pi}_{U'})\Delta_{\hat{Q}_u}^{\opp}(\omega),\qquad \forall \omega \in Q_* \subseteq  C^*(\msG).
\end{equation}
Considering the strictly continuous extension of $\Delta_{\hat{Q}_u}$  to $\Mult(C^*(\msG))$, and viewing the diagonal units $1_i \in \Mult(C^*(\G_{ii}))$ as  elements $e_i \in \Mult(C^*(\msG))$, one shows that $\Delta_{\hat{Q}_u}(e_i) = e_i \otimes e_i$.

For non-degenerate $*$-representations $\hat{\pi},\hat{\pi}'\in \Rep(C^*(\msG))$ we then write  
\[
\hat{\pi} * \hat{\pi}'= \textrm{(co-restriction of) }(\hat{\pi} \otimes \hat{\pi}')\Delta_{\hat{Q}_u}^{\opp},
\] 
which is a non-degenerate $*$-representation of $C^*(\msG)$ on $\mc{H}\otimes_{\msG}\mc{H}'$ (see \eqref{eq5}). Then \eqref{EqTensProd2} can equivalently be written as 
\[
\hat{\pi}_{U \TensRep U'} = \hat{\pi}_U*\hat{\pi}_{U'},\qquad 
\forall U,U' \in \Rep(\msG).
\]
In this way, we may naturally view $\Rep(C^*(\msG))$ as a monoidal W$^*$-category (using the opposite coproduct), such that \eqref{EqIsoCat} becomes a monoidal isomorphism.  

Also the other results of \cite{Kus01} have their analogues in this setting. First of all, there is a unique lift of the unitary antipode $R_{\hat{Q}}$ to a $*$-preserving, anti-comultiplicative, anti-multiplicative involution 
\begin{equation}\label{EqUnitaryAntipodeUnivQgr}
R_{\hat{Q}_u}\colon C^*(\msG) \rightarrow C^*(\msG),
\end{equation}
splitting up into maps 
\[
\hat{R}_{u,ij}\colon C^*(\G_{ij})\rightarrow C^*(\G_{ji}). 
\]
Secondly, by applying \eqref{EqIsoCat} with respect to a universal $*$-representation of $C^*(\G)$, we find the existence of a `right half-lifted universal' unitary $\msG$-representation 
\begin{equation}\label{EqHalfUniversalLift}
\Ww_{\hat{Q}_u[21]}^* \in \Mult(C_0(\msG)\otimes C^*(\msG))
\end{equation}
such that, for any unitary $\msG$-representation $U$, we have 
\begin{equation}\label{EqUnivPropUnivLift}
(\id\otimes \hat{\pi}_{U})(\Ww_{\hat{Q}_u[21]}^*) = U.
\end{equation}
Writing 
\[
\pi^{\red}_{\hat{Q}_u} = \hat{\pi}_{\ww_{\hat{Q}[21]}^*}\colon C^*(\msG) \rightarrow W^*(\msG) \subseteq \mcB(L^2(\msG))
\]
for the reducing map, the operator $\Ww_{\hat{Q}_u}$ satisfies 
\begin{equation}\label{EqHalfUniversalRepQgr}
(\id\otimes \pi_{\hat{Q}_u}^{\red})\Delta_{\hat{Q}_u}(x) =
\Ww_{\hat{Q}_u}^*(1\otimes \pi^{\red}_{\hat{Q}_u}(x))\Ww_{\hat{Q}_u},\qquad \forall x\in C^*(\msG). 
\end{equation}

As before, $\Ww_{\hat{Q}_u}$ breaks up into pieces
\begin{equation}\label{EqPiecesUniversal}
\hat{\Ww}_{ij} \in \Mult(C^*(\G_{ji})\otimes C_0(\G_{ij})),
\end{equation}
where the two off-diagonal corners may simply be interpreted as corners of the multiplier C$^*$-algebra of $C^*(\msG)\otimes C_0(\msG)$.

Finally, we also have the trivial unitary $\msG$-representation given by $\Hsp_1 = \Hsp_2 = \C$ and $U_{ij} = 1_{ij}\in L^{\infty}(\G_{ij})$ for $i,j\in \{1,2\}$, which acts as the unit for $\otimes_{\msG}$. The corresponding $*$-representation of $C^*(\msG)$ will give us a non-degenerate $*$-homomorphism 
\begin{equation}\label{EqUnivCounit}
\varepsilon_{\hat{Q}_u}\colon \hat{Q}_u \rightarrow M_2(\C),\qquad \begin{pmatrix} a & b \\ c & d\end{pmatrix} \mapsto \begin{pmatrix} \varepsilon_{\hat{\Hh}}(a) & \varepsilon_{\hat{\X}}(b) \\ \varepsilon_{\hat{\Y}}(c) & \varepsilon_{\hat{\G}}(d)\end{pmatrix} =  \begin{pmatrix} \hat{\varepsilon}_{11}(a) & \hat{\varepsilon}_{12}(b) \\ \hat{\varepsilon}_{21}(c) & \hat{\varepsilon}_{22}(d)\end{pmatrix},
\end{equation}
which we refer to as the \emph{counit} for $\hat{\msG}$.

\subsection{Actions of linking quantum groupoids}

Let us now turn to actions of $\msG$ on W$^*$-algebras. We refer to \cite{DC09}*{Section 7 and Section 8} for details on the next results (see \cite{Cre18} for similar results in the C$^*$-algebra setting).

\begin{Def}
Let $\msG = (Q,\Delta_Q,\{1_{ij}\})$ be a linking quantum groupoid. An \emph{action} of $\msG$ on a W$^*$-algebra $A$ consists of a decomposition $A = A_1\oplus A_2$ as a direct sum W$^*$-algebra, and a (non-unital) faithful, normal $*$-homomorphism $\gamma\colon A \rightarrow Q\bar{\otimes}A$ such that, writing $1_i$ for the unit of $A_i$, we have 
\[
(\Delta_Q\otimes \id)\gamma = (\id\otimes \gamma)\gamma,\qquad \gamma(1_i) = \sum_{j=1}^{2} 1_{ij} \otimes 1_j. 
\]
\end{Def} 
Again, this definition can be given piecewise by asking that we have faithful, unital, normal $*$-homomorphisms
\[
\gamma_i^j\colon A_i \rightarrow L^{\infty}(\G_{ij})\bar{\otimes} A_j, 
\]
satisfying the piecewise coaction property
\[
(\Delta_{ij}^k\otimes \id)\gamma_i^j = (\id\otimes \gamma_k^j)\gamma_i^k,\qquad \forall i,j,k\in \{1,2\}. 
\]

There is the following analogue of the crossed product construction. 
\begin{Def}
Let $\msG$ be a linking quantum groupoid acting on $A = A_1 \oplus A_2$, and put
\[
L^2(\msG) \underset{\C^2}{\otimes} L^2(A)  := \bigoplus_{i,j=1}^{2} L^2(\G_{ij}) \otimes L^2(A_j)\subseteq L^2(\msG)\otimes L^2(A).
\]
Then the \emph{crossed product} W$^*$-algebra $\msG\ltimes A$ is defined as 
\[
[(\pi_{\hat{Q}}(x)\otimes 1)\gamma(a)\mid  x\in \hat{Q},a\in A]^{\sigma\textrm{-weak}} \subseteq \mcB(L^2(\msG) \underset{\C^2}{\otimes} L^2(A)). 
\]
\end{Def}
One shows that $\msG\ltimes A$ is indeed a W$^*$-algebra. Note that $\hat{Q}\otimes 1 \nsubseteq \msG\ltimes A$, although we \emph{do} have a normal, faithful copy of $\hat{Q}$ inside the crossed product via
\[
x \mapsto (\pi_{\hat{Q}}(x)\otimes 1)\gamma(1).
\]
In fact, writing $\hat{p}_i := (\pi_{\hat{Q}}(e_i) \otimes 1)\gamma(1)$, we can split 
\begin{equation}\label{EqSplitMatrixCross}
\msG \ltimes A = \begin{pmatrix} \G_{11}\ltimes A_1 & \G_{12} \ltimes A_2 \\ 
\G_{21}\ltimes A_1 & \G_{22}\ltimes A_2\end{pmatrix},  \qquad \G_{ij}\ltimes A_j  = \hat{p}_i (\msG\ltimes A)\hat{p}_j. 
\end{equation}
The defining $*$-representation of $\msG\ltimes A$ then splits into components 
\[
\pi^{\ltimes,k}_{ij}\colon \G_{ij}\ltimes A_j \rightarrow \mcB(L^2(\G_{jk}),L^2(\G_{ik}))\bar{\otimes} A_k \subseteq \mcB(L^2(\G_{jk})\otimes L^2(A_k),L^2(\G_{ik})\otimes L^2(A_k)),
\]
with both $\pi^{\ltimes,k} = \oplus_{i,j} \pi^{\ltimes,k}_{ij}$ faithful, normal $*$-representations of $\msG \ltimes A$. The faithfulness follows from the (proof of the) next lemma, since a normal $*$-representation of a linking W$^*$-algebra that is faithful on one of the diagonal corners, is faithful on the whole algebra.

\begin{Lem}\label{LemMoritaEquivCrossProd}
The W$^*$-algebras $\G_{ii}\ltimes A_i$'s are canonically isomorphic with the usual crossed products, and the matrix decomposition \eqref{EqSplitMatrixCross} establishes a W$^*$-Morita equivalence between the $\G_{11}\ltimes A_1$ and $\G_{22}\ltimes A_2$ through the linking W$^*$-algebra $\msG\ltimes A$.
\end{Lem} 
\begin{proof}
The fact that $\msG\ltimes A$ is a W$^*$-linking algebra follows immediately from the fact that the $e_i$'s are full projections for $\hat{Q}$. Furthermore, $\pi_{ii}^{\ltimes,i}(\G_{ii}\ltimes A_i)$ coincides on the nose with the usual LCQG crossed product. 
\end{proof}

The W$^*$-algebra $\msG \ltimes A$ comes with a dual right action of 
$\hat{\msG} =  (W^*(\msG),\Delta_{\hat{Q}})$ via 
\[
\hat{\gamma}\colon \msG\ltimes A \rightarrow (\msG\ltimes A)\bar{\otimes}\hat{Q},\qquad 
\gamma(a) \mapsto \gamma(a)\otimes 1,\quad (\pi_{\hat{Q}}(x)\otimes 1)\gamma(1)\mapsto ((\pi_{\hat{Q}}\otimes \pi_{\hat{Q}})(\Delta_{\hat{Q}}(x))_{[13]}(\gamma(1)\otimes 1).
\]
For an action of this particular measured quantum groupoid, the lack of unitality for the associated (right) action is implemented by the requirement that the projections $\hat{p}_i\in \msG\ltimes A$ satisfy 
\[
\hat{\gamma}(\hat{p}_i) =  \hat{p}_i\otimes e_i. 
\]

\subsection{Canonical unitary implementation of an action of a linking quantum groupoid}\label{sec:groupoidimplementation}

The crossed product construction also gives us access to the construction of a canonical unitary implementation of an action of a linking quantum groupoid. This is what is accomplished in the full generality of measured quantum groupoids in \cite{Eno10}, but let us again spell out the details in our simpler case of linking quantum groupoids, where one can directly   mimick the constructions of \cite{Vae01}.

Let $(A,\gamma)$ be a (left) $\msG$-W$^*$-algebra. One argues first that there is an nsf operator valued weight 
\[
(\id\otimes \id\otimes \vp_{\hat{Q}})\circ \hat{\gamma}\colon \msG \ltimes A \rightarrow \gamma(A)_{\ext}^+.
\]
If then $\theta = \theta_1\oplus \theta_2$ is an nsf weight on $A$, we get an induced nsf weight on $\msG \ltimes A$ by 
\[
\widetilde{\theta}\colon z \mapsto \theta ( \gamma^{-1} ((\id\otimes \id\otimes \vp_{\hat{Q}})\hat{\gamma}(z))).
\]
We can then endow $L^2(\msG) \underset{\C^2}{\otimes} L^2(A)$ with a GNS map for $\widetilde{\theta}$ by closing the linear extension of the map
\[
(\pi_{\hat{Q}}(x)\otimes 1)\gamma(a) \mapsto \Lambda_{\hat{Q}}(x)\otimes \Lambda_{\theta}(a),\qquad \forall x\in \msN_{\vp_{\hat{Q}}}e_j,1_ja\in \msN_{\theta},
\]
effectively identifying $L^2(\msG) \underset{\C^2}{\otimes} L^2(A)$ with $L^2(\msG \ltimes A)$ (note that the left hand side elements are zero if  $x\in \hat{Q}_{ij}$ and $a \in A_k$ with $k\neq j$).

We can now view the modular conjugation $J_{\msG\ltimes A}$ as a partial anti-isometry on $L^2(\msG) \otimes L^2(A)$ with source and range projection equal to the projection on $L^2(\msG) \underset{\C^2}{\otimes} L^2(A)$, and splitting up into anti-unitaries 
\[
J^{\ltimes}_{ij}\colon L^2(\G_{ij}) \otimes L^2(A_j) \rightarrow L^2(\G_{ji}) \otimes L^2(A_i).
\] 
Mimicking \cite{Vae01}, we get the partial isometric implementation of $\gamma$ as\footnote{Note that our convention on corepresentations is opposite to \cite{Vae01}, which will result in our formulas being adjoints to the ones in \cite{Vae01}.}
\begin{equation}\label{EqIsometImpl}
U_{\gamma} = (J_{\hat{Q}}\otimes J_A)J_{\msG \ltimes A},
\end{equation}
which is a unitary $\msG$-representation with decomposition into unitaries
\[
U_{\gamma,ij}\colon L^2(\G_{ij}) \otimes L^2(A_j) \rightarrow L^2(\G_{ij})\otimes L^2(A_i),\qquad U_{\gamma,ij} \in L^{\infty}(\G_{ij})\bar{\otimes} \mcB(L^2(A_j),L^2(A_i)).
\]
It implements the action of $\G$ on $A$ via 
\[
\gamma(a) = U_{\gamma}^*(1\otimes a)U_{\gamma},\qquad \forall a \in A. 
\]
Let us note that for $i\in \{1,2\}$, the component $\gamma_i^i$ is an action of $\G_{ii}$ on $A_i$. Furthermore, as the above modular conjugations are compatible with taking projections, the unitary $U_{\gamma,ii}\in L^{\infty}(\G_{ii})\bar\otimes \mcB(L^2(A_i))$ will be equal to $U_{\gamma_i^i}$, the standard implementation of $\gamma^i_i$.

\subsection{Closed linking quantum subgroupoids}

Assume that $(\hat{Q},\Delta_{\hat{Q}})$ and $(\hat{Q}_1,\Delta_{\hat{Q}_1})$ define two linking quantum groupoids
\[
\hat{Q} = W^*(\msG) =  (W^*(\G_{ij}))_{i,j=1}^{2},\qquad \hat{Q}_1 = W^*(\msG_1) = (W^*(\G_{1,ij}))_{i,j=1}^{2}.
\]
We then say that $\msG_1$ is a \emph{quotient linking quantum groupoid} of $\msG$ if we have given a $\Delta$-intertwining faithful, normal, unital $*$-homomorphism
\[
\theta = \bigoplus_{i,j=1}^{2}\theta_{ij}\colon L^{\infty}(\msG_1) \rightarrow L^{\infty}(\msG),\quad \theta_{ij}\colon L^{\infty}(\G_{1,ij}) \rightarrow L^{\infty}(\G_{ij}),\qquad   (\theta_{ik}\otimes \theta_{kj})\Delta_{1,ij}^k = \Delta_{ij}^k \theta_{ij}. 
\]
Borrowing terminology of locally compact quantum groups, we also say that $\hat{\msG}_1$ is a \emph{closed quantum subgroupoid} of $\hat{\msG}$ \cites{Vae05,DKSS12}. 

By means of the above $\theta_{ij}$'s, we get a surjective $*$-homomorphism 
\[
\hat{\theta}\colon C^*(\msG) \rightarrow C^*(\msG_1),\qquad (\id\otimes \omega)(\Ww_{\hat{Q}}) \mapsto (\id\otimes \omega\theta)\Ww_{\hat{Q}_1},\qquad \forall \omega \in Q_*, 
\]
preserving the matrix decomposition: 
\[
\hat{\theta} = \begin{pmatrix} \hat{\theta}_{11}& \hat{\theta}_{12}\\ \hat{\theta}_{21} & \hat{\theta}_{22}\end{pmatrix},\qquad \hat{\theta}_{ij}\colon C^*(\G_{ij})\rightarrow C^*(\G_{1,ij}). 
\]
These maps then also preserve the coproducts: 
\[
(\hat{\theta}_{ij}\otimes \hat{\theta}_{ij})\hat{\Delta}_{u,ij} = \hat{\Delta}_{1,u,ij}\hat{\theta}_{ij}. 
\]

\section{Induction and standard Hilbert space for twisted actions}

\subsection{Induction of actions}

Let $\G$ be a LCQG, and assume that $\G = \G_{22}$ is a corner of a linking quantum groupoid $\msG = \begin{pmatrix} \G_{11} & \G_{12}\\ \G_{21} & \G_{22}\end{pmatrix}= \begin{pmatrix} \Hh & \X \\ \Y & \G\end{pmatrix}$ with associated function W$^*$-algebra 
\[
Q = L^{\infty}(\msG) = L^{\infty}(\Hh)\oplus L^{\infty}(\Y)\oplus L^{\infty}(\X)\oplus L^{\infty}(\G) = P \oplus O\oplus N\oplus M. 
\] 

Let $(A,\gamma)$ be a left $\G$-W$^*$-algebra. Then using \eqref{EqFormAlph}, we define\footnote{Note that $\Ind_{\X}(A)$ can be seen as a von Neumann algebraic version of the \emph{cotensor product} $N\overline{\square}A$, see \cite{HS20}*{Section 2.2}.}
\begin{equation}\label{EqInductionAction}
\Ind_{\X}(A) =  \{z\in N \bar{\otimes}A\mid (\alpha_{\X}\otimes \id)z = (\id\otimes \gamma)z\} \subseteq N\bar{\otimes}A.
\end{equation}
This is a W$^*$-algebra, which by the weak Fubini slice map property for W$^*$-algebras carries a left $\Hh$-action
\[
\Ind_{\X}(\gamma) = (\gamma_{\X}\otimes \id)_{\mid \Ind_{\X}(A)}\colon  \Ind_{\X}(A)  \rightarrow L^{\infty}(\Hh)\bar{\otimes} \Ind_{\X}(A),
\]
with $\gamma_{\X}$ as in \eqref{EqFormGamm}.

\begin{Def}
We call $\Ind_{\X}(A)$ the \emph{induced W$^*$-algebra}, and $\Ind_{\X}(\gamma)$ the \emph{induced $\Hh$-action}. If $\msG$ arises from a dual unitary $2$-cocycle $\hat{\Omega}\in \hat{M}\bar{\otimes}\hat{M}$ as in \eqref{EqCoprodCocTwist}, we also write 
\begin{equation}
\Ind_{\X}(A) = A_{\hat{\Omega}}. 
\end{equation}
\end{Def}

The following is \cite{DC09}*{Theorem 8.2.2}. Note that the reflected groupoid (see Remark \ref{rem1}) is used to define the second induction.

\begin{Theorem}\label{TheoInductInvol}
The above correspondence is involutive: there is a natural identification of $\G$-W$^*$-algebras
\[
 (A,\gamma)\cong (\Ind_{\Y}(\Ind_{\X}(A)),\Ind_{\Y}(\Ind_{\X}(\gamma))), 
\]
provided through co-restricting the normal, unital $*$-homomorphism
\[
A \rightarrow O \bar{\otimes}N \bar{\otimes} A,\quad a \mapsto (\Delta_{22}^1\otimes \id)\gamma(a).
\]
\end{Theorem}

This theorem can be used to give a one-to-one-correspondence between actions of $\G$ and actions of $\msG$. Namely, assume that we are given a $\G$-W$^*$-algebra $(A,\gamma)$, put  
\begin{equation}\label{EqInducedAction}
\Ind_{\msG}(A) = A_1\oplus A_2 := \Ind_{\X}(A)\oplus A.
\end{equation}
 We endow $\Ind_{\msG}(A)$ with an action $\Ind_{\msG}(\gamma)$ of $\msG$ via the normal, unital $*$-homomorphisms 
\[
\gamma_i^j\colon A_i \rightarrow L^{\infty}(\mathbb{G}_{ij})\bar{\otimes} A_j
\]
defined as follows: 
\begin{itemize}
\item we have $\gamma_2^2 = \gamma$,
\item the map $\gamma_2^1$ is given by 
\begin{equation}\label{eq50}
\gamma_2^1 =(\Delta_{22}^1\otimes \id) \gamma\colon A_2 \rightarrow L^{\infty}(\G_{21})\bar{\otimes} A_1,
\end{equation}
\item the map $\gamma_1^2$ is just the inclusion map $\Ind_{\X}(A) \subseteq L^{\infty}(\G_{12})\bar{\otimes} A$, and finally
\item the map $\gamma_1^1$ is
\[
\gamma_1^1 = (\Delta_{12}^1\otimes \id)_{\mid A_1}=\Ind_{\X}(\gamma). 
\]
\end{itemize}

\begin{Prop}
The assignment $(A,\gamma)\mapsto (\Ind_{\msG}(A),\Ind_{\msG}(\gamma))$ establishes a bijection between isomorphism classes of actions of $\G=\G_{22}$ and $\msG$.
\end{Prop}

\begin{proof}
One easily checks that $\Ind_{\msG}(\gamma)$ is an action of $\msG$ on the von Neumann algebra $\Ind_{\msG}(A)$.

Conversely, assume that $\msG$ acts with $\gamma$ on $A=A_1\oplus A_2$. Then $\gamma_1^1$ (resp.~$\gamma_2^2$) establishes an action of $\G_{11}$ on $A_1$ (resp.~$\G_{22}$ on $A_2$). We claim that $\gamma_1^2\colon A_1\rightarrow \Ind_{\X}(A_2)$ is a well defined $\G_{11}$-equivariant isomorphism, where $\G_{11}$ acts on $\Ind_{\X}(A_2)$ via 
\[
\Ind_{\X}(\gamma_2^2)=(\Delta^1_{12}\otimes\id)_{|\Ind_{\X}(A_2)}=\Ind_{\msG}(\gamma_2^2)_1^1.
\]

Indeed, observe first that the action condition implies $\gamma^2_1(A_1)\subseteq\Ind_{\X}(A_2)$, thus we can consider $\gamma^2_1$ as an injective map $A_1\rightarrow \Ind_{\X}(A_2)$. Next, as $(\Delta^1_{12}\otimes \id) \gamma^2_1 = 
(\id\otimes \gamma^2_1) \gamma^1_1$, we have that $\gamma^2_1$ is $\G_{11}$-equivariant. Consequently, after taking induction, we get an injective $\G_{22}$-equivariant map
\begin{equation}\label{eq1}
\Ind_{\Y}(A_1)\ni z \mapsto (\id\otimes \gamma^2_1)z\in 
\Ind_{\Y}(\Ind_{\X}(A_2)).
\end{equation}
 Using Theorem \ref{TheoInductInvol} we have
\[
(\id \otimes \gamma_1^2)(\Ind_{\Y}(A_1)) \supseteq (\id\otimes \gamma_1^2)\gamma_2^1(A_2) = (\Delta_{22}^1\otimes \id)\gamma_2^2(A_2) = \Ind_{\Y}(\Ind_{\X}(A_2)),
\]
hence \eqref{eq1} is an isomorphism. Applying now the induction $\Ind_{\X}$ to \eqref{eq1}, we get a chain of isomorphisms
\begin{equation}\label{eq2}
A_1\cong \Ind_{\X}(\Ind_{\Y}(A_1))\ni z\mapsto
(\id\otimes\id\otimes\gamma^2_1)z\in 
\Ind_{\X}(\Ind_{\Y}(\Ind_{\X}(A_2) ))\cong 
\Ind_{\X}(A_2),
\end{equation}
where the first and the third isomorphisms are given by Theorem \ref{TheoInductInvol}. On elements, \eqref{eq2} is given by
\[\begin{split}
A_1\ni a\mapsto (\Delta^1_{22}\otimes\id)\gamma_1^1(a)&\mapsto 
(\Delta^1_{22}\otimes\id \otimes \id)
(\id\otimes \gamma^2_1)\gamma_1^1(a)=
(\Delta^1_{22}\otimes\id \otimes \id)
(\Delta^1_{12}\otimes \id)\gamma_1^2(a)\\
&=
(\Delta^1_{22}\otimes\id\otimes\id)
\Ind_{\X}(\gamma^2_2)\gamma^2_1(a)\mapsto
\gamma^2_1(a),
\end{split}\]
hence it is equal to $\gamma_1^2$. This proves the claim.

Now, taking the direct sum with the identity, we obtain an isomorphism 
\[
\gamma^2_1\oplus\id\colon A=A_1\oplus A_2\cong \Ind_{\X}(A_2)\oplus A_2 = \Ind_{\msG}(A_2)
\] 
of von Neumann algebras; note that it preserves the direct sum decomposition. Next we check that it is $\msG$-equivariant, where the right hand side is equipped with action $\Ind_{\msG}(\gamma_2^2)$, as explained above. Equivalently, we will check that $\gamma_i^j$ is mapped to $\Ind_{\msG}(\gamma_2^2)_i^j$ for $i,j\in \{1,2\}$.

This claim is trivial for $(i,j)=(2,2)$ as $\Ind_{\msG}(\gamma_2^2)^2_2=\gamma^2_2$, and we have already checked it for $(i,j)=(1,1)$. If $(i,j)=(1,2)$ then $\Ind_{\msG}(\gamma^2_2)_1^2$ is the inclusion map $\Ind_{\X}(A_2)\rightarrow L^{\infty}(\G_{12})\bar\otimes A_2$, hence the claim is also trivial. For $(i,j)=(2,1)$ we have
\[
(\id\otimes \gamma^2_1)\gamma_2^1=
(\Delta^1_{22}\otimes\id)\gamma^2_2=\Ind_{\msG}(\gamma_2^2)^1_2.
\]
\end{proof}

\subsection{Induction of representations} 

Let $\msG$ be a linking quantum groupoid between the LCQGs $\G$ and $\Hh$. 

As a particular case of induction of actions, we can induce unitary left $\G$-representations. Indeed, if $U\in M\bar{\otimes} \mcB(\Hsp)$ is a unitary left $\G$-representation, endow $A := \mcB\begin{pmatrix} \Hsp \\\C\end{pmatrix}= \begin{pmatrix} \mcB(\Hsp) & \Hsp \\ \overline{\Hsp} & \C\end{pmatrix}$ with the left $\G$-action
\[
\gamma_U\colon  A \rightarrow M\bar{\otimes} A,\qquad  \begin{pmatrix} x & \xi \\ \eta^* & \lambda\end{pmatrix} \mapsto \begin{pmatrix} U^*(1\otimes x)U & U^*(1\otimes \xi) \\ (1\otimes \eta^*)U & 1\otimes \lambda\end{pmatrix},
\]
where we identify $\ov{\mc{H}}=\mc{H}^*$, $M\bar\otimes A=\begin{pmatrix} M\bar\otimes\mcB(\Hsp) & M\bar\otimes\Hsp \\ M\bar\otimes\overline{\Hsp} & M\bar\otimes\C\end{pmatrix}$ and identify elements of $M\bar\otimes \mc{H}$, $M\bar\otimes \ov{\mc{H}}$ with operators $L^2(M)\rightarrow L^2(M)\otimes \mc{H}$ and $L^2(M)\otimes \mc{H}\rightarrow L^2(M)$ as explained in the introduction. The induced W$^*$-algebra $\Ind_{\X}(A)$ again splits as a matrix algebra, and the $12$-corner can be identified with\footnote{Similarly as $\Ind_{\X}(A)$, the space $\Ind_{\X}(\Hsp)$ can be seen as an analytic version of cotensor product $N\overline{\square}\mc{H}$.} 
\[
\Ind_{\X}(\Hsp) = \{\xi\in N \bar{\otimes} \Hsp \mid (\alpha_{\X} \otimes \id)\xi = (\id\otimes U^*)\xi_{[13]}\} \subseteq \mcB(L^2(N),L^2(N)\otimes \Hsp). 
\]
This is in fact a Hilbert space via
\[
\langle \xi ,\eta\rangle = \xi^*\eta,\qquad \forall\xi,\eta \in \Ind_{\X}(\mc{H}), 
\]
using that $(N,\alpha_{\X})$ is ergodic. We then obtain an induced unitary $\Hh$-representation $\Ind_{\X}(U)$ on $\Ind_{\X}(\Hsp)$ via 
\[
\Ind_{\X}(U)^*(\eta\otimes \xi) = (\gamma_{\X}\otimes \id)(\xi)(\eta\otimes 1),\qquad \forall \eta\in L^2(\Hh),\xi\in \Ind_{\X}(\Hsp),
\]
where the right hand side is seen as a map $L^2(\X)\rightarrow L^2(\Hh)\otimes L^2(\X)\otimes \mc{H}$.

Indeed, it is straightforward to check that $\Ind_{\X}(U)^*$ is a well defined isometry in $L^{\infty}(\Hh)\bar\otimes \mcB(\Ind_{\X}(\mathcal{H}))$. We also have 
\begin{equation}\label{EqCorepProp}
(\Delta_{\Hh}\otimes\id)(\Ind_{\X}(U)^*)=\Ind_{\X}(U)^*_{[23]}\Ind_{\X}(U)^*_{[13]}.
\end{equation}
Indeed, for $\eta_1,\eta_2\in L^2(\Hh), \xi\in \Ind_{\X}(\mc{H})$
\[\begin{split}
&\quad\;
(\Delta_{\HH}\otimes\id)(\Ind_{\X}(U)^*)(\eta_1\otimes\eta_2\otimes \xi)=
\ww^{*}_{\HH  [12]} 
\Ind_{\X}(U)^*_{[23]}
( \ww_{\HH}(\eta_1\otimes\eta_2)\otimes\xi)\\
&=
\ww^{*}_{\HH [12]} 
(1\otimes (\gamma_{\X}\otimes\id)(\xi))
( \ww_{\HH}(\eta_1\otimes\eta_2)\otimes 1)=
((\Delta_{\HH}\otimes\id)\gamma_{\X}\otimes\id)(\xi)
(\eta_1\otimes\eta_2\otimes 1)\\
&=
((\id\otimes\gamma_{\X})\gamma_{\X}\otimes\id)(\xi)
(\eta_1\otimes\eta_2\otimes 1)
\end{split}\]
and on the other hand
\[\begin{split}
&\quad\;
\Ind_{\X}(U)^*_{[23]} \Ind_{\X}(U)^*_{[13]}
(\eta_1\otimes\eta_2\otimes\xi)=
\Ind_{\X}(U)^*_{[23]} 
(\gamma_{\X}\otimes\id)(\xi)_{[134]}
(\eta_1\otimes\eta_2\otimes 1)\\
&=
( (\id\otimes\gamma_{\X})\gamma_{\X}\otimes \id)(\xi)
(\eta_1\otimes\eta_2\otimes 1),
\end{split}\]
which proves \eqref{EqCorepProp}. Consequently, \cite{BDS13}*{Corollary 4.15} implies that $\Ind_{\X}(U)$ is unitary.
\begin{Def}
We call $(\Ind_{\X}(\Hsp),\Ind_{\X}(U))\in\Rep(\Hh)$ the \emph{induced unitary $\Hh$-representation} from $(\Hsp,U)\in \Rep(\G)$.
\end{Def}

The above induction can also be performed on the level of $\msG$ itself. Namely, starting from $(\Hsp,U) \in \Rep(\G)$, put 
\[
\Ind_{\msG}(\Hsp) = \begin{pmatrix} \Ind_{\X}(\Hsp) \\ \Hsp\end{pmatrix},
\] 
and define on it the unitary $\msG$-representation $\Ind_{\msG}(U)$, for whose components (see Definition \ref{def1}) we use the shorthand notation $U_{ij}$ where  
\begin{equation}
U_{22} = U,\qquad U_{11} = \Ind_{\X}(U)
\end{equation}
 and $U_{12},U_{21}$ are defined via 
\begin{equation}\begin{split}\label{eq44}
U_{12}^*&\colon L^2(\G_{12})\otimes \Ind_{\X}(\Hsp) \rightarrow L^2(\G_{12})\otimes \Hsp,\quad \xi\otimes z \mapsto  z \xi, \\
U_{21}^*&\colon L^2(\G_{21}) \otimes \Hsp \rightarrow L^2(\G_{21})\otimes \Ind_{\X}(\Hsp),\quad \xi \mapsto (\Delta_{22}^1\otimes \id)(U^*)\xi_{[13]}.
\end{split}\end{equation}

\begin{Lem}
$\Ind_{\msG}(U)$ is a unitary representation of $\msG$ on $\Ind_{\msG}(\mc{H})$.
\end{Lem}

\begin{proof}
It is easily checked that $U_{ij}^*$ is an isometry in $L^{\infty}(\G_{ij})\bar\otimes \mc{B}(\mc{H}_i,\mc{H}_j)$ (where $\mc{H}_1=\Ind_{\X}(\mc{H})$, $\mc{H}_2=\mc{H}$) and $U_{ii}$ are unitary. Next we claim that $(\Delta_{ij}^k\otimes\id)(U_{ij})=U_{ik[13]}U_{kj [23]}$ for $i,j,k\in \{1,2\}$, or equivalently
\begin{equation}\label{eq3}
(\Delta_{ij}^k\otimes\id)(U_{ij}^*)=U_{kj [23]}^*U_{ik[13]}^*.
\end{equation}
Let us show this for $i=j=1,k=2$. Take $\eta_1\in L^2(\GG_{12})$, $\eta_2\in L^2(\GG_{21})$, $\xi\in \Ind_{\X}(\mc{H})$, then
\[\begin{split}
&\quad\;
(\Delta^2_{11}\otimes \id)(U^*_{11})(\eta_1\otimes\eta_2\otimes\xi)=
(\Delta^2_{11}\otimes \id)(\Ind_{\X}(U)^*)(\eta_1\otimes\eta_2\otimes\xi)\\
&=
\Sigma_{[12]}
\hat{\ww}^{1}_{12[12]} \Ind_{\X}(U)^*_{[13]}
\hat{\ww}^{1 *}_{12 [12]} (\eta_2\otimes\eta_1\otimes\xi)\\
&=
\Sigma_{[12]}
\hat{\ww}^{1}_{12[12]} 
(\gamma_{\X}\otimes\id)(\xi)_{[134]}
(\hat{\ww}^{1 *}_{12 } (\eta_2\otimes\eta_1))\\
&=
 (\Delta_{11}^2\otimes\id\otimes\id)(\Delta_{12}^1\otimes\id)(\xi)
 (\eta_1\otimes\eta_2)\\
 &=
  (\id\otimes \Delta_{22}^{1} \otimes\id)(\Delta_{12}^2\otimes\id)(\xi)
 (\eta_1\otimes\eta_2)\\
 &= 
  (\id\otimes \Delta_{22}^{1} \otimes\id)(\alpha_{\X}\otimes\id)(\xi)
 (\eta_1\otimes\eta_2)\\
 &=
  (\id\otimes \Delta_{22}^{1} \otimes\id)
  ( (\id\otimes U^*)\xi_{[13]})
 (\eta_1\otimes\eta_2)\\
 &=
   (\Delta_{22}^{1} \otimes\id)
  ( U^*)_{[234]}\xi_{[1 4]}
 (\eta_1\otimes\eta_2)
\end{split}\]
and on the other hand
\[\begin{split}
&\quad\;
U^*_{21 [23]} U^*_{12[13]} (\eta_1\otimes\eta_2\otimes\xi)=
\Sigma_{[12]} U^*_{21 [13]} U^*_{12[23]} (\eta_2\otimes\eta_1\otimes\xi)=
\Sigma_{[12]} U^*_{21 [13]}  (\eta_2\otimes\xi\eta_1)\\
&=
\Sigma_{[12]} 
(\Delta^1_{22}\otimes\id)(U^*)_{[134]} (\eta_2\otimes\xi\eta_1)_{[124]}=
(\Delta^1_{22}\otimes\id)(U^*)_{[234]} (\eta_2\otimes\xi\eta_1)_{[214]},
\end{split}\]
which proves the claim. The remaining equations \eqref{eq3} can be checked in a similar way.

Since $\Delta_{11}^2$ is a unital $*$-homomorphism, the left hand side of
\begin{equation}\label{eq4}
(\Delta_{11}^2\otimes\id)(U_{11})=U_{12 [13]} U_{21 [23]}
\end{equation}
is unitary, in particular it is an isometry. Now, we already know that $U_{12}$ is a contraction. If $U_{21}$ is not an isometry, we can find $\zeta\in L^2(\GG_{21})\otimes \Ind_{\X}(\mc{H})$ such that $\|U_{21} \zeta\|<\|\zeta\|$. For any $0\neq \eta\in L^2(\GG_{12})$ we have
\[
\|\eta\|\, \|\zeta\|=\| (\Delta_{11}^2\otimes\id)(U_{11}) (\eta\otimes\zeta)\| = 
\|U_{12 [13]} U_{21 [23]}(\eta\otimes\zeta)\| \le 
\|\eta\|\,\|U_{21} \zeta\|<\|\eta\| \,\|\zeta\|
\]
which gives a contradiction. Consequently, $U_{21}$ is unitary. Now equation \eqref{eq4} implies that $U_{12}$ is also unitary and consequently $\Ind_{\msG}(U)$ is indeed a unitary representation.
\end{proof}

\begin{Theorem}\label{TheoIndRep}
The assignment 
\begin{equation}\label{EqMonoidEquiv}
\Ind_{\X}\colon \Rep(\G) \rightarrow \Rep(\Hh),\qquad (\Hsp,U) \mapsto (\Ind_{\X}(\Hsp),\Ind_{\X}(U)),\quad 
T\mapsto 1\otimes T
\end{equation}
becomes an equivalence of tensor W$^*$-categories through the unique unitaries determined by 
\begin{equation}\label{EqInterTwinerMonoid}
u_{\Hsp,\Hsp'}\colon \Ind_{\X}(\Hsp)\otimes \Ind_{\X}(\Hsp') \rightarrow \Ind_{\X}(\Hsp\otimes \Hsp'),\qquad \xi\otimes \eta \mapsto \eta_{[13]}\xi_{[12]}.
\end{equation}
\end{Theorem}

\begin{proof}
Observe that the C$^*$-algebras $C^*(\G),C^*(\msG)$ are  C$^*$-Morita equivalent \eqref{EqStrongMorita}, hence the C$^*$-algebraic induction gives an equivalence of W$^*$-categories $\Ind^{C^*}_{\msG}\colon \Rep(C^*(\G))\rightarrow \Rep(C^*(\msG))$ (see \cite{RW98}*{Example 3.6, Theorem 3.29}). Consider the restriction functor
\[
\Res^{C^*}_{\msG}\colon \Rep(C^*(\msG)) \rightarrow \Rep(C^*(\G)),\qquad (\Hsp,\hat{\pi}) \mapsto (\Hsp_2,\hat{\pi}_2),\quad T_1\oplus T_2\mapsto T_2
\]
with $\Hsp=\Hsp_1\oplus\Hsp_2$ and $\hat{\pi}_2$ determined through 
\[
\begin{pmatrix} 0 \\ \hat{\pi}_2(x)\xi\end{pmatrix}  = \hat{\pi}\begin{pmatrix} 0 & 0 \\ 0 & x\end{pmatrix}\begin{pmatrix}0 \\ \xi\end{pmatrix},\qquad \forall \xi\in\Hsp_2,x\in C^*(\G).
\] 
 It follows from the concrete description of $\Ind^{C^*}_{\msG}$ that $\Res^{C^*}_{\msG}\circ \Ind^{C^*}_{\msG}\cong\id$, in particular $\Res^{C^*}_{\msG}$ is also an equivalence of W$^*$-categories. 
 
This restriction functor is clearly monoidal, through the identity map
\[
\Hsp_{\Res(\pi \ast\pi')} = \Hsp_{\pi,2}\otimes \Hsp_{\pi',2} = \Hsp_{\Res(\pi)*\Res(\pi')}.
\]

Using the tensor W$^*$-equivalences $\Rep(\msG) \cong \Rep(C^*(\msG))$ and $\Rep(\G)\cong \Rep(C^*(\G))$, we then also find the tensor W$^*$-equivalence 
\[
\Res_{\msG}\colon \Rep(\msG) \rightarrow \Rep(\G),\qquad (\Hsp,U) \mapsto (\Hsp_2,U_{22}),\quad 
T_1\oplus T_2\mapsto T_2
\]
the coherence map again being the identity. 

Now, consider the induction functor
\[
\Ind_{\msG}\colon \Rep(\G)\rightarrow \Rep(\msG),\quad (\Hsp,U)\mapsto (\Ind_{\msG}(\Hsp),\Ind_{\msG}(U)),\quad
T\mapsto (1\otimes T)\oplus T.
\]
Since $\Res_{\msG}\circ \Ind_{\msG}=\id $, we conclude that $\Ind_{\msG}$ is an equivalence of W$^*$-categories. This then already shows that \eqref{EqMonoidEquiv} will be an equivalence of W$^*$-categories. 

To see that it is monoidal through \eqref{EqInterTwinerMonoid}, we note first that $u_{\Hsp,\Hsp'}$ in \eqref{EqInterTwinerMonoid} is clearly an isometric map and it is not difficult to check that $u_{\Hsp,\Hsp'}^*$ is an intertwiner of unitary $\Hh$-representations. Furthermore, $u_{\Hsp,\Hsp'}$ lifts to an isometric intertwiner 
\begin{equation}\label{EqInterTwinerMonoidAlt}
\widetilde{u}_{\Hsp,\Hsp'}\colon \Ind_{\msG}(\Hsp)\otimes_{\msG} \Ind_{\msG}(\Hsp') \rightarrow \Ind_{\msG}(\Hsp\otimes_{\G} \Hsp'),\qquad \begin{pmatrix} \xi\otimes \eta \\ \zeta\end{pmatrix} \mapsto \begin{pmatrix} \eta_{[13]}\xi_{[12]}\\ \zeta\end{pmatrix}
\end{equation}
of unitary $\msG$-representations (see \eqref{EqTensProd2}). Let us check this for index $21$. Take $\zeta=\zeta_1\otimes\zeta_2\in \mc{H}\otimes\mc{H}',\theta\in L^2(\GG_{21})$. We have
\[
\Ind_{\msG}(U\otop U')_{21}^* (1\otimes \widetilde{u}_{\mc{H},\mc{H}'}) (\theta\otimes \zeta)=
\Ind_{\msG}(U\otop U')_{21}^*(\theta\otimes \zeta)
=
(\Delta_{22}^1\otimes\id)(U_{[12]}U'_{[13]})^* 
(\theta\otimes 1 \otimes \zeta)
\]
and on the other hand
\[\begin{split}
&\quad\;
(1\otimes \widetilde{u}_{\mc{H},\mc{H}'})
(\Ind_{\msG}(U)\otop\Ind_{\msG}(U'))_{21}^*  (\theta\otimes \zeta)=
(1\otimes \widetilde{u}_{\mc{H},\mc{H}'})
(\Ind_{\msG}(U)_{21[12] } \Ind_{\msG}(U')_{21[13]})^*  (\theta\otimes \zeta_1\otimes\zeta_2)\\
&=
(1\otimes\widetilde{u}_{\mc{H},\mc{H}'})
\Ind_{\msG}(U')_{21 [14]}^*
(\Delta^1_{22}\otimes \id)(U^*)_{[123]}
(\theta\otimes 1\otimes \zeta_1\otimes\zeta_2)\\
&=
(1\otimes\widetilde{u}_{\mc{H},\mc{H}'})
(\Delta^1_{22}\otimes\id)(U'^*)_{[145]}
(\Delta^1_{22}\otimes \id)(U^*)_{[123]}
(\theta\otimes 1\otimes \zeta_1\otimes 1\otimes \zeta_2)\\
&=
(\Delta^1_{22}\otimes\id)(U'^*)_{[124]}
(\Delta^1_{22}\otimes \id)(U^*)_{[123]}
(\theta\otimes 1\otimes \zeta_1\otimes \zeta_2)\\
&=
(\Delta^1_{22}\otimes \id)( U'^*_{[13]}U^*_{[12]})
(\theta\otimes 1\otimes \zeta_1\otimes \zeta_2).
\end{split}\]
The remaining equations can be checked in a similar way. Since $\tilde{u}_{\Hsp,\Hsp'}$ becomes the identity under $\Res_{\msG}$, we must necessarily have that $\widetilde{u}_{\Hsp,\Hsp'}$ is a unitary, providing the monoidal structure on $\Ind_{\msG}$. It then follows that this claim holds for the $u_{\Hsp,\Hsp'}$ as well. 
\end{proof}

\begin{Rem}
It follows from the proof of Theorem \ref{TheoIndRep} that, after the identifications $\Rep(\G)\cong \Rep(C^*(\G))$ and  $\Rep(\msG)\cong \Rep(C^*(\msG))$, we get a natural isomorphism between induction functors:
\[
\bigl(\,\Ind^{C^*}_{\msG}\colon \Rep(C^*(\G))\rightarrow \Rep(C^*(\msG))\,\bigr)\quad \cong \quad \bigl(\,\Ind_{\msG}\colon \Rep(\G)\rightarrow \Rep(\msG)\,\bigr);
\] 
a conclusion which is not immediate just by looking at the definitions.
\end{Rem}

We will also need a more concrete form of the  natural isomorphism 
\[
\bigl(\,\Ind_{\X}\colon \Rep(\GG)\rightarrow \Rep(\HH) \,\bigr)
\quad \cong \quad 
\bigl(\, \Ind_{\X}^{C^*}\colon \Rep(C^*(\GG))\rightarrow \Rep(C^*(\HH))\,\bigr).
\] 
To make this precise, let us define 
\begin{equation}
\widetilde{\Ind}_{\X}\colon \Rep(C^*(\GG))\rightarrow \Rep(C^*(\HH))
\end{equation}
as the composition of $\Ind_{\X}$ with the isomorphisms $\Rep(C^*(\GG))\cong \Rep(\GG)$ and  $\Rep(\HH)\cong \Rep(C^*(\HH))$, and similarly define 
\begin{equation}
\widetilde{\Ind}_{\msG}\colon \Rep(C^*(\GG))\rightarrow \Rep(C^*(\msG))
\end{equation}
as the composition of $\Ind_{\msG}$ with $\Rep(C^*(\G))\cong \Rep(\GG)$ and $\Rep(\msG)\cong \Rep(C^*(\msG))$.

\begin{Lem}\label{lemma9}
Let $(\mc{H},U)\in \Rep(\G)$ with the associated $*$-representation $(\mc{H},\hat{\pi})\in \Rep(C^*(\G))$, and define $\hat{\pi}_{ij}$ to be the components of $\widetilde{Ind}_{\msG}(\hat{\pi})\colon C^*(\msG)\rightarrow \mcB(Ind_{\msG}(\mc{H}))$. The family of unitaries
\[
\theta_{(\mc{H},\hat{\pi})}\colon 
C^*(\X)\otimes_{C^*(\G)}\mc{H} \ni b\otimes\xi\mapsto \hat{\pi}_{12}(b)\xi\in \Ind_{\X}(\mc{H})
\]
establishes a natural isomorphism 
\[
\Ind_{\X}^{C^*}\cong \widetilde{\Ind}_{\X}.
\]
\end{Lem}

\begin{proof}
It is straightforward to check that $\theta_{(\mc{H},\hat{\pi})}$ is well-defined and isometric. Let us prove that it is surjective, using that $\widetilde{\Ind}_{\msG}(\hat{\pi})$ is non-degenerate:
\[\begin{split}
\Ind_{\X}(\mc{H})&=
\ov{\operatorname{span}}\{ \hat{\pi}_{11}(a)\xi + \hat{\pi}_{12}(b)\eta\mid a\in C^*(\HH),\xi\in \Ind_{\X}(\mc{H}), b\in C^*(\X), \eta\in \mc{H}\}\\
&=
\ov{\operatorname{span}}\{ \hat{\pi}_{12}(b')\hat{\pi}_{21}(c)\xi + \hat{\pi}_{12}(b)\eta\mid c\in C^*(\Y),\xi\in \Ind_{\X}(\mc{H}), b,b'\in C^*(\X), \eta\in \mc{H}\}\\
&\subseteq 
\ov{\operatorname{span}}\{  \hat{\pi}_{12}(b)\xi\mid b\in C^*(\X),\xi\in \mc{H} \}\subseteq \Ind_{\X}(\mc{H}).
\end{split}\]
Next we check that $\theta_{(\mc{H},\hat{\pi})}$ is intertwiner: for $a\in C^*(\HH),b\in C^*(\X),\xi\in \mc{H}$ we have
\[\begin{split}
&\quad\;
\theta_{(\mc{H},\hat{\pi})} \Ind_{\X}^{C^*}(\hat{\pi})(a) (b\otimes \xi)=
 \hat{\pi}_{12}(ab)\xi=
 \hat{\pi}_{11}(a)\hat{\pi}_{12}(b)\xi=
 \widetilde{\Ind}_{\X}(\hat{\pi})(a) \theta_{(\mc{H},\hat{\pi})}(b\otimes \xi).
\end{split}\]
Finally, it is easy to see that the family $(\theta_{(\mc{H},\hat{\pi})})_{(\mc{H},\hat{\pi})\in \Rep(C^*(\GG))}$ is natural, which ends the proof.
\end{proof}

\subsection{Induction of actions}\label{SubSecStatement}

Let $\msG$ be a linking quantum groupoid between the LCQGs $\G$ and $\Hh$. Let  $(A,\gamma)$ be a left $\G$-action on a W$^*$-algebra $A$, and let $U_{\gamma}$ be the canonical unitary $\G$-representation on $L^2(A)$ implementing $\gamma$.

\begin{Theorem}\label{TheoIndL}
There is a canonical identification of unitary $\Hh$-representations
\[
I\colon (L^2(\Ind_{\X}(A)),U_{\Ind_{\X}(\gamma)}) \cong (\Ind_{\X}(L^2(A)),\Ind_{\X}(U_{\gamma})).
\]
Moreover, the induced left $\Ind_{\X}(A)$-representation on $\Ind_{\X}(L^2(A))$ is given by left multiplication;
\[
I\pi_{\Ind_{\X}(A)}(w)I^* \eta = (\id\otimes \pi_A)(w)\eta,\qquad \forall w\in \Ind_{\X}(A), \eta\in \Ind_{\X}(L^2(A)).  
\]
\end{Theorem}

\begin{proof}
Consider $\Ind_{\msG}(A) = A_1\oplus A_2 = \Ind_{\X}(A)\oplus A$, with $\msG$-action $\Ind_{\msG}(\gamma)$ and implementation $U_{\Ind_{\msG}(\gamma)}$ as in \eqref{EqInducedAction} and Section \ref{sec:groupoidimplementation}. Define
\[
I\colon L^2(\Ind_{\X}(A)) \rightarrow \Ind_{\X}(L^2(A)) \subseteq L^{\infty}(\X)\bar\otimes L^2(A),\qquad \xi \mapsto U_{\Ind_{\msG}(\gamma),12}^*(1\otimes \xi). 
\]
Clearly $I$ is an isometric map $L^2(\Ind_{\X}(A))\rightarrow L^{\infty}(\X)\bar\otimes L^2(A)$, and it is easy to check that the image of $I$ indeed lands in $\Ind_{\X}(L^2(A))$. The isometry $I$ is an intertwiner between the unitary $\HH$-representations $U_{\Ind_{\X}(\gamma)}$ and $\Ind_{\X}(U_\gamma)$. Indeed, for $\xi\in L^2(\Ind_{\X}(A)),\eta\in L^2(\HH)$ we have
\[\begin{split}
&\quad\;
\Ind_{\X}(U_\gamma)^* (1\otimes I)(\eta\otimes \xi)=
\Ind_{\X}(U_\gamma)^*
U^*_{\Ind_{\msG}(\gamma),12 [23]}(\eta\otimes 1\otimes \xi)=
(\gamma_{\X}\otimes \id)( U^*_{\Ind_{\msG}(\gamma),12 }(1\otimes \xi) )\, (\eta\otimes 1)\\
&=
(\Delta_{12}^1\otimes \id)( U^*_{\Ind_{\msG}(\gamma),12 }(1\otimes \xi) )\, (\eta\otimes 1)=
U^*_{\Ind_{\msG}(\gamma),12 [23]}
U^*_{\Ind_{\msG}(\gamma),11 [13]}(\eta\otimes 1\otimes \xi)
\end{split}\]
and
\[
(1\otimes I)U_{\Ind_{\X}(\gamma)}^*(\eta\otimes \xi)=
U^*_{\Ind_{\msG}(\gamma),12 [23]}
U^*_{\Ind_{\X}(\gamma) [13]}(\eta\otimes 1\otimes \xi)=
U^*_{\Ind_{\msG}(\gamma),12 [23]}
U^*_{\Ind_{\msG}(\gamma),11 [13]}(\eta\otimes 1\otimes \xi)
\]
which proves the claim. Note that we have used $U_{\Ind_{\msG}(\gamma),11}=U_{\Ind_{\X}(\gamma)}$.  Next we define an isometric map
\[
\widetilde{I} = I \oplus \id\colon \begin{pmatrix} L^2(\Ind_{\X}(A)) \\ L^2(A) \end{pmatrix} \rightarrow \begin{pmatrix} \Ind_{\X}(L^2(A)) \\ L^2(A)\end{pmatrix}
\]
and we claim that $\tilde{I}$ is an intertwiner between $U_{\Ind_{\msG}(\gamma)}$ and $\Ind_{\msG}(U_\gamma)$. Let us check that
\begin{equation}\label{eq6}
\Ind_{\msG}(U_\gamma)_{21}(1\otimes I)=
U_{\Ind_{\msG}(\gamma),21}\colon 
L^2(\Y)\otimes L^2(\Ind_{\X}(A))\rightarrow 
L^2(\Y)\otimes L^2(A)
\end{equation}
or equivalently $(1\otimes I) U_{\Ind_{\msG}(\gamma),21}^*=
\Ind_{\msG}(U_\gamma)_{21}^*$. Take $\eta\in L^2(\Y), \zeta\in L^2(A)$: using $U_{\Ind_{\msG}(\gamma),22}=U_\gamma$, we calculate
\[
(1\otimes I)U^*_{\Ind_{\msG}(\gamma),21}(\eta\otimes \zeta)=
U^*_{\Ind_{\msG}(\gamma),12 [23]}
U^*_{\Ind_{\msG}(\gamma),21 [13]}(\eta\otimes 1\otimes\zeta)
\]
and
\[
\Ind_{\msG}(U_\gamma)^*_{21}(\eta\otimes \zeta)=
(\Delta_{22}^1\otimes\id)(U_\gamma^*)(\eta\otimes 1\otimes \zeta)=
(\Delta_{22}^1\otimes\id)(U_{\Ind_{\msG}(\gamma),22}^*)(\eta\otimes 1\otimes \zeta),
\]
which proves that \eqref{eq6} holds. The remaining equations can be proved in a similar way. Equation \eqref{eq6} implies
\[
U_{\Ind_{\msG}(\gamma),21}(1\otimes I^*)=\Ind_{\msG}(U_\gamma)_{21}
\]
which proves that $I^*$ is isometric and consequently $I$ is a unitary. Alternatively, in the proof of Theorem \ref{TheoIndRep} we have argued that $\oon{Res}\colon \Rep(C^*(\msG))\rightarrow \Rep(C^*(\GG))$ is an equivalence of categories. Since it maps the morphism $\tilde{I}$ to the unitary $\id$, it follows that $\tilde{I}$ must be a unitary as well.

Recall that $\Ind_{\msG}(\gamma)_1^2\colon \Ind_{\X}(A)\rightarrow L^{\infty}(\X)\bar\otimes A$ is the inclusion map, thus
\[
U_{\Ind_{\msG}(\gamma),12}^*(1\otimes \pi_{\Ind_{\X}(A)}(w) )
U_{\Ind_{\msG}(\gamma),12}=
(\id\otimes\pi_A)(w)
\]
for $w\in \Ind_{\X}(A)$ and 
\[\begin{split}
&\quad\;
I\pi_{\Ind_{\X}(A)}(w)I^* (I\xi)=
I \pi_{\Ind_{\X}(A)}(w)\xi=
U^*_{\Ind_{\msG}(\gamma),12} (1\otimes \pi_{\Ind_{\X}(A)}(w)\xi)\\
&=
(\id\otimes \pi_A)(w)
U^*_{\Ind_{\msG}(\gamma),12}(1\otimes \xi)=
(\id\otimes\pi_A)(w) I\xi
\end{split}\]
for $w\in \Ind_{\X}(A),\xi\in L^2(\Ind_{\X}(A))$. Since $I$ is surjective, this proves the second claim.
\end{proof} 

The following lemma will be useful to record for calculations in the presence of a unitary $2$-cocycle. 

\begin{Lem}\label{LemIndLModConj}
Let $(A,\gamma)$ be a $\G$-W$^*$-algebra. Under the identification in Theorem \ref{TheoIndL}, the modular conjugation $J_{\Ind_{\X}(A)}$ of $\Ind_{\X}(A)$ is the unique anti-linear map $\Ind_{\X}(L^2(A))\rightarrow \Ind_{\X}(L^2(A))$ such that 
\begin{multline}\label{EqFormModConj}
(\hat{J}_{11}\otimes J_{\Ind_{\X}(A)})
\bigl(
\hat{\pi}_{12}^1(x)_{[1]}
(\Delta_{22}^1\otimes \id)(\gamma(a)U_{\gamma}^*)_{[123]}
((\hat{J}_{12}\otimes J_A)(\hat{\pi}_{12}^2(z)\otimes 1)\gamma(b)\xi)_{[13]}\bigr) \\
=
\Ind_{\X}(U_{\gamma})
\hat{\pi}_{12}^1(z)_{[1]}
(\Delta_{22}^1\otimes \id)(\gamma(b)U_{\gamma}^*)_{[123]}
((\hat{J}_{12}\otimes J_A)(\hat{\pi}_{12}^2(x)\otimes 1)\gamma(a)U_\gamma^* (\hat{J}_{22}\otimes J_A)\xi)_{[13]}
\end{multline}
for $\xi\in L^2(\G)\otimes L^2(A)$, $x,z\in \hat{N}$ and $a,b\in A$. 
\end{Lem}

\begin{proof}
Applying \eqref{EqStandardMorita} to the W$^*$-Morita equivalence \eqref{EqSplitMatrixCross}, we obtain a composition of unitary identifications 
\begin{equation}\begin{split}\label{eq49}
(\G_{12}\ltimes A) \overline{\otimes}_{\G_{22}\ltimes A} \overline{L^2(\G_{12})\otimes L^2(A)} &\cong (\G_{12}\ltimes A) \overline{\otimes}_{\G_{22}\ltimes A} \overline{L^2(\G_{12}\ltimes A)}
\cong L^2(\G_{11}\ltimes \Ind_{\X}(A)) \\
&\cong L^2(\G_{11})\otimes L^2(\Ind_{\X}(A)) \cong L^2(\G_{11}) \otimes \Ind_{\X}( L^2(A)).
\end{split}\end{equation}
Following the precise prescription of the maps in \eqref{eq49}, we see that this identification becomes
\[\begin{split}
&\quad\;
\hat{p}_1(\pi_{\hat{Q}}(x)\otimes 1)
\Ind_{\msG}(\gamma)(a)\hat{p}_2
 \otimes \ov{
 \hat{p}_1(\pi_{\hat{Q}}(z)\otimes 1)
\Ind_{\msG}(\gamma)(b)\hat{p}_2 \xi
}\\
&\mapsto 
\hat{p}_1(\pi_{\hat{Q}}(x)\otimes 1)
\Ind_{\msG}(\gamma)(a)\hat{p}_2
J_{12}^{\ltimes} 
 \hat{p}_1(\pi_{\hat{Q}}(z)\otimes 1)
\Ind_{\msG}(\gamma)(b)\hat{p}_2 \xi\\
&\mapsto 
(1\otimes U^*_{\Ind_{\msG}(\gamma),12})
\bigl(
\hat{p}_1(\pi_{\hat{Q}}(x)\otimes 1)
\Ind_{\msG}(\gamma)(a)\hat{p}_2
J_{12}^{\ltimes} 
 \hat{p}_1(\pi_{\hat{Q}}(z)\otimes 1)
\Ind_{\msG}(\gamma)(b)\hat{p}_2 \xi
\bigr)_{[13]}
\end{split}\]
for $x,z\in \hat{N},a,b\in A, \xi\in L^2(\G)\otimes L^2(A)$ (the first and the third isomorphisms in \eqref{eq49} are formal identities). Let us remark that we consider $L^2(\Ind_{\X}(A))$ as having one leg in this computation. Using \eqref{EqIsometImpl}, and by tracing to which space particular vectors belong to, we can simplify this expression to
\[\begin{split}
&\quad\;
U^*_{\Ind_{\msG}(\gamma),12 [23]}
\pi_{\hat{Q}}(x)_{[1]}
\Ind_{\msG}(\gamma)_2^1 (a)_{[13]}
U^*_{\Ind_{\msG}(\gamma),21[13]}
\bigl(
(\hat{J}_{12}\otimes J_A)
(\pi_{\hat{Q}}(z)\otimes 1)
\gamma(b) \xi
\bigr)_{[13]}\\
&=
\pi_{\hat{Q}}(x)_{[1]}
(\id\otimes \Ind_{\msG}(\gamma)_1^2)\Ind_{\msG}(\gamma)_2^1 (a)
U^*_{\Ind_{\msG}(\gamma),12 [23]}
U^*_{\Ind_{\msG}(\gamma),21[13]}
\bigl(
(\hat{J}_{12}\otimes J_A)
(\pi_{\hat{Q}}(z)\otimes 1)
\gamma(b) \xi
\bigr)_{[13]}.
\end{split}\]
Next, using that $\Ind_{\msG}(\gamma)^2_1$ is simply the inclusion map $\Ind_{\X}(A)\rightarrow L^{\infty}(\G_{12})\bar\otimes A$, we further obtain from \eqref{eq50} that
\[\begin{split}
&\quad\;
\pi_{\hat{Q}}(x)_{[1]}
\Ind_{\msG}(\gamma)_2^1 (a)
U^*_{\Ind_{\msG}(\gamma),12 [23]}
U^*_{\Ind_{\msG}(\gamma),21[13]}
\bigl(
(\hat{J}_{12}\otimes J_A)
(\pi_{\hat{Q}}(z)\otimes 1)
\gamma(b) \xi
\bigr)_{[13]}\\
&=
\pi_{\hat{Q}}(x)_{[1]}
(\Delta^1_{22}\otimes \id)\gamma(a)
(\Delta_{22}^1\otimes \id)(U^*_{\Ind_{\msG}(\gamma),22})
\bigl(
(\hat{J}_{12}\otimes J_A)
(\pi_{\hat{Q}}(z)\otimes 1)
\gamma(b) \xi
\bigr)_{[13]}\\
&=
\pi_{\hat{Q}}(x)_{[1]}
(\Delta^1_{22}\otimes \id)(\gamma(a)
U^*_{\gamma})
\bigl(
(\hat{J}_{12}\otimes J_A)
(\pi_{\hat{Q}}(z)\otimes 1)
\gamma(b) \xi
\bigr)_{[13]}.
\end{split}\]
Using \eqref{EqIsometImpl} and Theorem \ref{TheoIndL}, we see that under the isomorphisms \eqref{eq49}, the modular conjugation \eqref{EqMoritaModCon} is mapped to $\Ind_{\X}(U_{\gamma})^*(\hat{J}_{11}\otimes J_{\Ind_{\X}(A)})$. This gives \eqref{EqFormModConj}.
\end{proof}

\section{Application to cocycle twists}\label{SecCocyc}

\subsection{Unitary \texorpdfstring{$2$}{2}-cocycles}

Assume again that $\G = (M,\Delta)$ is a LCQG, and assume that $\msG$ is a linking quantum groupoid with $\G$ as its $22$ corner. Then $\msG$ can be constructed from a dual unitary $2$-cocycle $\hat{\Omega}$ for $\G$ as in Section \ref{SubSecCocBich}, if and only if the underlying linking W$^*$-algebra $W^*(\msG)$ is trivialisable, in the sense that there exists an isomorphism of W$^*$-algebras (with projection) 
\begin{equation}\label{EqTrivialisationCoc}
\hat{Q} = \begin{pmatrix} \hat{P} & \hat{N}\\ \hat{O} & \hat{M}\end{pmatrix} \cong \begin{pmatrix} \hat{M}& \hat{M}\\ \hat{M}&\hat{M}\end{pmatrix}.
\end{equation}
Indeed, note first that if the isomorphism \eqref{EqTrivialisationCoc} holds, then we can choose it to be an identity on the lower right hand corner. Then the associated unitary $2$-cocycle is given by $\hat{\Omega} := \hat{\Delta}_{12}(1_{\hat{M}})$. Any other choice of trivialisation \eqref{EqTrivialisationCoc} will lead to a cohomologous $2$-cocycle $\hat{\Omega}'$, meaning that there exists a unitary $u\in \hat{M}$ such that 
\begin{equation}\label{eq7}
(u\otimes u)\hat{\Omega}' = \hat{\Omega}\Delta_{\hat{M}}(u). 
\end{equation}
We then refer to $u$ as a \emph{unitary coboundary} between $\hat{\Omega}$ and $\hat{\Omega}'$. 

Fix in the following a dual unitary $2$-cocycle $\hat{\Omega}$ for $\G$, with $W^*(\msG) = \begin{pmatrix} W^*(\G) & W^*(\G)\\ W^*(\G)& W^*(\G)\end{pmatrix}$ as the associated linking quantum groupoid $\msG$. Then we can also identify
\[
L^2(\msG)\cong \begin{pmatrix} L^2(\G)& L^2(\G)\\L^2(\G)&L^2(\G)\end{pmatrix}, 
\] 
where on the right hand side we have the direct sum of Hilbert spaces, with the standard structure and modular conjugation given by 
\[
\pi_{\hat{Q}}\begin{pmatrix} x & y \\ w & z\end{pmatrix}
\begin{pmatrix}
\xi & \eta \\ \zeta & \vartheta
\end{pmatrix}
 = \begin{pmatrix} \pi_{\hat{M}}(x)\xi + \pi_{\hat{M}}(y) \zeta & \pi_{\hat{M}}(x)\eta +\pi_{\hat{M}}(y)\vartheta\\ \pi_{\hat{M}}(w)\xi+\pi_{\hat{M}}(z)\zeta & \pi_{\hat{M}}(w)\eta +\pi_{\hat{M}}(z)\vartheta\end{pmatrix},\qquad J_{\hat{Q}}\begin{pmatrix} \xi & \eta\\ \zeta & \vartheta\end{pmatrix} = \begin{pmatrix} J_{\hat{M}}\xi & J_{\hat{M}}\zeta\\ J_{\hat{M}}\eta & J_{\hat{M}}\vartheta\end{pmatrix}. 
\]
Through the natural GNS-representation of $L^{\infty}(\msG)$ on $L^2(\msG)$ via the dual weight construction, we also find that $L^2(\G)$ acts as the standard Hilbert space for the (generally distinct) W$^*$-algebras $L^{\infty}(\G_{ij})$. 

A crucial role is played by the modular conjugation $J_N$ for $N = L^{\infty}(\X) = L^{\infty}(\G_{12})$ under this construction. Indeed, its knowledge is sufficient to express concretely the multiplicative partial isometry $\ww_{\hat{Q}}$ of $\hat{Q}$: putting again $\hat{\ww}= \ww_{\hat{M}}$, it follows from \cite{DC11b}*{Proposition 6.5} and its proof that
\begin{equation}\label{EqMultUnitCocy1}
\hat{\ww}^k_{22} = \hat{\ww},\quad \hat{\ww}^k_{12} = \hat{\ww}\hat{\Omega}^*,
\end{equation}
\begin{equation}\label{EqMultUnitCocy2}
\hat{\ww}^k_{21} = (J_N\otimes J_{\hat{M}})\hat{\Omega} \hat{\ww}^*(J_M\otimes J_{\hat{M}}),\quad \hat{\ww}^k_{11} = (J_N\otimes J_{\hat{M}})\hat{\Omega} \hat{\ww}^*(J_M\otimes J_{\hat{M}})\hat{\Omega}^*
\end{equation}
for $k\in \{1,2\}$ (see equation \eqref{EqDecompDualW} for the definition of these unitaries. Note also that the above formulas do not depend on $k$).

We can write these last formulas a bit differently as follows. First note that associated to $\hat{\Omega}$ is the \emph{conjugate unitary $2$-cocycle}
\[
\hat{\Omega}_c :=  (\hat{R}\otimes \hat{R})(\hat{\Omega}_{[21]}^*) \in \hat{M}\bar{\otimes}\hat{M}. 
\]
If we write 
\[
X_{\hat{\Omega}} := J_NJ_M,
\]
then it can be shown \cite{DC11b}*{Proposition 6.3} that $X_{\hat{\Omega}}\in \hat{M}$, and that $X_{\hat{\Omega}}$ is a unitary coboundary between $\hat{\Omega}$ and $\hat{\Omega}_c$,
\begin{equation}
(X_{\hat{\Omega}}\otimes X_{\hat{\Omega}})\hat{\Omega}_c = \hat{\Omega}\Delta_{\hat{M}}(X_{\hat{\Omega}}). 
\end{equation}
If we now recall the fundamental unitary $u_{\G} = \nu_{\G}^{i/8}J_{M}J_{\hat{M}}$, then by using \eqref{EqModConjMultUn} and \eqref{EqFormUnitaryAntipo} the formulas in \eqref{EqMultUnitCocy2} can be rewritten as 
\[
\hat{\ww}^k_{21} = (X_{\hat{\Omega}} \otimes u_{\G})\hat{\Omega}_{c[21]}(1\otimes u_{\G})\hat{\ww},\qquad \hat{\ww}^k_{11} = (X_{\hat{\Omega}} \otimes u_{\G})\hat{\Omega}_{c[21]}(1\otimes u_{\G})\hat{\ww}\hat{\Omega}^*. 
\]

Let us end with a more conceptual description of the element $X_{\hat{\Omega}}$. Recall from \eqref{EqUnitaryAntipodeQgr} the unitary antipode on $W^*(\msG)$. Then we see that in fact $X_{\hat{\Omega}}^* = \hat{R}_{12}(1_{\hat{M}})$. From \eqref{EqUnitAntipodQGrProp}, we can then also recover the other components of this unitary antipode using just the unitary antipode $\hat{R}$ of $(W^*(\G),\Delta_{\hat{M}})$ and the element $X_{\hat{\Omega}}$:
\begin{equation}\label{eq12}
\hat{R}_{22}(x)= \hat{R}(x),\quad \hat{R}_{12}(x) =\hat{R}(x) X_{\hat{\Omega}}^*,\quad \hat{R}_{21}(x) = X_{\hat{\Omega}} \hat{R}(x),\quad \hat{R}_{11}(x) = X_{\hat{\Omega}}\hat{R}(x)X_{\hat{\Omega}}^*,\qquad \forall x\in W^*(\G). 
\end{equation}

\subsection{Universally continuous \texorpdfstring{$2$}{2}-cocycles}

Having the linking W$^*$-algebra $\hat{Q} = W^*(\msG)$ of a linking quantum groupoid $\msG = \begin{pmatrix} \Hh & \X\\ \Y & \G\end{pmatrix}$ trivial, does not mean that also the associated universal (or even reduced) linking C$^*$-algebra $\hat{Q}_u = C^*(\msG)$ is trivial (see e.g.\ \cite{DC10}). We therefore introduce the following definition.

\begin{Def}\label{DefUniversalLift}
We say that a linking quantum groupoid $\msG= \begin{pmatrix} \Hh & \X\\ \Y & \G\end{pmatrix}$ has \emph{trivialisable universal linking C$^*$-algebra} if there exists an isomorphism of C$^*$-algebras
\begin{equation}\label{EqTrivialisationCstar}
C^*(\msG)=\begin{pmatrix} \hat{P}_u & \hat{N}_u\\ \hat{O}_u & \hat{M}_u\end{pmatrix} \cong \begin{pmatrix} \hat{M}_u& \hat{M}_u\\ \hat{M}_u&\hat{M}_u\end{pmatrix},
\end{equation}
in a manner preserving the matrix decomposition. 
\end{Def}

Assume now that $\msG$ is a linking quantum groupoid with trivialisable universal linking C$^*$-algebra.
\begin{enumerate}
\item 
By composing with an automorphism of $\begin{pmatrix} \hat{M}_u& \hat{M}_u\\ \hat{M}_u&\hat{M}_u\end{pmatrix}$ of the form $\theta\otimes \id_{M_2}$, we can assume that the isomorphism $\begin{pmatrix} \hat{P}_u& \hat{N}_u\\ \hat{O}_u&\hat{M}_u\end{pmatrix}\cong \begin{pmatrix} \hat{M}_u& \hat{M}_u\\ \hat{M}_u&\hat{M}_u\end{pmatrix}$ preserves the matrix structure, and is the identity on the lower right hand corner.
\item The counit \eqref{EqUnivCounit} induces a non-degenerate $*$-homomorphism
\[
\begin{pmatrix} \hat{\varepsilon}_{11} & \hat{\varepsilon}_{12} \\ \hat{\varepsilon}_{21} & \hat{\varepsilon}_{22}\end{pmatrix}\colon \begin{pmatrix}\hat{M}_u & \hat{M}_u \\ \hat{M}_u & \hat{M}_u \end{pmatrix} \rightarrow \begin{pmatrix} \C&\C\\ \C & \C\end{pmatrix} 
\]
satisfying $\hat{\varepsilon}_{22} = \hat{\eps}$, where $\hat{\varepsilon}$ is the counit of $(\hat{M}_u,\hat{\Delta}_u)$. Write $\mu=\hat{\varepsilon}_{12}(1_{\hat{M}_u})\in \mathbb{T}$. By composing the isomorphism \eqref{EqTrivialisationCstar} with the automorphism $\begin{pmatrix}
a & b \\ c & d 
\end{pmatrix}\mapsto \begin{pmatrix}
\mu a \ov{\mu} & \mu b \\  c\ov{\mu} & d
\end{pmatrix}$ of $\begin{pmatrix} \hat{M}_u& \hat{M}_u\\ \hat{M}_u&\hat{M}_u\end{pmatrix}$, we can choose the trivialisation \eqref{EqTrivialisationCstar} in such a way that $\hat{\varepsilon}_{12}(1_{\hat{M}_u}) = 1$. We obtain $\hat{\varepsilon}_{ij} = \hat{\varepsilon}$. Indeed, we have
\[\begin{split}
\begin{pmatrix}
\hat{\varepsilon}_{11}(a) & 0\\
0 & 0 
\end{pmatrix}&=
\begin{pmatrix} \hat{\varepsilon}_{11} & \hat{\varepsilon}_{12} \\ \hat{\varepsilon}_{21} & \hat{\varepsilon}_{22}\end{pmatrix}
\biggl(
\begin{pmatrix}
0 & 1_{\hat{M}_u} \\
0 & 0
\end{pmatrix}
\begin{pmatrix}
0 & 0 \\
0 & a
\end{pmatrix}
\begin{pmatrix}
0 & 0 \\
1_{\hat{M}_u} & 0
\end{pmatrix}
\biggr)\\
&=
\begin{pmatrix}
0 & 1 \\
0 & 0
\end{pmatrix}
\begin{pmatrix}
0 & 0\\
0 & \hat{\varepsilon}_{22}(a) 
\end{pmatrix}
\begin{pmatrix}
0 & 0 \\
1 & 0
\end{pmatrix}=
\begin{pmatrix}
\hat{\varepsilon}(a) & 0 \\
0 & 0
\end{pmatrix}
\end{split}\]
and the remaining equalities can be checked in a similar way.
\item Assuming that we have made adjustments as in points $(1),(2)$, the reduction map induces via the isomorphism \eqref{EqTrivialisationCstar} a $*$-homomorphism
\begin{equation}\label{eq8}
\pi_{\hat{Q}_u}^{\red}= \begin{pmatrix} \hat{\pi}_{\red,11} & \hat{\pi}_{\red,12} \\ \hat{\pi}_{\red,21} & \hat{\pi}_{\red,22}\end{pmatrix}\colon \begin{pmatrix} \hat{M}_u& \hat{M}_u\\ \hat{M}_u&\hat{M}_u\end{pmatrix} \rightarrow \begin{pmatrix} \hat{M}_{11}& \hat{M}_{12}\\ \hat{M}_{21}&\hat{M}_{22}\end{pmatrix} = \begin{pmatrix} \hat{P} & \hat{N}\\ \hat{O}& \hat{M},\end{pmatrix} 
\end{equation}
which is non-degenerate when interpreted as a a map into $ \Mult(M_2(\C) \otimes \mathcal{K}(L^2(M)))$. Under the isomorphism \eqref{EqTrivialisationCstar}, the element $\begin{pmatrix} 0 & 1_{\hat{M}_u}\\ 0 & 0 \end{pmatrix}$ can be seen as living in $ \Mult(\hat{Q}_u)$, and we can define an element
\[
u = \hat{\pi}_{\red,12}(1_{\hat{M}_u}) \in \hat{M}_{12}
\]
which satisfies
\[
u^*u = 1_{\hat{M}},\qquad uu^* = 1_{\hat{P}}. 
\]
Using this element, we obtain an isomorphism 
\[
\begin{pmatrix} \hat{M}_{11}& \hat{M}_{12}\\ \hat{M}_{21}&\hat{M}_{22}\end{pmatrix}\ni \begin{pmatrix}
a & b \\ 
c & d
\end{pmatrix}\mapsto 
\begin{pmatrix}
u^* a u & u^*b \\
 c u& d
\end{pmatrix}
\in 
\begin{pmatrix} \hat{M}& \hat{M}\\ \hat{M}&\hat{M}\end{pmatrix}.
\] 
Composing with this isomorphism, the map \eqref{eq8} trivializes to 
\[
\begin{pmatrix} \hat{\pi}_{\red} & \hat{\pi}_{\red} \\ \hat{\pi}_{\red} & \hat{\pi}_{\red}\end{pmatrix}\colon \begin{pmatrix} \hat{M}_u& \hat{M}_u\\ \hat{M}_u&\hat{M}_u\end{pmatrix} \rightarrow \begin{pmatrix} \hat{M}& \hat{M}\\ \hat{M}&\hat{M}\end{pmatrix}.
\]
\item By transport of structure we obtain coproducts on $\begin{pmatrix} \hat{M}_u& \hat{M}_u\\ \hat{M}_u&\hat{M}_u\end{pmatrix}$ and $\begin{pmatrix} \hat{M}& \hat{M}\\ \hat{M}&\hat{M}\end{pmatrix}$.
\end{enumerate}

\begin{Def}\label{def3}
We say that a linking quantum groupoid $\msG= \begin{pmatrix} \Hh & \X\\ \Y & \G\end{pmatrix}$ has \emph{trivial universal linking C$^*$-algebra} if $
W^*(\msG)=\begin{pmatrix} \hat{M}& \hat{M}\\ \hat{M}&\hat{M}\end{pmatrix}$ and $ C^*(\msG)=\begin{pmatrix}
\hat{P}_u & \hat{N}_u \\ \hat{O}_u & \hat{M}_u
\end{pmatrix}\cong \begin{pmatrix} \hat{M}_u& \hat{M}_u\\ \hat{M}_u&\hat{M}_u\end{pmatrix}$ under a fixed isomorphism which preserves the matrix structure\footnote{We will neglect writing these isomorphisms explicitely.}, is identity on the lower right hand corner, in which $ \hat{\varepsilon}_{ij}\cong\varepsilon_{\hat{M}_u}$
 and under which the reducing map corresponds to $\begin{pmatrix}
\hat{\pi}_{\red} & \hat{\pi}_{\red} \\
\hat{\pi}_{\red} & \hat{\pi}_{\red}
\end{pmatrix}$.
\end{Def}

The discussion after Definition \ref{DefUniversalLift} shows that any linking quantum groupoid with trivialisable universal linking C$^*$-algebra leads to an isomorphic linking quantum groupoid, with trivial universal linking C$^*$-algebra.

Let $\msG= \begin{pmatrix} \Hh & \X\\ \Y & \G\end{pmatrix}$ be a linking quantum groupoid with trivial universal linking C$^*$-algebra. Then since $W^*(\G_{ij})=\hat{M}_{ij}=\hat{M}$, we also have $L^2(\G_{ij})=L^2(\G)$. Define the unitary
\[
\hat{\Omega}_u = \hat{\Delta}_{u,12}(1_{\hat{M}_u})\in \Mult(\hat{M}_u \otimes \hat{M}_u),
\]
and call it the associated \emph{universal dual unitary $2$-cocycle} for $\G$. We then also write 
\begin{equation}\label{eq10}
\hat{\Omega} = (\hat{\pi}_{\red}\otimes \hat{\pi}_{\red})(\hat{\Omega}_u) \in \Mult(C_{\red}^*(\G) \otimes C_{\red}^*(\G))\subseteq \hat{M}\bar{\otimes}\hat{M}
\end{equation}
for the associated W$^*$-algebraic dual unitary $2$-cocycle of $\G$. Note that
\[
(\hat{\varepsilon}\otimes \id)\hat{\Omega}_u = (\id\otimes \hat{\varepsilon})\hat{\Omega}_u = 1_{\hat{M}_u}. 
\] 

\begin{Def}\label{DefUnivLift2Coc}
We say that a unitary $2$-cocycle $\hat{\Omega} \in \hat{M}\bar{\otimes}\hat{M}$ \emph{admits a universal lift} if it arises in the above way. The corresponding unitary $2$-cocycle $\hat{\Omega}_u=\hat{\Delta}_{u,12}(1_{\hat{M}_u})$ is called a \emph{universal lift} of $\hat{\Omega}$.
\end{Def}

\begin{Rem}
Let us stress that the requirement that $\hat{\Omega}$ admits a universal lift in the sense of Definition \ref{DefUnivLift2Coc} is in principle stronger than merely requiring the existence of a unitary $2$-cocycle $\hat{\Omega}_u$ satisfying $(\hat{\pi}_{\red}\otimes \hat{\pi}_{\red})(\hat{\Omega}_u)=\hat{\Omega}$. Indeed, in the latter case it is not clear if twisting the amplified comultiplication of $M_2(C^*(\G))$ with $\hat{\Omega}_u$ will lead to a copy of $(C^*(\msG),\hat{\Delta}_u)$. We do not know however of a concrete example illustrating this phenomenon. 
\end{Rem}

\begin{Rem}\label{rem2}
A universal lift will never be unique unless $\hat{\G}$ is coamenable, i.e.\ $C^*(\G)\cong C^*_{\red}(\G)$. Indeed, if $\hat{\G}$ is not coamenable, we can take any non-trivial unitary $v\in \Mult(C^*(\G))$ such that $\hat{\pi}_{\red}(v) = 1_{\hat{M}}$ and $\hat{\eps}(v)=1$, and then 
\[
(v^*\otimes v^*)\hat{\Omega}_u\hat{\Delta}_u(v)
\] 
is a universal lift of $\hat{\Omega}$ as well. However, it is clear that this is the only thing that can go wrong: if $C$ is a C$^*$-algebra, then any isomorphism 
\[
\begin{pmatrix} C & C \\C &C\end{pmatrix} \rightarrow \begin{pmatrix} C & C \\C &C\end{pmatrix}
\]
that is the identity in the lower hand corner must be of the form 
\[
\begin{pmatrix} a & b \\ c &  d \end{pmatrix} \mapsto \begin{pmatrix} vav^* & vb \\ cv^* & d\end{pmatrix}
\]
for some unitary $v\in \Mult(C)$. In the case of a linking quantum groupoid this will lead to the coboundary between the respective two universal lifts. 
\end{Rem}

Recall from Lemma \ref{lemma9} that, for any linking quantum groupoid, we have fixed a natural isomorphism $\Ind_{\X}^{C^*}\rightarrow \widetilde{\Ind}_{\X}$. Since the functor $\Ind_{\X}$ is monoidal via unitaries $u_{\mc{H},\mc{H}'}$ by Theorem \ref{TheoIndRep}, so is the functor $\widetilde{\Ind}_{\X}$, with the same family of unitaries. We hence obtain a monoidal structure on $\Ind_{\X}^{C^*}$ which we denote by $v_{\mc{H},\mc{H}'}$:
\begin{equation}\label{eq36}\begin{split}
v_{\mc{H},\mc{H}'}=&\theta_{(\mc{H}\otimes\mc{H}',\hat{\pi}\ast\hat{\pi}')}^*
u_{\mc{H},\mc{H}'}
(\theta_{(\mc{H},\hat{\pi})}\otimes 
\theta_{(\mc{H}',\hat{\pi}')})\\
&\colon 
(C^*(\X)\otimes_{C^*(\GG)}\mc{H} )\otimes 
(C^*(\X)\otimes_{C^*(\GG)}\mc{H}' )
\rightarrow
C^*(\X)\otimes_{C^*(\GG)}(\mc{H}\otimes\mc{H}').
\end{split}
\end{equation}
for $(\mc{H},\hat{\pi}),(\mc{H}',\hat{\pi}')\in \Rep(C^*(\GG))$.

Fix now a linking quantum groupoid as above with trivial universal linking  C$^*$-algebra and the associated universal unitary $2$-cocycle $\hat{\Omega}_u$. Our next goal is to prove that we can modify the induction functor $\Ind_{\X}^{C^*}$ to the identity one, but with non-trivial monoidal coherence maps (see Lemma \ref{lemma10} and Proposition \ref{prop1}). 

Define the identity functor 
\[
\Ind_{\X}^{C^*,\triv} = \id\colon \Rep(C^*(\G)) \rightarrow \Rep(C^*(\Hh)),
\]
implicitly using the isomorphism $C^*(\GG)\cong C^*(\HH)$, with non-trivial monoidal coherence maps given by 
\begin{equation}\label{EqMonoidalCoherence}
(\hat{\pi}\otimes \hat{\pi}')\hat{\Omega}_{u [21]}^*\colon \Hsp\otimes_{\Hh}\Hsp'\rightarrow \Hsp\otimes_{\G}\Hsp',\quad\quad\forall (\mc{H},\hat{\pi}),(\mc{H}',\hat{\pi}')\in \Rep(C^*(\GG)).
\end{equation}

It follows from Proposition \ref{prop1} below that the above family in fact gives a tensor structure on $\Ind^{C^*,\triv}_{\X}$. 

First, we create a concrete link between $\Ind_{\X}^{C^*}$ and  $\Ind_{\X}^{C^*,\triv}$. The proof of the next lemma is an easy exercise.

\begin{Lem}\label{lemma10}
For $(\mc{H},\hat{\pi})\in \Rep(C^*(\G))$, define a unitary map
\[
\mc{O}_{(\mc{H},\hat{\pi})}\colon C^*(\X)\otimes_{C^*(\G)}\mc{H}\in a\otimes\xi\mapsto \hat{\pi}(a)\xi \in\mc{H}.
\]
Then the family $(\mc{O}_{(\mc{H},\hat{\pi})})_{(\mc{H},\hat{\pi})\in \Rep(C^*(\G))}$ forms a natural isomorphism $\Ind_{\X}^{C^*}\rightarrow \Ind_{\X}^{C^*,\triv}$.
\end{Lem}

Next, we need an auxilliary result.

\begin{Lem}\label{lemma11}
Let $(\mc{H},U),(\mc{H}',U'),(\mc{H}\otimes\mc{H}',U\otop U')\in \Rep(\GG)$ with the associated $*$-representations $\hat{\pi},\hat{\pi}'$, $\hat{\pi}\ast \hat{\pi}'$, and let $U_{12}, U'_{12}, (U\otop U')_{12}$ be the $12$ corners of the induced representations of $\msG$, with associated maps $\hat{\pi}_{12},\hat{\pi}'_{12},(\hat{\pi}\ast\hat{\pi}')_{12}$. We have
\[
(\hat{\pi}\ast \hat{\pi}')_{12}(a)=
u_{\mc{H},\mc{H}'}(\hat{\pi}_{12}\otimes \hat{\pi}'_{12})\hat{\Delta}_{u,12}^{op}(a),\qquad \forall a\in C^*(\X).
\]
\end{Lem}

\begin{proof}
Recall that $U_{12}^*\in L^{\infty}(\GG_{12})\bar\otimes \mcB( \Ind_{\X}(\mc{H}),\mc{H})$ satisfies
\[
U_{12}^*\colon L^2(\GG_{12})\otimes \Ind_{\X}(\mc{H})\ni \xi\otimes z\mapsto z\xi \in L^2(\G_{12})\otimes \mc{H}
\]
and similarly for $U'_{12},(U\otop U')_{12}$. We claim that
\begin{equation}\label{eq37}
(U\otop U')_{12}^* (1\otimes u_{\mc{H},\mc{H}'})= 
U'^*_{[13]} U^*_{[12]}.
\end{equation}
Indeed, we have
\[\begin{split}
&\quad\;
(U\otop U')_{12}^* (1\otimes u_{\mc{H},\mc{H}'})(\xi\otimes z \otimes w)=
(U\otop U')_{12}^* (\xi\otimes w_{[13]} z_{[12]})=
w_{[13]}z_{[12]}\xi\\
&=
U'^*_{[13]}( z_{[12]}\xi\otimes w)=
U'^*_{[13]}U^*_{[12]} (\xi\otimes z\otimes w). 
\end{split}\]
for $\xi\in L^2(\G_{12}),z\in \Ind_{\X}(\mc{H}),w\in \Ind_{\X}(\mc{H}')$. 

Observe that (c.f.~\cite{DC09}*{Lemma 11.1.5})
\[
(\hat{\Delta}_{u,21}\otimes \id)(\hat{\Ww}_{12})=
\hat{\Ww}_{12[13]} \hat{\Ww}_{12[23]}\;\Rightarrow\;
(\id\otimes \hat{\Delta}_{u,12}^{op})(\hat{\Ww}_{12 [21]}^*) = \hat{\Ww}_{12[21]}^*\hat{\Ww}_{12 [31]}^*,
\]
hence, using \eqref{eq37},
\[\begin{split}
&\quad\;
(\id\otimes (\hat{\pi}_{12}\otimes \hat{\pi}'_{12})\hat{\Delta}_{u,12}^{op})(\hat{\Ww}_{12[21]}^*)=
(\id\otimes \hat{\pi}_{12} )(\hat{\Ww}_{12 [21]}^*)_{[12]}
(\id\otimes \hat{\pi}'_{12})(\hat{\Ww}_{12 [21]}^*)_{[13]}=
 U_{12[12]}U'_{12[13]}\\
 &=
(1\otimes u_{\mc{H},\mc{H}'})^*
(U\otop U')_{12}=
 (1\otimes u_{\mc{H},\mc{H}'})^*
(\id\otimes (\hat{\pi}\ast\hat{\pi}')_{12} )(\hat{\Ww}_{12 [21]}^*).
\end{split}\]
The claim now follows by slicing off the first leg.
\end{proof}

Finally, we prove that the natural isomorphism $(\mc{O}_{(\mc{H},\hat{\pi})})_{(\mc{H},\hat{\pi})\in \Rep(C^*(\G))}$ respects the monoidal structures \eqref{eq36} and \eqref{EqMonoidalCoherence}.

\begin{Prop}\label{prop1}
For $(\mc{H},\hat{\pi}),(\mc{H}',\hat{\pi}')\in \Rep(C^*(\GG))$ we have
\begin{equation}\label{eq38}
\mc{O}_{(\mc{H}\otimes \mc{H}', \hat{\pi}\ast\hat{\pi}')} v_{\mc{H},\mc{H}'}
(\mc{O}_{(\mc{H}, \hat{\pi})}\otimes 
\mc{O}_{( \mc{H}', \hat{\pi}')})^*=
(\hat{\pi}\otimes\hat{\pi}')(\hat{\Omega}^*_{u [21]}).
\end{equation}
\end{Prop}

\begin{proof}
Using the definition \eqref{eq36} of $v_{\mc{H},\mc{H}'}$, we see that \eqref{eq38} is equivalent to
\begin{equation}\label{eq39}
(\mc{O}_{(\mc{H},\hat{\pi})}\otimes \mc{O}_{(\mc{H}',\hat{\pi}')})
( \theta_{(\mc{H},\pi)}^*\otimes \theta_{(\mc{H}',\pi')}^*) u_{\mc{H},\mc{H}'}^*  =
 (\hat{\pi}\otimes \hat{\pi}') (\hat{\Omega}_{u [21]})
 \mc{O}_{(\mc{H}\otimes\mc{H}',\hat{\pi}\ast \hat{\pi}')}
  \theta_{(\mc{H}\otimes \mc{H}',\pi\ast \pi')}^*.
\end{equation}
Fix $a\in C^*(\X),b,c\in C^*(\GG),\zeta\in \mc{H},\zeta'\in \mc{H}'$. Then
$(\hat{\pi}\ast\hat{\pi}')_{12}(a) (\hat{\pi}(b)\zeta \otimes \hat{\pi}'(c) \zeta')$ is a typical element of $\Ind_{\X}(\mc{H}\otimes \mc{H}')$, in the sense that such elements span a dense subspace (see Lemma \ref{lemma9}). Then on the one hand, the right hand side of \eqref{eq39} evaluates on this element to
\begin{equation}\begin{split} \label{eq40}
&\quad\;
 (\hat{\pi}\otimes \hat{\pi}') (\hat{\Omega}_{u [21]})
 \mc{O}_{(\mc{H}\otimes \mc{H}',\hat{\pi}\ast\hat{\pi}')}
  \theta_{(\mc{H}\otimes \mc{H}',\pi\ast \pi')}^*\;
(\hat{\pi}\ast\hat{\pi}')_{12}(a) (\hat{\pi}(b)\zeta \otimes \hat{\pi}'(c) \zeta')\\
&=
 (\hat{\pi}\otimes \hat{\pi}') (\hat{\Omega}_{u [21]})
 \mc{O}_{(\mc{H}\otimes \mc{H}',\hat{\pi}\ast\hat{\pi}')}
\bigl( 
a\otimes (\hat{\pi}(b)\zeta \otimes \hat{\pi}'(c) \zeta')\bigr)\\
&=
(\hat{\pi}\otimes\hat{\pi}')(\hat{\Omega}_{u[21]})
(\hat{\pi}\ast\hat{\pi}')(a)  (\hat{\pi}(b)\zeta \otimes \hat{\pi}'(c) \zeta')\\
&=(\hat{\pi}\otimes\hat{\pi}')\hat{\Delta}_{u,12}^{op} (1)
(\hat{\pi}\otimes \hat{\pi}')\hat{\Delta}^{op}_{u,22}(a)
(\hat{\pi}(b)\zeta \otimes \hat{\pi}'(c) \zeta')=
(\hat{\pi}\otimes\hat{\pi}')\hat{\Delta}_{u,12}^{op} (a)
(\hat{\pi}(b)\zeta \otimes \hat{\pi}'(c) \zeta').
\end{split}
\end{equation}
On the other hand, using Lemma \ref{lemma11}, we see that the left hand side of \eqref{eq39} evaluates to
\[\begin{split}
&\quad\;
(\mc{O}_{(\mc{H},\hat{\pi})}
\otimes \mc{O}_{(\mc{H}',\hat{\pi}')})
( \theta_{(\mc{H},\hat{\pi})}^*\otimes \theta_{(\mc{H}',\hat{\pi}')}^*) u_{\mc{H},\mc{H}'}^*  \;
(\hat{\pi}\ast\hat{\pi}')_{12}(a) (\hat{\pi}(b)\zeta \otimes \hat{\pi}'(c) \zeta')\\
&=
(\mc{O}_{(\mc{H},\hat{\pi})}
\otimes \mc{O}_{(\mc{H}',\hat{\pi}')})
( \theta_{(\mc{H},\hat{\pi})}^*\otimes \theta_{(\mc{H}',\hat{\pi}')}^*)
(\hat{\pi}_{12}\otimes \hat{\pi}'_{12})\hat{\Delta}^{op}_{u,12}(a) (\hat{\pi}(b)\zeta \otimes \hat{\pi}'(c) \zeta')\\
&=
(\mc{O}_{(\mc{H},\hat{\pi})}\otimes \mc{O}_{(\mc{H}',\hat{\pi}')})
( \theta_{(\mc{H},\hat{\pi})}^*\otimes \theta_{(\mc{H}',\hat{\pi}')}^*)
(\hat{\pi}_{12}\otimes \hat{\pi}'_{12})(
\hat{\Delta}^{op}_{u,12}(a) (b\otimes c))(
\zeta \otimes  \zeta').
\end{split}\]
We have $\hat{\Delta}^{op}_{u,12}(a) (b\otimes c)\in C^*(\GG)\otimes C^*(\GG)$, hence we can write $\hat{\Delta}^{op}_{u,12}(a) (b\otimes c)=\lim_{n\to\infty}\sum_{i=1}^{N_n} d^{(n)}_{1,i}\otimes d^{(n)}_{2,i}$ in norm, for some $d^{(n)}_{1,i},d^{(n)}_{2,i}\in C^*(\GG)$. As $\hat{\pi}_{12},\hat{\pi}'_{12}$ are restrictions of $*$-homomorphisms, they are continuous and we further have
\begin{equation}\begin{split}\label{eq41}
&\quad\;
(\mc{O}_{(\mc{H},\hat{\pi})}\otimes \mc{O}_{(\mc{H}',\hat{\pi}')})
( \theta_{(\mc{H},\hat{\pi})}^*\otimes \theta_{(\mc{H}',\hat{\pi}')}^*) u_{\mc{H},\mc{H}'}^*  \;
(\hat{\pi}\ast\hat{\pi}')_{12}(a) (\hat{\pi}(b)\zeta \otimes \hat{\pi}'(c) \zeta')\\
&=
\lim_{n\to\infty}\sum_{i=1}^{N_n}
(\mc{O}_{(\mc{H},\hat{\pi})}\otimes \mc{O}_{(\mc{H}',\hat{\pi}')})
( \theta_{(\mc{H},\hat{\pi})}^*\otimes \theta_{(\mc{H}',\hat{\pi}')}^*)  
(\hat{\pi}_{12}(d^{(n)}_{1,i})\zeta \otimes  \hat{\pi}'_{12}(d^{(n)}_{2,i})\zeta')\\
&=
\lim_{n\to\infty}\sum_{i=1}^{N_n}
(\mc{O}_{(\mc{H},\hat{\pi})}\otimes \mc{O}_{(\mc{H}',\hat{\pi}')})
\bigl(
(d^{(n)}_{1,i}\otimes \zeta) \otimes
(d^{(n)}_{2,i}\otimes \zeta')\bigr)=
\lim_{n\to\infty}\sum_{i=1}^{N_n}
\hat{\pi}(d^{(n)}_{1,i})\zeta \otimes
\hat{\pi}'(d^{(n)}_{2,i})\zeta'\\
&=
(\hat{\pi}\otimes\hat{\pi}')(\hat{\Delta}^{op}_{u,12}(a) (b\otimes c))(\zeta\otimes \zeta')=
(\hat{\pi}\otimes\hat{\pi}')\hat{\Delta}^{op}_{u,12}(a)
 (\hat{\pi}(b)\zeta\otimes \hat{\pi}'(c)\zeta').
\end{split}\end{equation}

Equations \eqref{eq40}, \eqref{eq41} show \eqref{eq39} and end the proof.
\end{proof}

Write again $\hat{R}_u\colon C^*(\G) \rightarrow C^*(\G)$ for the lift of the unitary antipode to $C^*(\G)$. Then we can also lift $\hat{\Omega}_c$ to 
\begin{equation}\label{EqConjCocy}
\hat{\Omega}_{u,c} :=  (\hat{R}_u\otimes \hat{R}_u)(\hat{\Omega}_{u[21]}^*) \in \Mult(\hat{M}_u \otimes \hat{M}_u). 
\end{equation}

The associated coboundary at the universal level is defined in an analogous way as the one at the reduced level, see \eqref{eq7}. In fact:

\begin{Lem}\label{LemUniqueUnitCob}
There exists a unique unitary coboundary $X_{\hat{\Omega},u} \in \Mult(C^*(\G))$ between  $\hat{\Omega}_u$ and $\hat{\Omega}_{u,c}$ such that 
\begin{equation}\label{EqImageUnderRed}
\hat{\pi}_{\red}(X_{\hat{\Omega},u}) = X_{\hat{\Omega}}.
\end{equation}
\end{Lem}

\begin{proof}
Recall from \eqref{EqUnitaryAntipodeUnivQgr} the universal unitary antipode $R_{\hat{Q}_u}$ on $\hat{Q}_u$. If we then put \[
X_{\hat{\Omega},u} := \hat{R}_{u,12}(1_{\hat{M}_u})^* =\hat{R}_{u,21}(1_{\hat{M}_u}) \in \Mult(C^*(\G)),
\]
 we find that 
\[\begin{split}
&\quad\;\hat{\Omega}_u\hat{\Delta}_{u,22}(X_{\hat{\Omega},u}) = \hat{\Delta}_{u,12}(\hat{R}_{u,21}(1_{\hat{M}_u}))\\
&= (\hat{R}_{u,21}\otimes \hat{R}_{u,21})\hat{\Delta}^{\opp}_{u,21}(1_{\hat{M}_u})=
(\hat{R}_{u,21}\otimes \hat{R}_{u,21})(\hat{\Omega}^*_{u[21]})\\
&
= (X_{\hat{\Omega},u}\otimes X_{\hat{\Omega},u})(\hat{R}_u\otimes \hat{R}_u)(\hat{\Omega}_{u[21]}^*)
=  (X_{\hat{\Omega},u}\otimes X_{\hat{\Omega},u})\hat{\Omega}_{u,c},
\end{split}\]
i.e.~$X_{\hat{\Omega},u}$ is a coboundary between $\hat{\Omega}_u$ and $\hat{\Omega}_{u,c}$. The formula \eqref{EqImageUnderRed} holds since $\pi^{\red}_{\hat{Q}_u}\colon C^*(\msG) \rightarrow W^*(\msG)$ intertwines unitary antipodes. 

To see that $X_{\hat{\Omega},u}$ is uniquely determined by these conditions, let $\hat{\Ww} \in \Mult( C^*(\G)\otimes \mcK(L^2(\G)))$ be the half-lifted universal left multiplicative unitary for $\hat{\G}$. If then $Y \in \Mult(C^*(\G))$ satisfies the conditions in the lemma, we see from the LCQG-analogon of \eqref{EqHalfUniversalRepQgr} that 
\[
(X_{\hat{\Omega},u}\otimes X_{\hat{\Omega}})(\id\otimes \hat{\pi}_{\red})(\hat{\Omega}_{u,c})=
(\id\otimes \hat{\pi}_{\red})(\hat{\Omega}_u)\hat{\Ww}^* (1\otimes X_{\hat{\Omega}})\hat{\Ww} = (Y\otimes X_{\hat{\Omega}})(\id\otimes \hat{\pi}_{\red})(\hat{\Omega}_{u,c}),
\]
showing that necessarily $Y = X_{\hat{\Omega},u}$. 
\end{proof}

We let
\begin{equation}\label{eq11}
\Ind_{\msG}^{C^*,\triv}\colon \Rep(C^*(\G))\rightarrow \Rep(C^*(\msG))
\end{equation}
be the functor given on objects by the extension of $\hat{\pi}_U\in \Rep(C^*(\G))$ to a non-degenerate $*$-representation of $C^*(\msG)$ on $\begin{pmatrix} \Hsp \\ \Hsp \end{pmatrix}$ according to the trivialisation $C^*(\msG)\cong \begin{pmatrix}
\hat{M}_u & \hat{M}_u\\
\hat{M}_u & \hat{M}_u
\end{pmatrix}$. On morphisms, $\Ind_{\msG}^{C^*,\triv}$ is given by $\Ind_{\msG}^{C^*,\triv}(T)=T\oplus T$. We let $\Ind_{\msG}^{\triv}$ be the associated functor $\Rep(\G)\rightarrow \Rep(\msG)$.

\begin{Prop}\label{PropInducedCorep}
Consider $(\Hsp,U)$ a unitary $\G$-representation, with associated $C^*(\G)$-representation $\hat{\pi}_U$. The corresponding unitary left $\msG$-representation $\Ind_{\msG}^{\triv}(U)$ is given by the unitaries 
\begin{eqnarray}
U_{22}^* &=& U^*,\\
\label{EqU12Form} U_{12}^* &=&  U^*(\hat{\pi}_{\red}\otimes \hat{\pi}_U)(\hat{\Omega}_{u[21]}^*),\\
U_{21}^* &=& (u_{\G}\hat{\pi}_{\red}( - )u_{\G}^*\otimes \hat{\pi}_U)((1\otimes X_{\hat{\Omega},u})\hat{\Omega}_{u,c})U^*,\\
 U_{11}^* &=& (u_{\G}\hat{\pi}_{\red}( - )u_{\G}^*\otimes \hat{\pi}_U)((1\otimes X_{\hat{\Omega},u})\hat{\Omega}_{u,c})U^*(\hat{\pi}_{\red}\otimes \hat{\pi}_U)(\hat{\Omega}_{u[21]}^*).
\end{eqnarray}
\end{Prop} 

\begin{proof}
Recall the notations \eqref{EqHalfUniversalRepQgr} and \eqref{EqPiecesUniversal} for the dual half-lifted universal multiplicative unitary for $\msG$ and its components. In our specific case it then follows that we have 
\[
\hat{\Ww}_{22}(\id\otimes \hat{\pi}_{\red})(\hat{\Omega}_u^*) = \hat{\Ww}_{22}(\id\otimes \hat{\pi}_{\red})\hat{\Delta}_{u,21}(1_{\hat{M}_u}) = \hat{\Ww}_{12}.
\]
Similarly, using that $(R_{\hat{Q}_u} \otimes R_Q)\Ww_{\hat{Q}_u} = \Ww_{\hat{Q}_u}$, where $R_Q(x)= J_{\hat{Q}}x^*J_{\hat{Q}}$ for $x\in Q$ and the universal analogon of \eqref{eq12}, we obtain  
\[\begin{split}
&\quad\;
\hat{\Ww}_{21} = (\hat{R}_{u,21} \otimes \hat{J}( - )^*\hat{J})\hat{\Ww}_{12} = (X_{\hat{\Omega},u}\otimes 1)(\hat{R}_{u,22} \otimes \hat{J}( - )^*\hat{J})(\hat{\Ww}_{22}(\id\otimes \hat{\pi}_{\red})(\hat{\Omega}_u^*))\\
&=  (X_{\hat{\Omega},u}\otimes 1)(\id\otimes u_{\G}\hat{\pi}_{\red}( - )u_{\G}^*)((\hat{R}_{u,22}\otimes \hat{R}_{u,22})\hat{\Omega}_u^*)\hat{\Ww}_{22} = (X_{\hat{\Omega},u}\otimes 1)(\id\otimes u_{\G}\hat{\pi}_{\red}( - )u_{\G}^*)(\hat{\Omega}_{u,c[21]})\hat{\Ww}_{22}.
\end{split}\]
The formula 
\[
\hat{\Ww}_{11} = (X_{\hat{\Omega},u}\otimes 1)(\id\otimes u_{\G}\hat{\pi}_{\red}( - )u_{\G}^*)(\hat{\Omega}_{u,c[21]})\hat{\Ww}_{22} (\id\otimes \hat{\pi}_{\red})(\hat{\Omega}_u^*)
\]
follows completely similarly. The proposition now follows from \eqref{EqUnivPropUnivLift}.
\end{proof}

We can now also show the triviality of induction on the level of induced $\G$-representations. Let us write $\Ind_{\X}^{\triv}$ for the W$^*$-functor 
\[
\Ind_{\X}^{\triv}\colon \Rep(\GG)\rightarrow \Rep(\HH),\qquad  \Rep(\G)\ni (\Hsp,U) \mapsto (\Hsp,\Ind_{\msG}^{\triv}(U)_{11}) \in \Rep(\Hh),
\] 
with monoidal coherence through \eqref{EqMonoidalCoherence}. On morphisms $\Ind_{\X}^{\triv}$ is defined as the identity. Observe that $\Ind_{\X}^{\triv}$ corresponds to $\Ind_{\X}^{C^*,\triv}$ under the identifications $\Rep(\GG)\cong \Rep(C^*(\GG))$, $\Rep(\HH)\cong \Rep(C^*(\HH))$.

\begin{Prop}\label{PropIdentStandard}
We have a monoidal unitary natural transformation of tensor W$^*$-functors $\Ind_{\X}^{\triv} \cong \Ind_{\X}$ via 
\begin{equation}\label{EqInductionTriv}
\theta_{(\mc{H},\hat{\pi}_U)}\mc{O}^*_{(\mc{H},\hat{\pi}_U)}\colon  \Hsp \cong \Ind_{\X}(\Hsp),\qquad \xi  \mapsto U^*(\hat{\pi}_{\red}\otimes \hat{\pi}_U)(\hat{\Omega}_{u[21]}^*)(1\otimes \xi)
\end{equation}
for $(\mc{H},U)\in \Rep(\GG)$.
\end{Prop} 

\begin{proof}
We already know that the family of unitary operators $\theta_{(\mc{H},\hat{\pi}_U)}\mc{O}^*_{(\mc{H},\hat{\pi}_U)}$ establishes a monoidal unitary natural transformation $\Ind_{\X}^{\triv}\rightarrow \Ind_{\X}$ (see Lemma \ref{lemma9}, the discussion after Remark \ref{rem2}, Lemma \ref{lemma10} and  Proposition \ref{prop1}).

It is left to check the expression \eqref{EqInductionTriv}. Take $a=(\id\otimes\omega_{\eta,\eta'})(\hat{\Ww}_{21})\in C^*(\X)$ for $\eta,\eta'\in L^2(\GG)$ and $\xi\in \mc{H}$. Denote by $\hat{\pi}_{U,ij}$ the $*$-representation of $C^*(\GG_{ij})$ corresponding to $U_{ij}=\Ind_{\msG}(U)_{ij}$. We then calculate, using equations \eqref{eq44} and \eqref{EqU12Form},
\[\begin{split}
&\quad\;
\theta_{(\mc{H},\hat{\pi}_U)}\mc{O}^*_{(\mc{H},\hat{\pi}_U)} (\hat{\pi}_U(a) \xi)=
\theta_{(\mc{H},\hat{\pi}_U)}(a\otimes \xi)=
\hat{\pi}_{U,12}(a)\xi=
\hat{\pi}_{U,12}((\omega_{\eta,\eta'}\otimes\id)(\hat{\Ww}_{21 [21]}))\xi\\
&=
(\omega_{\eta,\eta'}\otimes\id)
\bigl( (\id\otimes\hat{\pi}_{U,21})(\hat{\Ww}_{21[21]}^*)^*\bigr)
\xi=
(\omega_{\eta,\eta'}\otimes\id)(U_{21}^*)\xi=
( \eta^* \otimes \id)(U_{21}^*(\eta'\otimes \xi))\\
&=
( \eta^* \otimes \id)(
(\Delta^1_{22}\otimes \id)(U^*)
(\eta'\otimes 1 \otimes \xi ))=
( \eta^* \otimes \id)(
U_{12 [23]}^*U_{21 [13]}^* 
(\eta'\otimes 1 \otimes \xi ))\\
&=
(\id\otimes \hat{\pi}_{U, 21})(\hat{\Ww}_{12 [21]})
(1\otimes \hat{\pi}_{U, 12}( (\omega_{\eta,\eta'}\otimes\id)\hat{\Ww}_{21[21]})\xi)=
(\id\otimes \hat{\pi}_{U})(\hat{\Ww}_{12 [21]})
(1\otimes \hat{\pi}_{U}(a)\xi)\\
&=
U^* (\hat{\pi}_{\red}\otimes \hat{\pi}_U)(\hat{\Omega}^*_{u[21]})
(1\otimes \hat{\pi}_{U}(a)\xi)
\end{split}\]
Since $\hat{\pi}_U$ is non-degenerate, this ends the proof. 
\end{proof}

\subsection{Modular conjugation for induced actions}
Assume again that $\msG = \begin{pmatrix} \Hh & \X\\ \Y & \G\end{pmatrix}$ has trivial universal linking C$^*$-algebra, with the associated universal dual unitary $2$-cocycle $\hat{\Omega}_u$ for $\G$. 

Assume that $(A,\gamma)$ is a $\G$-W$^*$-algebra. By Theorem \ref{TheoIndL} and the identification \eqref{EqInductionTriv}, we can canonically identify $L^2(A)$ as a standard Hilbert space for $\Ind_{\X}(A)$.

\begin{Theorem}\label{TheoModConjGen}
The modular conjugation map for the above standard representation of $\Ind_{\X}(A)$ on $L^2(A)$ is given by 
\begin{equation}\label{EqInducedModCon}
J_{\Ind_{\X}(A)}\colon L^2(A) \rightarrow L^2(A),\qquad J_{\Ind_{\X}(A)} = \hat{\pi}_{U_{\gamma}}(X_{\hat{\Omega},u})J_A.
\end{equation}
\end{Theorem} 

\begin{proof}
We will write $U_{\gamma,ij}$ for the components of the trivial induction $\Ind_{\msG}^{\triv}(U_\gamma)$ and use its description from Proposition \ref{PropInducedCorep}.

Take $\xi\in L^2(\G)\otimes L^2(A)$, $x,z\in \hat{M}$ and $a,b\in A$ and observe that 
\[
(\Delta^1_{22}\otimes\id)(\gamma(a)U_\gamma^*)=
U_{\gamma, 12 [23]}^* U_{\gamma, 21 [13]}^*
a_{[3]}.
\] 
Lemma \ref{LemIndLModConj} gives us a formula involving the action of $J_{\Ind_{\X}(A)}$ on $\Ind_{\X}(L^2(A))\cong_{I} L^2(\Ind_{\X}(A))$:
\[\begin{split}
&\quad\;
(J_{\hat{M}}\otimes J_{\Ind_{\X}(A)})
\bigl(
\hat{\pi}_{12}^1(x)_{[1]}
U_{\gamma [23]}^*(\hat{\pi}_{\red}\otimes \hat{\pi}_{U_\gamma})(\hat{\Omega}^*_{u[21]})_{[23]}
U^*_{\gamma,21[13]} a_{[3]}
((J_{\hat{M}}\otimes J_A)(\hat{\pi}_{12}^2(z)\otimes 1)\gamma(b)\xi)_{[13]}\bigr) \\
&=
\Ind_{\X}(U_\gamma)
\hat{\pi}_{12}^1(z)_{[1]}
U_{\gamma [23]}^*(\hat{\pi}_{\red}\otimes\! \hat{\pi}_{U_\gamma})(\hat{\Omega}^*_{u[21]})_{[23]}
U^*_{\gamma,21[13]} b_{[3]}
((J_{\hat{M}}\otimes J_A)(\hat{\pi}_{12}^2(x)\otimes\! 1)\gamma(a)U_\gamma^* (J_{\hat{M}}\otimes \! J_A)\xi)_{[13]}.
\end{split}\]
Next we use the identification 
\[
L^2(A)\ni \eta\mapsto U^*_\gamma (\hat{\pi}_{\red}\otimes \hat{\pi}_{U_\gamma})(\hat{\Omega}^*_{u[21]})(1\otimes \eta)\in \Ind_{\X}(L^2(A))
\] 
from Proposition \ref{PropIdentStandard} and set $a=b=1, x=z=1$:
\[
(J_{\hat{M}}\otimes J_{\Ind_{\X}(A)})
U^*_{\gamma,21} 
(J_{\hat{M}}\otimes J_A)\xi =
U_{\gamma,11}
U^*_{\gamma,21}
(J_{\hat{M}}\otimes J_A)U_\gamma^* (J_{\hat{M}}\otimes J_A)\xi=
U_{\gamma,11}
U^*_{\gamma,21}
U_\gamma\xi
\]
(we are abusing notation and now consider $J_{\Ind_{\X}(A)}$ as acting on $L^2(A)$). Note that we have used $U_\gamma = (J_{\hat{M}}\otimes J_A)U_{\gamma}^* (J_{\hat{M}}\otimes J_A)$ \cite{Vae01}*{Proposition 3.7}. Using again the formulas from Proposition \ref{PropInducedCorep} gives
\[
(J_{\hat{M}}\otimes J_{\Ind_{\X}(A)})
u_{\GG [1]}( \hat{\pi}_{\red} \otimes \hat{\pi}_{U_{\gamma}})(
(1\otimes X_{\hat{\Omega},u})\hat{\Omega}_{u,c})
u_{\GG [1]}
U_{\gamma}^*
(J_{\hat{M}}\otimes J_A)=
(\hat{\pi}_{\red}\otimes\hat{\pi}_{U_{\gamma}})(\hat{\Omega}_{u[21]})
U_\gamma.
\]
An easy simplification using $J_A\hat{\pi}_{U_{\gamma}}(x)^*J_A=\hat{\pi}_{U_{\gamma}}(\hat{R}_u(x))$ for $x\in C^*(\G)$ (\cite{Vae01}*{Proposition 3.7}) gives
\[
J_{\hat{M}}\otimes J_{\Ind_{\X}(A)}=
(J_{\hat{M}}\otimes J_A)
 \hat{\pi}_{U_{\gamma}}(X_{\hat{\Omega},u})_{[2]}^*.
\]
Taking adjoints and slicing off the first leg ends the proof. 
\end{proof}

\section{Application to braided tensor products}

\subsection{Linking quantum groupoid for a bicharacter}\label{SecLQGFAB}

We resume the setting discussed at the end of Section \ref{SubSecCocBich}. So, $\G_1 = (M_1,\Delta_1),\G_2 = (M_2,\Delta_2)$ are two LCQG, and we put 
\[
\G = \G_1\times \G_2 = (M,\Delta) = (M_1\bar{\otimes}M_2,\Delta_{\otimes}).
\]
Then 
\[
\hat{\G} = \hat{\G}_1\times \hat{\G}_2 = (\hat{M},\hat{\Delta}) = (\hat{M}_1\bar{\otimes}\hat{M}_2,\hat{\Delta}_{\otimes}). 
\]

\begin{Lem}\label{lemma2}
The universal group C$^*$-algebra of $\G = \G_1\times \G_2$ is  given by 
\[
(\hat{M}_u,\hat{\Delta}_u),\qquad \hat{M}_u = C^*(\G_1\times \G_2) = \hat{M}_{1,u}\otimes_{\max}\hat{M}_{2,u}, 
\]
where 
\[
\hat{\Delta}_u(x\otimes y) = \hat{\Delta}_{1,u}(x)_{[13]}\hat{\Delta}_{2,u}(y)_{[24]} \in  \Mult((\hat{M}_{1,u}\otimes_{\max} \hat{M}_{2,u})\otimes (\hat{M}_{1,u}\otimes_{\max}\hat{M}_{2,u})).
\]
\end{Lem}

\begin{proof}
As explained in \cite{Kra21}*{Proposition 13.32}, we can identify 
\[
C^*(\G)=C^*(\G_1)\otimes_{\max}C^*(\G_2),\qquad \wW_{\G}=\wW_{\G_1 [13]} \wW_{\G_2 [24]}
\] 
(with leg numbering notation in $(C_0(\G_1)\otimes C_0(\G_2))\otimes (C^*(\G_1)\otimes_{\max}C^*(\G_2))$). Consequently
\[\begin{split}
(\id\otimes \id\otimes \hat{\Delta}_u)
(\wW_{\G_1 [13]} \wW_{\G_2 [24]})
&=(\id\otimes \hat{\Delta}_u)\wW_{\G}=
\wW_{\G [1256]}\wW_{\G [1234]}
=
\wW_{\G_1 [15]}\wW_{\G_2 [26]}
\wW_{\G_1 [13]}\wW_{\G_2 [24]}\\
&=
(\id\otimes \hat{\Delta}_{1,u})(\wW_{\G_1})_{[135]}
(\id\otimes \hat{\Delta}_{2,u})(\wW_{\G_2})_{[246]}
\end{split}\]
and the claim follows.
\end{proof}

It is then also easily shown that the associated reduced group C$^*$-algebra is 
\[
C^*_{\red}(\G_1\times \G_2) = C^*_{\red}(\G_1)\otimes C^*_{\red}(\G_2)
\]
with $\hat{\pi}_{\red}\colon C^*(\G_1)\otimes_{\max} C^*(\G_2) \rightarrow C^*_{\red}(\G_1) \otimes C^*_{\red}(\G_2)$ being the composition of $\hat{\pi}_{1,\red}\otimes_{\max} \hat{\pi}_{2,\red}$ with the tensor product quotient map. The universal unitary antipode of $\G_1\times \G_2$ is 
\[
\hat{R}_u = \hat{R}_{1,u}\otimes_{\max}\hat{R}_{2,u}
\]
(\cite{Kus01}*{Proposition 9.3}).

Assume now that we are given a unitary skew bicharacter
\[
\wh{\mcX}\in \hat{M}_1\bar{\otimes}\hat{M}_2,
\] 
with associated unitary $2$-cocycle 
\begin{equation}\label{EqCocycBiChar}
\hat{\Omega} := \wh{\mcX}_{[32]} \in \hat{M}\bar{\otimes}\hat{M} = (\hat{M}_1\bar{\otimes}\hat{M}_2)\bar{\otimes}(\hat{M}_1\bar{\otimes}\hat{M}_2). 
\end{equation}
for $(\hat{M},\hat{\Delta})$. It can be shown that $\wh{\mcX}$ lifts (uniquely) to a unitary skew bicharacter $\wh{\mcX}_u \in \Mult(\hat{M}_{1,u}\otimes \hat{M}_{2,u})$ \cite{Kus01}*{Section 6} (see also \cite{Daw16}*{Section 2.1}). In particular, we can make sense of the unitary $2$-cocycle
\[
\hat{\Omega}_u :=  \wh{\mcX}_{u [32]} \in \Mult((\hat{M}_{1,u}\otimes_{\max}\hat{M}_{2,u})\otimes (\hat{M}_{1,u}\otimes_{\max}\hat{M}_{2,u})) = \Mult(\hat{M}_u \otimes \hat{M}_u).
\]

In the next theorem, we use the terminology as introduced in Definition \ref{DefUnivLift2Coc}.
\begin{Theorem}\label{TheoUnivLift}
The unitary $2$-cocycle $\hat{\Omega}$ for $(\hat{M}_1\bar{\otimes}\hat{M}_2,\hat{\Delta}_{\otimes})$ admits $\hat{\Omega}_u$ as a universal lift.
\end{Theorem} 
Although this statement looks quite innocuous, it is not at all clear to us how to obtain it in a straightforward way. The proof we offer is quite technical, so we relegate it to the separate subsection \ref{SecProofTech}. Note that the theorem seems non-trivial, even in the case when $\hat{\G}_1$ and $\hat{\G}_2$ are both coamenable!

Consider now the conjugate universal unitary $2$-cocycle $\hat{\Omega}_{u,c}$ defined by \eqref{EqConjCocy}. 

\begin{Lem}\label{LemConjCocy}
We have 
\[
\hat{\Omega}_{u,c} = \wh{\mcX}_{u [14]}^*.
\]
\end{Lem}
\begin{proof}
This follows immediately from the definition \eqref{EqConjCocy} and the fact that 
\[
\hat{R}_u = \hat{R}_{1,u}\otimes_{\max}\hat{R}_{2,u},\qquad 
(\hat{R}_{1,u} \otimes \hat{R}_{2,u})\wh{\mcX}_u = \wh{\mcX}_u.
\]
\end{proof}

Consider now the associated universal element as in Lemma \ref{LemUniqueUnitCob}, so 
\[
X_{\hat{\Omega},u} \in \Mult(C^*(\G)) = \Mult(C^*(\G_1)\otimes_{\max}C^*(\G_2)).
\]

\begin{Lem}\label{LemProjCobPair}
Under the quotient map 
\[
\hat{\theta}\colon C^*(\G_1)\otimes_{\max}C^*(\G_2) \rightarrow C^*(\G_1)\otimes C^*(\G_2),
\]
we have that 
\[
X_{\hat{\Omega},u} \mapsto \wh{\mcX}_u^*.
\]
\end{Lem}

\begin{proof}
The maps 
\[\begin{split}
\hat{\pi}_1 = \id_{C^*(\G_1)}\times \hat{\varepsilon}_2&\colon C^*(\G) \rightarrow C^*(\G_1),\qquad x_1\otimes x_2 \mapsto \hat{\varepsilon}_2(x_2)x_1,\qquad \forall x_1\in C^*(\G_1),x_2\in C^*(\G_2),\\
\hat{\pi}_2 =\hat{\varepsilon}_1\times \id_{C^*(\G_2)}&\colon C^*(\G) \rightarrow C^*(\G_2),\qquad x_1\otimes x_2 \mapsto \hat{\varepsilon}_1(x_1)x_2,\qquad \forall x_1\in C^*(\G_1),x_2\in C^*(\G_2)
\end{split}\]
define homomorphisms of universal linking quantum groupoids 
\[
\hat{\pi}_i^{(2)} = \begin{pmatrix} \hat{\pi}_i & \hat{\pi}_i\\ \hat{\pi}_i& \hat{\pi}_i\end{pmatrix}\colon C^*(\msG) \rightarrow M_2(C^*(\G_i)),
\]
where the right hands side is equipped with the trivial linking quantum groupoid structure $\begin{pmatrix} \hat{\Delta}_{i,u} & \hat{\Delta}_{i,u} \\ \hat{\Delta}_{i,u} & \hat{\Delta}_{i,u}\end{pmatrix}$. By the linking quantum groupoid analogon of \cite{KV00}*{Proposition 5.45} (see also \cite{Kus01}*{Remark 12.1}), we get that 
\[
\hat{\pi}_i\circ\hat{R}_{u,kl} = \hat{R}_{i,u} \circ \hat{\pi}_i,\qquad \forall i,k,l\in \{1,2\}. 
\]

Now since the coproduct of $C^*(\msG)$ is obtained by cocycle twisting with $\hat{\Omega}_u = \wh{\mcX}_{u [32]}$, we see that 
\[
\hat{\theta}= (\hat{\pi}_1\otimes \hat{\pi}_2)\hat{\Delta}_{u,12}.
\]
It follows that 
\[\begin{split}
\hat{\theta}(X_{\hat{\Omega},u}) 
= (\hat{\pi}_1\otimes \hat{\pi}_2)\hat{\Delta}_{u,12}(\hat{R}_{u,21}(1))
&=  (\hat{\pi}_1\otimes \hat{\pi}_2)(\hat{R}_{u,21}\otimes \hat{R}_{u,21})\hat{\Delta}_{u,21}^{\opp}(1)
=  (\hat{R}_{1,u}\otimes \hat{R}_{2,u})(\hat{\pi}_1\otimes \hat{\pi}_2)\hat{\Delta}_{u,21}^{\opp}(1)
\\
&=  (\hat{R}_{1,u}\otimes \hat{R}_{2,u})(\hat{\pi}_1\otimes \hat{\pi}_2)(\wh{\mcX}_{u [14]}^*)
= (\hat{R}_{1,u}\otimes \hat{R}_{2,u})(\wh{\mcX}_{u}^*)
= \wh{\mcX}_u^*.
\end{split}\]
\end{proof}

In the following, we write again
\[
W^*(\msG) = \begin{pmatrix} W^*(\Hh) & W^*(\X)\\ W^*(\Y) & W^*(\G)\end{pmatrix} 
\]
for the linking quantum groupoid $\msG$ associated to the cocycle $\hat{\Omega}$. In particular, we can consider the associated right Galois object $(L^{\infty}(\X),\alpha_{\X})$ as in Section \ref{SecBiGal}. 

\begin{Lem}\label{lemma1}
The modular conjugation of $N=L^{\infty}(\X)$ is $J_{N}=(J_1\otimes J_2)\wh{\mc{X}}$.
\end{Lem}

\begin{proof}
According to \eqref{EqMultUnitCocy2}, for $a\in \{1,2\}$ we have
\begin{equation}\label{eq14}
\hat{\ww}^{a}_{11}=(J_N\otimes J_{\hat{M}})\hat{\Omega}\hat{\ww}^* (J_M\otimes J_{\hat{M}})\hat{\Omega}^*.
\end{equation}
Recall that $W^*(\Hh)=\hat{M}_1\bar\otimes \hat{M}_2$ with comultiplication $\hat{\Omega}\Delta_{\hat{M}}(\cdot)\hat{\Omega}^*$; see Section \ref{SubSecCocBich}. This means, that $\hat{\Hh}$ can be also constructed as a double crossed product, for quantum groups $\G_1^{op},\G_2$ and matching $\operatorname{Ad}(\wh{\mc{X}})$ (\cite{BV05}*{Definition 3.1}). After small simplifications, Theorem 5.3 in \cite{BV05}  gives
\[
\hat{\ww}^a_{11}=
\hat{\ww}_{1 [13]}
\bigl(\wh{\mc{X}}(\hat{J}_1\otimes \hat{J}_2)
\wh{\mc{X}}(\hat{J}_1\otimes \hat{J}_2)\bigr)^*_{[34]}
\hat{\ww}_{2[24]}
\bigl(\wh{\mc{X}}(\hat{J}_1\otimes \hat{J}_2)
\wh{\mc{X}}(\hat{J}_1\otimes \hat{J}_2)\bigr)_{[34]}
\]
which can be manipulated to
\begin{equation}\label{eq15}
\hat{\ww}^a_{11}=
(J_1\otimes J_2\otimes \hat{J}_1\otimes \hat{J}_2)
\wh{\mc{X}}_{[12]}\wh{\mc{X}}_{[32]}
\hat{\ww}_{1[13]}^* \hat{\ww}_{2[24]}^*
(J_1\otimes J_2\otimes \hat{J}_1\otimes \hat{J}_2)
\wh{\mc{X}}_{[32]}^*.
\end{equation}
Comparing expressions \eqref{eq14} and \eqref{eq15} leads to the claim $J_N=(J_1\otimes J_2)\wh{\mc{X}}$.
\end{proof}

For $i\in\{1,2\}$, the locally compact quantum group $\G_i$ acts on $L^{\infty}(\G_i)$ by the translation action, i.e.\ via comultiplication. The bicharacter $\wh{\mc{X}}\in L^{\infty}(\widehat{\G}_1)\bar\otimes L^{\infty}(\widehat{\G}_2)$ then allows us to define the braided tensor product $L^{\infty}(\G_1)\ov\boxtimes L^{\infty}(\G_2)$ (see the Introduction). Let us denote the generators in this specific case by $\varpi_{\boxtimes,1}(x_1),\varpi_{\boxtimes,2}(x_2)\,(x_1\in L^{\infty}(\GG_1),x_2\in L^\infty(\GG_2))$, in order to avoid confusion with the generators of $A_1\ov\boxtimes A_2$ later on.

\begin{Lem}\label{LemBraidGGX}
We have an isomorphism
\[
\Theta_{\boxtimes}\colon L^{\infty}(\G_1) \ov{\boxtimes} L^{\infty}(\G_2)\rightarrow 
L^{\infty}(\X)\colon z \mapsto \widehat{\mc{X}}^* z \widehat{\mc{X}}
\]
which acts on generators via
\[
\varpi_{\boxtimes,1}(x_1)\mapsto x_1\otimes 1,\quad 
\varpi_{\boxtimes,2}(x_2)\mapsto \widehat{\mc{X}}^*(1\otimes x_2)\widehat{\mc{X}},\qquad \forall x_i\in L^{\infty}(\G_i).
\]
Furthermore
\begin{equation}\label{EqActionBraid}
\alpha_{\X}(\Theta_{\boxtimes} \varpi_{\boxtimes,1}(x_1)) = (\Theta_{\boxtimes}\varpi_{\boxtimes,1}\otimes \id)\Delta_1(x_1)_{[123]},\quad 
\alpha_{\X}(\Theta_{\boxtimes} \varpi_{\boxtimes,2}(x_2)) = (\Theta_{\boxtimes}\varpi_{\boxtimes,2}\otimes \id)\Delta_2(x_2)_{[124]}
\end{equation}
for $x_1 \in L^{\infty}(\G_1), x_2\in L^{\infty}(\G_2)$.
\end{Lem} 

\begin{proof}
We first compute that 
\begin{equation}\begin{split}\label{eq59}
\wh{\mcX}_{[13]}\ww_{2[23]}
&= \wh{\mcX}_{[12]}^*(\wh{\mcX}_{[12]} \wh{\mcX}_{[13]}\ww_{2[23]})=
 \wh{\mcX}_{[12]}^*\hat{\ww}_{2[32]}^*\wh{\mcX}_{[12]}\hat{\ww}_{2[32]}\ww_{2[23]}
= \wh{\mcX}_{[12]}^* \ww_{2[23]}\wh{\mcX}_{[12]}.
\end{split}\end{equation}
Using \eqref{EqMultUnitCocy1} with $k\in\{1,2\}$, adjoint of \eqref{eq59} and the fact that the space $L^{\infty}(\X)\subseteq \mcB(L^2(\G_1)\otimes L^2(\G_2))$ is closed under adjoints, we can write it as
\begin{eqnarray*}
L^{\infty}(\X) &=& 
\{(\id\otimes \id \otimes \omega_1\otimes \omega_2)
( \hat{\ww}^{k *}_{12 [3412]}) \mid \omega_i \in \mcB(L^2(\G_i))_*\}'' \\
& = &
\{(\id\otimes \id \otimes \omega_1\otimes \omega_2)(\wh{\mcX}_{[14]}\ww_{1[13]}\ww_{2[24]}) \mid \omega_i \in \mcB(L^2(\G_i))_*\}'' \\
&=&  \{(\id\otimes \id\otimes \omega_2)(\wh{\mcX}_{[13]}\ww_{2[23]}) ((\id\otimes \omega_1)(\ww_1)\otimes 1) \mid \omega_i \in \mcB(L^2(\G_i))_*\}''\\
&=&  \{ ((\id\otimes \omega_1)(\ww_1^*)\otimes 1)(\id\otimes \id\otimes \omega_2)(\ww_{2[23]}^*\wh{\mcX}_{[13]}^*) \mid \omega_i \in \mcB(L^2(\G_i))_*\}''\\
&=& ((L^{\infty}(\G_1)\otimes1) \cup \wh{\mcX}^*(1\otimes L^{\infty}(\G_2))\wh{\mcX})''
\end{eqnarray*}
Comparing with \eqref{EqIdentification}, we see that this is just a conjugation of $L^{\infty}(\G_1)\ov{\boxtimes} L^{\infty}(\G_2)$ by $\wh{\mcX}^*$, and we find the concrete isomorphism
\[
\Theta_{\boxtimes}\colon L^{\infty}(\G_1)\ov{\boxtimes} L^{\infty}(\G_2) \cong L^{\infty}(\X),\quad \varpi_{\boxtimes,1}(x_1) \mapsto x_1\otimes 1,\quad \varpi_{\boxtimes,2}(x_2)\mapsto \wh{\mcX}^*(1\otimes x_2)\wh{\mcX},\qquad \forall x_i \in L^{\infty}(\G_i). 
\]

Take $x_1=(\omega_1\otimes\id)\hat{\ww}_{1}, x_2=(\omega_2\otimes\id)\hat{\ww}_2$ for $\omega_1\in \mcB(L^2(\GG_1))_*,\omega_2\in \mcB(L^2(\GG_2))_*$. We will prove
\begin{equation}\label{eq42}
\alpha_{\X}\Theta_{\boxtimes}\bigl(
\varpi_{\boxtimes,1}(x_1) \varpi_{\boxtimes,2}(x_2)\bigr)
= (\Theta_{\boxtimes}\varpi_{\boxtimes,1}\otimes \id)\Delta_1(x_1)_{[123]}\,
(\Theta_{\boxtimes}\varpi_{\boxtimes,2}\otimes \id)\Delta_2(x_2)_{[124]},
\end{equation}
from which \eqref{EqActionBraid} follows by a limit argument. First observe that using $\wh{\mc{X}}^*_{[13]} \hat{\ww}_{2[23]} \wh{\mc{X}}_{[13]}=\hat{\ww}_{2[23]}\wh{\mc{X}}^*_{[12]}$ we have
\[\begin{split}
\Theta_{\boxtimes}\bigl(
\varpi_{\boxtimes,1}(x_1) \varpi_{\boxtimes,2}(x_2)\bigr)&=
\bigl( (\omega_1\otimes\id)\hat{\ww}_1\otimes 1\bigr)\,
\wh{\mc{X}}^* \bigl(1\otimes (\omega_2\otimes\id)\hat{\ww}_2\bigr)\wh{\mc{X}}\\
&=
(\omega_1\otimes\omega_2\otimes\id\otimes\id)
\bigl(
\hat{\ww}_{1 [13]} \hat{\ww}_{2[24]} \wh{\mc{X}}^*_{[32]}\bigr).
\end{split}\]
Consequently, using \eqref{eq17} and \eqref{EqMultUnitCocy1}
\[\begin{split}
&\quad\;
\alpha_{\X}\Theta_{\boxtimes} \bigl(
\varpi_{\boxtimes,1}(x_1) \varpi_{\boxtimes,2}(x_2)\bigr)=
(\omega_1\otimes\omega_2\otimes\Delta_{12}^{2})(
\hat{\ww}_{1[13]} \hat{\ww}_{2[24]} \wh{\mc{X}}^*_{[32]})\\
&=
(\omega_1\otimes\omega_2\otimes\id^{\otimes 4})(
\hat{\ww}_{1[15]}\hat{\ww}_{2[26]}
\hat{\ww}_{1[13]}\hat{\ww}_{2[24]}
\wh{\mc{X}}^*_{[32]}
)\\
&=
(\omega_1\otimes \id^{\otimes 4})
(\hat{\ww}_{1[14]}\hat{\ww}_{1[12]})
\,
(\omega_2\otimes \id^{\otimes 4})
(\hat{\ww}_{2[15]} \hat{\ww}_{2[13]} \wh{\mc{X}}^*_{[21]})\\
&=
(\Theta_{\boxtimes}\varpi_{\boxtimes,1}\otimes \id)
\Delta_1( (\omega_1\otimes \id)\hat{\ww}_1)_{[123]}\,
(\Theta_{\boxtimes}\varpi_{\boxtimes,2}\otimes \id)
\Delta_2( (\omega_2\otimes \id)\hat{\ww}_2)_{[124]}\\
&=
(\Theta_{\boxtimes}\varpi_{\boxtimes,1}\otimes \id)
\Delta_1( x_1)_{[123]}\,
(\Theta_{\boxtimes}\varpi_{\boxtimes,2}\otimes \id)
\Delta_2( x_2)_{[124]}.
\end{split}\]
This shows \eqref{eq42} and ends the proof.
\end{proof}

Fix left $\G_i$-W$^*$-algebras $(A_i,\gamma_i)$ and equip $A_1\bar\otimes A_2$ with the left $\GG_1\times \GG_2$ action  
\[
\gamma_{\otimes}(a_1\otimes a_2) = \gamma_1(a_1)_{[13]}\gamma_2(a_2)_{[24]}
\] 
for $a_1\in A_1,a_2\in A_2$. We want to obtain an alternative description of the braided tensor product $A_1\ov\boxtimes A_2$ as the induced algebra $\Ind_{\X}(A_1\bar\otimes A_2)$ (see Theorem \ref{CorIdent}).

\begin{Lem}\label{lemma12}
We have an inclusion of von Neumann algebras 
\begin{equation}\label{eq43}
\bigl((\Theta_{\boxtimes}\varpi_{\boxtimes,1}\otimes \id)(\gamma_1(A_1))_{[123]} \cup (\Theta_{\boxtimes}\varpi_{\boxtimes,2} \otimes \id)(\gamma_2(A_2))_{[124]}\bigr)''\subseteq
\Ind_{\X}(A_1\bar{\otimes} A_2)
\end{equation}
inside $\mcB(L^2(\G_1)\otimes L^2(\G_2)\otimes L^2(A_1)\otimes L^2(A_2))$.
\end{Lem} 

\begin{proof}
By definition, we have 
\begin{equation}\label{EqInductionFromX}
\Ind_{\X}(A_1\bar{\otimes}A_2) = \{z\in L^{\infty}(\X) \bar{\otimes} A_1 \bar{\otimes}A_2 \mid (\alpha_{\X}\otimes \id\otimes \id)z = (\id\otimes\id\otimes \gamma_{\otimes})z\}.
\end{equation}
The claim now  follows immediately from Lemma \ref{LemBraidGGX}.
\end{proof} 

Next we want to show that \eqref{eq43} is an equality. It will be convenient to introduce an auxilliary notation for the left hand side:
\begin{equation}\label{eq45}
\mc{M}=\bigl((\Theta_{\boxtimes}\varpi_{\boxtimes,1}\otimes \id)(\gamma_1(A_1))_{[123]} \cup (\Theta_{\boxtimes}\varpi_{\boxtimes,2} \otimes \id)(\gamma_2(A_2))_{[124]}\bigr)''.
\end{equation}
For future reference, let us note that generators of $\mc{M}$ are 
\begin{equation}\label{eq46}
 (\Theta_{\boxtimes}\varpi_{\boxtimes,1}\otimes \id)(\gamma_1(a_1))_{[123]}=\gamma_1(a_1)_{[13]},
 \quad 
 (\Theta_{\boxtimes}\varpi_{\boxtimes,2}\otimes \id)(\gamma_2(a_2))_{[124]}=
\wh{\mc{X}}^*_{[12]}\gamma_2(a_2)_{[24]}
\wh{\mc{X}}_{[12]}
\end{equation}
for $a_1\in A_1,a_2\in A_2$.

Recall that the quantum group $\HH$ (dual to $D_{\wh{\mc{X}}}(\wh{\GG}_1,\wh{\GG}_2)$) acts on $\Ind_{\X}(A_1\bar\otimes A_2)$ via
\[
\Ind_{\X}(\gamma_\otimes)=(\Delta_{12}^1\otimes \id\otimes \id)_{\mid \Ind_{\X}(A_1\bar\otimes A_2)}
\colon \HH\curvearrowright \Ind_{\X}(A_1\bar\otimes A_2).
\]

\begin{Lem}\label{lemma13}
$\mc{M}\subseteq \Ind_{\X}(A_1\bar\otimes A_2)$ is invariant under the action of $\HH$, i.e.
\[
\Ind_{\X}(\gamma_\otimes) (\mc{M})\subseteq 
L^{\infty}(\HH)\bar\otimes \mc{M}.
\]
\end{Lem}

\begin{proof}
Let us first establish a description of $\Delta_{12}^1$: take $a\in\{1,2\}$. For $y\in L^{\infty}(\X)$ we have using \eqref{EqMultUnitCocy2} and Lemma \ref{lemma1} (or directly equation \eqref{eq15}) and $\wh{\mc{X}}_{[34]} \wh{\mc{X}}_{[14]} \ww_{1[13]}=\ww_{1[13]}\wh{\mc{X}}_{[14]}$ that

\[\begin{split}
&\quad\;
\Delta_{12}^1(y)=
\bigl( \hat{\ww}_{11}^a (y\otimes 1\otimes 1) \hat{\ww}_{11}^{a *}\bigr)_{[3412]}\\
&=
(\hat{J}_1\otimes \hat{J}_2\otimes 
J_1\otimes J_2)
\wh{\mc{X}}_{[34]}\wh{\mc{X}}_{[14]}
\ww_{1[13]} \ww_{2[24]}
(\hat{J}_1\otimes \hat{J}_2\otimes 
J_1\otimes J_2)
\wh{\mc{X}}_{[14]}^*\\
&\quad\quad\quad\quad\quad
\quad\quad\quad\quad\quad
\quad\quad\quad\quad\quad y_{[34]}
\wh{\mc{X}}_{[14]}
(\hat{J}_1\otimes \hat{J}_2\otimes 
J_1\otimes J_2)
\ww_{2[24]}^*
\ww_{1[13]}^*
\wh{\mc{X}}_{[14]}^*\wh{\mc{X}}_{[34]}^*
(\hat{J}_1\otimes \hat{J}_2\otimes 
J_1\otimes J_2)\\
&=
(\hat{J}_1\otimes \hat{J}_2\otimes 
J_1\otimes J_2)
\ww_{1[13]}\wh{\mc{X}}_{[14]} \ww_{2[24]}
(\hat{J}_1\otimes \hat{J}_2\otimes 
J_1\otimes J_2)
\wh{\mc{X}}_{[14]}^*\\
&\quad\quad\quad\quad\quad
\quad\quad\quad\quad\quad
\quad\quad\quad\quad\quad y_{[34]}
\wh{\mc{X}}_{[14]}
(\hat{J}_1\otimes \hat{J}_2\otimes 
J_1\otimes J_2)
\ww_{2[24]}^*
\wh{\mc{X}}_{[14]}^*
\ww_{1[13]}^*
(\hat{J}_1\otimes \hat{J}_2\otimes 
J_1\otimes J_2).
\end{split}\]
Now, take $a_1\in A_1$ and calculate using \eqref{eq46} that
\begin{equation}\begin{split}\label{eq54}
&\quad\;
\Ind_{\X}(\gamma_\otimes )\bigl(
 (\Theta_{\boxtimes}\varpi_{\boxtimes,1}\otimes \id)(\gamma_1(a_1))_{[123]}
\bigr)=
(\Delta_{12}^1\otimes\id\otimes\id)(\gamma_1(a_1)_{[13]})\\
&=
(\hat{J}_1\otimes \hat{J}_2\otimes 
J_1\otimes J_2\otimes J_{A_1}\otimes J_{A_2})
\ww_{1[13]}\wh{\mc{X}}_{[14]} \ww_{2[24]}
(\hat{J}_1\otimes \hat{J}_2\otimes 
J_1\otimes J_2\otimes J_{A_1}\otimes J_{A_2})
\wh{\mc{X}}_{[14]}^*\\
&\quad
\gamma_1(a_1)_{[35]}
\wh{\mc{X}}_{[14]}
(\hat{J}_1\otimes \hat{J}_2\otimes 
J_1\otimes J_2\otimes J_{A_1}\otimes J_{A_2})
\ww_{2[24]}^*
\wh{\mc{X}}_{[14]}^*
\ww_{1[13]}^*
(\hat{J}_1\otimes \hat{J}_2\otimes 
J_1\otimes J_2\otimes J_{A_1}\otimes J_{A_2})\\
&=
\ww_{1 [13]}^*  \gamma_1(a_1)_{[35]}\ww_{1[13]}=
(\Delta_1\otimes\id)\gamma_1(a_1)_{[135]}=
(\id\otimes\gamma_1)\gamma_1(a_1)_{[135]},
\end{split}\end{equation}
which belongs to $\mc{B}(L^2(\GG_1)\otimes L^2(\GG_2))\bar\otimes \mc{M}$. On the other hand, this element also belongs to the space $L^{\infty}(\HH)\bar\otimes \Ind_{\X}(A_1\bar\otimes A_2)$ by Lemma \ref{lemma12}. Thus \cite{Tak02}*{Corollary 5.10} gives
\[
\Ind_{\X}(\gamma_\otimes )\bigl(
 (\Theta_{\boxtimes}\varpi_{\boxtimes,1}\otimes \id)(\gamma_1(a_1))_{[123]}
\bigr)\in 
L^{\infty}(\HH)\bar\otimes \mc{M}.
\]
Next, take $a_2\in A_2$ and calculate using again \eqref{eq46}
\begin{equation}\begin{split}\label{eq55}
&\quad\;
\Ind_{\X}(\gamma_\otimes )\bigl(
 (\Theta_{\boxtimes}\varpi_{\boxtimes,2}\otimes \id)(\gamma_2(a_2))_{[124]}
\bigr)=
(\Delta_{12}^1\otimes\id\otimes\id)(\wh{\mc{X}}^*_{[12]} \gamma_2(a_2)_{[24]}
\wh{\mc{X}}_{[12]})\\
&=
(\hat{J}_1\otimes \hat{J}_2\otimes 
J_1\otimes J_2\otimes J_{A_1}\otimes J_{A_2})
\wh{\mc{X}}_{[34]}\wh{\mc{X}}_{[14]} \ww_{1[13]}\ww_{2[24]}
(\hat{J}_1\otimes \hat{J}_2\otimes 
J_1\otimes J_2\otimes J_{A_1}\otimes J_{A_2})
\wh{\mc{X}}_{[14]}^*\\
&\quad
\wh{\mc{X}}^*_{[34]}
\gamma_2(a_2)_{[46]}
\wh{\mc{X}}_{[34]}
\wh{\mc{X}}_{[14]}
(\hat{J}_1\otimes \hat{J}_2\otimes 
J_1\otimes J_2\otimes J_{A_1}\otimes J_{A_2})
\ww_{2[24]}^*
\ww_{1[13]}^*
\wh{\mc{X}}_{[14]}^*
\wh{\mc{X}}_{[34]}^*\\
&\quad\quad\quad\quad\quad\quad\quad\quad
\quad\quad\quad\quad\quad\quad\quad\quad
\quad\quad\quad\quad\quad\quad\quad\quad
\quad\quad\quad\quad
(\hat{J}_1\otimes \hat{J}_2\otimes 
J_1\otimes J_2\otimes J_{A_1}\otimes J_{A_2})\\
&=
\wh{\mc{X}}_{[34]}^*
(\hat{J}_1\otimes \hat{J}_2\otimes 
J_1\otimes J_2\otimes J_{A_1}\otimes J_{A_2})
\wh{\mc{X}}_{[14]} \ww_{2[24]}
(\hat{J}_1\otimes \hat{J}_2\otimes 
J_1\otimes J_2\otimes J_{A_1}\otimes J_{A_2})
\wh{\mc{X}}_{[14]}^*\ww_{1[13]}^*\\
&\quad
\gamma_2(a_2)_{[46]}
\ww_{1[13]}
\wh{\mc{X}}_{[14]}
(\hat{J}_1\otimes \hat{J}_2\otimes 
J_1\otimes J_2\otimes J_{A_1}\otimes J_{A_2})
\ww_{2[24]}^*
\wh{\mc{X}}_{[14]}^*
(\hat{J}_1\otimes \hat{J}_2\otimes 
J_1\otimes J_2\otimes J_{A_1}\otimes J_{A_2})
\wh{\mc{X}}_{[34]}\\
&=
\wh{\mc{X}}_{[34]}^*
u_{\G_1 [1]}
\wh{\mc{X}}_{[14]}^*
u_{\G_1 [1]} \ww_{2[24]}^*
\wh{\mc{X}}_{[14]}^*
\gamma_2(a_2)_{[46]}
\wh{\mc{X}}_{[14]}
\ww_{2[24]}
u_{\G_1 [1]}
\wh{\mc{X}}_{[14]}^*
u_{\G_1 [1]}\wh{\mc{X}}_{[34]}.
\end{split}\end{equation}
Consider $\wh{\mc{X}}_u$, the universal lift of bicharacter $\wh{\mc{X}}$. We can write $\wh{\mc{X}}_u=(\Phi\otimes\id)\WW_2$ for a non-degenerate $*$-homomorphism $\Phi\colon C_0^u(\GG_2)\rightarrow \Mult(C_0^u(\hat{\GG}_1))$ which respects comultiplication (see \cite{Kus01}*{Proposition 6.5}, \cite{MRW12}*{Proposition 4.2} and also \cite{DCKr24}*{Lemma 2.5}). We will write 
\[
\Phi^{\oon{vN}}\colon C_0^u(\GG_2)^{**}\rightarrow C_0^u(\hat{\GG}_1)^{**},\qquad \hat{\pi}^{\oon{vN}}_{1,\red}\colon C_0^u(\hat{\GG}_1)^{**}\rightarrow L^{\infty}(\hat{\GG}_1)
\]
for the unique extensions of $\Phi,\hat{\pi}_{1,\red}$ to normal, unital $*$-homomorphisms. We will also need the von Neumann algebraic versions $\Delta^{u,r,\oon{vN}}_2,\gamma_2^u$ of the maps $\Delta^{u,r}_2,\gamma_2$ (see the discussion at the beginning of \cite{DCKr24}*{Section 4.3}). Using these objects and \cite{DCKr24}*{Lemma 4.7}, we can calculate
\[\begin{split}
&\quad\;
\ww_{2[24]}^*\wh{\mc{X}}_{[14]}^*
\gamma_2(a_2)_{[46]}\wh{\mc{X}}_{[14]}\ww_{2[24]}\\
&=
(\hat{\pi}_{1,\red}^{\oon{vN}}\Phi^{\oon{vN}}\otimes\id\otimes\id\otimes\id)
\bigl(
\ww^*_{2 [23]} \Ww^*_{2 [13]}\gamma_2(a_2)_{[34]}
\Ww^*_{2 [13]}\ww^*_{2 [23]}
\bigr)_{[1246]}\\
&=
(\hat{\pi}_{1,\red}^{\oon{vN}}\Phi^{\oon{vN}}\otimes\id\otimes\id\otimes\id)
(\id\otimes \Delta_2\otimes \id)
(\Delta_2^{u,r,\oon{vN}}\otimes\id)\gamma_2(a_2)_{[1246]}\\
&=
(\hat{\pi}_{1,\red}^{\oon{vN}}\Phi^{\oon{vN}}\otimes\id\otimes\id\otimes\id)
(\id\otimes \id\otimes \gamma_2)
(\id\otimes \gamma_2)\gamma_2^u(a_2)_{[1246]}.
\end{split}\]
Since $C_0^u(\GG_2)\odot L^{\infty}(\GG_2)\odot A_2$ is weak$^*$-dense in $
C_0^u(\GG_2)^{**}\bar\otimes L^{\infty}(\GG_2)\bar\otimes A_2$, we can write
\[
C_0^u(\GG_2)^{**}\bar\otimes L^{\infty}(\GG_2)\bar\otimes A_2\ni (\id\otimes\gamma_2)\gamma_2^u(a_2)=\underset{i\in I}{w^*\text{-}\lim}\sum_{n=1}^{N_i}a_{n,i}\otimes b_{n,i}\otimes c_{n,i}
\]
for some $a_{n,i}\in C_0^u(\GG_2), b_{n,i}\in L^{\infty}(\GG_2),c_{n,i}\in A_2$. Consequently
\[\begin{split}
&\quad\;
\Ind_{\X}(\gamma_\otimes )\bigl(
 (\Theta_{\boxtimes}\varpi_{\boxtimes,2}\otimes \id)(\gamma_2(a_2))_{[124]}
\bigr)\\
&=
\underset{i\in I}{w^*\text{-}\lim}\sum_{n=1}^{N_i}
\wh{\mc{X}}_{[34]}^*
u_{\G_1 [1]}
\wh{\mc{X}}_{[14]}^*
u_{\G_1 [1]}
\bigl(
\hat{\pi}_{1,\red}\Phi(a_{n,i})\otimes b_{n,i}\otimes \gamma_2(c_{n,i})\bigr)_{[1246]}
u_{\G_1 [1]}
\wh{\mc{X}}_{[14]}
u_{\G_1 [1]}
\wh{\mc{X}}_{[34]}\\
&=
\underset{i\in I}{w^*\text{-}\lim}\sum_{n=1}^{N_i}
\bigl(\hat{\pi}_{1,\red}\Phi(a_{n,i})\otimes b_{n,i}\bigr)_{[12]}
\wh{\mc{X}}_{[34]}^*
u_{\G_1 [1]}
\wh{\mc{X}}_{[14]}^*
u_{\G_1 [1]}
\gamma_2(c_{n,i})_{[46]}
u_{\G_1 [1]}
\wh{\mc{X}}_{[14]}
u_{\G_1 [1]}
\wh{\mc{X}}_{[34]}\\
&=
\underset{i\in I}{w^*\text{-}\lim}\sum_{n=1}^{N_i}
\bigl(\hat{\pi}_{1,\red}\Phi(a_{n,i})\otimes b_{n,i}\bigr)_{[12]}
\wh{\mc{X}}_{[34]}^*
u_{\G_1 [1]}
(\hat{\pi}_{1,\red}^{\oon{vN}}\Phi^{\oon{vN}}\otimes\id\otimes\id)
(\Delta_{2}^{u,r,\oon{vN}}\otimes\id)
\gamma_2(c_{n,i})_{[146]}
u_{\G_1 [1]}
\wh{\mc{X}}_{[34]}\\
&=
\underset{i\in I}{w^*\text{-}\lim}\sum_{n=1}^{N_i}
\bigl(\hat{\pi}_{1,\red}\Phi(a_{n,i})\otimes b_{n,i}\bigr)_{[12]}
\wh{\mc{X}}_{[34]}^*
u_{\G_1 [1]}
(\hat{\pi}_{1,\red}^{\oon{vN}}\Phi^{\oon{vN}}\otimes\gamma_2)
\gamma_2^u (c_{n,i})_{[146]}
u_{\G_1 [1]}
\wh{\mc{X}}_{[34]}
\end{split}\]
which is easily seen (by slicing off legs $12$) to belong to $\mc{B}(L^2(\GG_1)\otimes L^2(\GG_2))\bar\otimes \mc{M}$. As before we conclude that 
\[
\Ind_{\X}(\gamma_\otimes )\bigl(
 (\Theta_{\boxtimes}\varpi_{\boxtimes,2}\otimes \id)(\gamma_2(a_2))_{[124]}
\bigr)\in L^{\infty}(\HH)\bar\otimes \mc{M},
\]
which ends the proof.
\end{proof}

Next we prove that the inclusion of Lemma \ref{lemma12} is in fact an equality.

\begin{Lem}\label{lemma14}
$\mc{M}=\Ind_{\X}(A_1\bar\otimes A_2)$.
\end{Lem}

\begin{proof}
The main strategy behind the proof of Lemma \ref{lemma14} is to use the result stating that taking second induction takes us back to the algebra that we have started with. Let us be more precise. Consider the reflected linking quantum groupoid $\mathscr{H}=\begin{pmatrix}
\GG & \Y \\ \X & \HH
\end{pmatrix}$ (Remark \ref{rem1}). Note that the comultiplication maps of $\wh{\mathscr{H}}$ are those of $\wh{\mathscr{G}}$, appropriately swapped. Since the LCQG $\HH$ acts on $\Ind_{\X}(A_1\bar\otimes A_2)$ and $\mc{M}$, we can consider induction along $\Y$:
\[
\Ind_{\Y}(\Ind_{\X}(A_1\bar\otimes A_2))=
\{z\in L^{\infty}(\Y)\bar\otimes 
\Ind_{\X}(A_1\bar\otimes A_2)\mid 
(\Delta_{\mscr{H},12}^2\otimes\id^{\otimes 4}) z=
(\id\otimes \id \otimes \Ind_{\X}(\gamma_\otimes))z\}
\]
and
\[
\Ind_{\Y}(\mc{M})=
\{z\in L^{\infty}(\Y)\bar\otimes 
\mc{M}\mid 
(\Delta_{\mscr{H},12}^2\otimes\id^{\otimes 4}) z=
(\id\otimes \id \otimes \Ind_{\X}(\gamma_\otimes))z\},
\]
where $\Delta_{\mscr{H},ij}^k$ are the comultiplication maps of $\mscr{H}$. Clearly we have 
$\Ind_{\Y}(\mc{M})\subseteq
\Ind_{\Y}(\Ind_{\X}(A_1\bar\otimes A_2))$. Now, \cite{DC09}
*{Theorem 8.2.2} and its proof states that
\[
\Xi\colon A_1\bar\otimes A_2\ni a \mapsto
(\Delta^1_{22}\otimes\id\otimes\id)\gamma_{\otimes}(a)\in \Ind_{\Y}(\Ind_{\X}(A_1\bar\otimes A_2))
\]
is a well defined isomorphism. We claim that
\begin{equation}\label{eq47}
\Xi(A_1\bar\otimes A_2)\subseteq \Ind_{\Y}(\mc{M}).
\end{equation}
If this is true, then
\[
\Ind_{\Y}(\Ind_{\X}(A_1\bar\otimes A_2))=\Ind_{\Y}(\mc{M})
\]
with equal induced action of $\G=\G_1\times \G_2$, hence taking yet another induction (along $\X$) and using again \cite{DC09}*{Theorem 8.2.2} gives us $\Ind_{\X}(A_1\bar\otimes A_2)=\mc{M}$.

Thus it is enough to prove \eqref{eq47}. First we need a more concrete description of the map $\Delta_{22}^1$. Take $a\in\{1,2\}, x\in L^{\infty}(\G)$. Then we have, using \eqref{EqMultUnitCocy2} and Lemma \ref{lemma1},
\begin{equation}\begin{split}\label{eq48}
&\quad\;
\Delta^1_{22}(x)=\bigl(
\hat{\ww}^a_{21} (x\otimes 1\otimes 1) \hat{\ww}^{a *}_{21}
 \bigr)_{[3412]}\\
 &=
 (\hat{J}_1\otimes \hat{J}_2\otimes J_1\otimes J_2)
 \wh{\mc{X}}_{[34]}\wh{\mc{X}}_{[14]}
\hat{\ww}^*_{1[31]}
\hat{\ww}^*_{2[42]}
(\hat{J}_1\otimes \hat{J}_2\otimes J_1\otimes J_2)\\
&\quad\quad\quad\quad 
\quad\quad\quad\quad 
x_{[34]}
(\hat{J}_1\otimes \hat{J}_2\otimes J_1\otimes J_2)
\hat{\ww}_{2[42]}
\hat{\ww}_{1[31]}
\wh{\mc{X}}_{[14]}^*
 \wh{\mc{X}}_{[34]}^*
 (\hat{J}_1\otimes \hat{J}_2\otimes J_1\otimes J_2)\\
 &=
  (\hat{J}_1\otimes \hat{J}_2\otimes J_1\otimes J_2)
 \wh{\mc{X}}_{[34]}\wh{\mc{X}}_{[14]}
(\hat{J}_1\otimes \hat{J}_2\otimes J_1\otimes J_2)
\ww_{1[13]}^*
\ww_{2[24]}^*
\\
&\quad\quad\quad\quad 
\quad\quad\quad\quad 
x_{[34]}
\ww_{2[24]}
\ww_{1[13]}
(\hat{J}_1\otimes \hat{J}_2\otimes J_1\otimes J_2)
\wh{\mc{X}}_{[14]}^*
 \wh{\mc{X}}_{[34]}^*
 (\hat{J}_1\otimes \hat{J}_2\otimes J_1\otimes J_2)\\
 &=
 \wh{\mc{X}}_{[34]}^*
 u_{\G_1 [1]} \wh{\mc{X}}_{[14]}^* u_{\G_1 [1]}
\Delta_{\G}(x)
u_{\G_1 [1]}
\wh{\mc{X}}_{[14]}
u_{\G_1 [1]}
 \wh{\mc{X}}_{[34]}.
\end{split}\end{equation}
Using equation \eqref{eq48}, we can now show \eqref{eq47} on generators. First, take $a_1\in A_1$ and calculate
\[\begin{split}
&\quad\;
\Xi(a_1\otimes 1)=
(\Delta^1_{22}\otimes \id\otimes \id)( \gamma_1(a_1)_{[13]})\\
&=
 (\hat{J}_1\otimes \hat{J}_2\otimes J_1\otimes J_2\otimes J_{A_1}\otimes J_{A_2})
 \wh{\mc{X}}_{[34]}\wh{\mc{X}}_{[14]}
\ww_{1[13]}
\ww_{2[24]}
(\hat{J}_1\otimes \hat{J}_2\otimes J_1\otimes J_2
\otimes J_{A_1}\otimes J_{A_2})\\
&\quad\quad\quad\quad 
\gamma_1(a_1)_{[35]}(\hat{J}_1\otimes \hat{J}_2\otimes J_1\otimes J_2
\otimes J_{A_1}\otimes J_{A_2})
\ww^*_{2[24]}
\ww^*_{1[13]}
\wh{\mc{X}}_{[14]}^*
 \wh{\mc{X}}_{[34]}^*
 (\hat{J}_1\otimes \hat{J}_2\otimes J_1\otimes J_2
 \otimes J_{A_1}\otimes J_{A_2})\\
 &=
  (\hat{J}_1\otimes \hat{J}_2\otimes J_1\otimes J_2
  \otimes J_{A_1}\otimes J_{A_2})
\ww_{1[13]}
\wh{\mc{X}}_{[14]}
(\hat{J}_1\otimes \hat{J}_2\otimes J_1\otimes J_2
\otimes J_{A_1}\otimes J_{A_2})\\
&\quad\quad\quad\quad 
\gamma_1(a_1)_{[35]}(\hat{J}_1\otimes \hat{J}_2\otimes J_1\otimes J_2
\otimes J_{A_1}\otimes J_{A_2})
\wh{\mc{X}}_{[14]}^*
\ww^*_{1[13]}
 (\hat{J}_1\otimes \hat{J}_2\otimes J_1\otimes J_2
 \otimes J_{A_1}\otimes J_{A_2})\\
 &=
 \ww_{1 [13]}^* \gamma_1(a_1)_{[35]} \ww_{1 [13]}=
 (\Delta_1\otimes\id)\gamma_1(a_1)_{[135]}=
  (\id\otimes \gamma_1)\gamma_1(a_1)_{[135]},
\end{split}\]
which belongs to $\mc{B}(L^2(\G_1)\otimes L^2(\G_2))\bar\otimes \mc{M}$ by \eqref{eq46}. Since also $\Xi(a_1\otimes 1)\in \Ind_{\Y}(\Ind_{\X}(A_1\bar\otimes A_2))$, it follows that $\Xi(a_1\otimes 1)\in \Ind_{\Y}(\mc{M})$. Next, take $a_2\in A_2$. We have
\[\begin{split}
&\quad\;
\Xi(1\otimes a_2)=(\Delta_{22}^1\otimes \id\otimes \id)(\gamma_2(a_2)_{[24]})\\
&=
 \wh{\mc{X}}_{[34]}^*
 u_{\G_1 [1]} \wh{\mc{X}}_{[14]}^* u_{\G_1 [1]}
(\Delta_{\G}\otimes\id\otimes\id)(\gamma_2(a_2)_{[24]})
u_{\G_1 [1]}
\wh{\mc{X}}_{[14]}
u_{\G_1 [1]} \wh{\mc{X}}_{[34]}\\
&=
 \wh{\mc{X}}_{[34]}^*
 u_{\G_1 [1]} \wh{\mc{X}}_{[14]}^*
(\id\otimes\gamma_2)\gamma_2(a_2)_{[246]}
\wh{\mc{X}}_{[14]}
u_{\G_1 [1]} \wh{\mc{X}}_{[34]}.
\end{split}\]
Observe that (using notation introduced in the proof of Lemma \ref{lemma13})
\[\begin{split}
&\quad\;
\wh{\mc{X}}_{[14]}^*
(\id\otimes\gamma_2)\gamma_2(a_2)_{[246]}
\wh{\mc{X}}_{[14]}=
(\hat{\pi}_{1,\red}\Phi\otimes\id)(\Ww_{2})_{[14]}^*
(\id\otimes\gamma_2)\gamma_2(a_2)_{[246]}
(\hat{\pi}_{1,\red}\Phi\otimes\id)(\Ww^*_{2})_{[14]}
\\
&=
\bigl((\id\otimes \hat{\pi}_{1,\red}^{\oon{vN}}\Phi^{\oon{vN}}\otimes\id\otimes\id)
(\id\otimes \Delta^{u,r,\oon{vN}}\otimes \id)
(\id\otimes\gamma_2)\gamma_2(a_2)\bigr)_{[2146]}\\
&=
\bigl((\id\otimes \hat{\pi}_{1,\red}^{\oon{vN}}\Phi^{\oon{vN}}\otimes\gamma_2)
(\id\otimes\gamma_2^u)\gamma_2(a_2)\bigr)_{[2146]},
\end{split}\]
hence we can write
\[
\wh{\mc{X}}_{[14]}^*
(\id\otimes\gamma_2)\gamma_2(a_2)_{[246]}
\wh{\mc{X}}_{[14]}=
\underset{i\in I}{w^*\text{-}\lim}\sum_{n=1}^{N_i}
(\hat{\pi}_{1,\red}\Phi(b_{i,n})\otimes a_{i,n} )_{[12]} \gamma_2(c_{i,n})_{[46]}
\]
for $(\id\otimes\gamma^u_2)\gamma_2(a_2)=\underset{i\in I}{w^*\text{-}\lim}\sum_{n=1}^{N_i}a_{i,n}\otimes b_{i,n}\otimes c_{i,n}$ and some $a_{i,n}\in L^{\infty}(\G_2)$, $b_{i,n}\in C_0^u(\G_2)$, $c_{i,n}\in A_2$. Consequently
\[
\Xi(1\otimes a_2)=
\underset{i\in I}{w^*\text{-}\lim}\sum_{n=1}^{N_i}
u_{\G_1 [1]}
(\hat{\pi}_{1,\red}\Phi(b_{i,n})\otimes a_{i,n} )_{[12]} u_{\G_1 [1]}
\wh{\mc{X}}^*_{[34]} \gamma_2(c_{i,n})_{[46]}
\wh{\mc{X}}_{[34]},
\]
and \eqref{eq46} ends the proof of \eqref{eq47}.
\end{proof}

As a corollary, we obtain an important result.

\begin{Theorem}\label{CorIdent}
We have an isomorphism of W$^*$-algebras
\[
A_1\overline{\boxtimes} A_2 \cong \Ind_{\X}(A_1\bar{\otimes} A_2)
\]
given on generators by 
\[\begin{split}
\pi_{\boxtimes,1}(a_1) &\mapsto (\Theta_{\boxtimes}\varpi_{\boxtimes,1}\otimes \id)(\gamma_1(a_1))_{[123]}=\gamma_1(a_1)_{[13]},\\
\pi_{\boxtimes,2}(a_2) &\mapsto (\Theta_{\boxtimes}\varpi_{\boxtimes,2}\otimes \id)(\gamma_2(a_2))_{[124]}=
\wh{\mc{X}}^*_{[12]}\gamma_2(a_2)_{[24]}
\wh{\mc{X}}_{[12]}
\end{split}\]
for $a_1\in A_1,a_2\in A_2$. 
\end{Theorem} 

\begin{proof}
We have a faithful, normal $*$-representation of $A_1$ on $L^2(\G_1)\otimes L^2(A_1)$ and of $A_2$ on $L^2(\G_2) \otimes L^2(A_2)$ via 
\[
\gamma_1\colon A_1\rightarrow \mc{B}(L^2(\G_1)\otimes L^2(A_1)),\quad 
\gamma_2\colon A_2\rightarrow \mc{B}(L^2(\G_2)\otimes L^2(A_2)).
\] 
Moreover, the unitary representations $\ww_{1[12]}$ and $\ww_{2[12]}$ implement the actions $\gamma_1,\gamma_2$. The corresponding $*$-representations of $C^*(\G_1),C^*(\G_2)$ are given by $\hat{\pi}_{1,\red}\otimes 1$ and $\hat{\pi}_{2,\red}\otimes 1$. Since the braided tensor product is up to isomorphism independent of the choices of implementations \cite{DCKr24}*{Proposition 4.4}, we obtain a faithful, normal $*$-homomorphism of $A_1\ov\boxtimes A_2$ into $\mcB(L^2(\G_1)\otimes L^2(A_1)\otimes L^2(\G_2)\otimes L^2(A_2))$: 
\[
A_1\ov\boxtimes A_2\ni \pi_{\boxtimes,1}(a_1)\pi_{\boxtimes,2}(a_2)\mapsto 
\wh{\mc{X}}_{[13]} \gamma_1(a_1)_{[12]}
\wh{\mc{X}}_{[13]}^* \gamma_2(a_2)_{[34]},\quad \forall a_1\in A_1,a_2\in A_2.
\]
Composing this map with $z\mapsto \wh{\mc{X}}^*_{[12]} z_{[1324]}\wh{\mc{X}}_{[12]}$, we obtain an isomorphism
\[
A_1\ov\boxtimes A_2\ni \pi_{\boxtimes,1}(a_1)\pi_{\boxtimes,2}(a_2)\mapsto 
\gamma_1(a_1)_{[13]}
\wh{\mc{X}}_{[12]}^* \gamma_2(a_2)_{[24]}
\wh{\mc{X}}_{[12]}\in \mc{M},  \quad \forall a_1\in A_1,a_2\in A_2.
\]
Lemma \ref{lemma14} now ends the proof.
\end{proof}

In the following, we will use that if $\G_1,\G_2$ are LCQG's and $(A_1,\gamma_1),(A_2,\gamma_2)$ are respectively a left $\G_1$- and a left $\G_2$-W$^*$-algebra, the standard implementation of the $\G_1\times \G_2$-action $\gamma_\otimes$ on $A_1\bar{\otimes}A_2$ is just 
\[
U_{\gamma_1 [13]}U_{\gamma_2[24]} \in L^{\infty}(\G_1)\bar{\otimes}L^{\infty}(\G_2)\bar{\otimes} \mcB(L^2(A_1)\otimes L^2(A_2)),
\]
upon identifying $L^2(A_1\bar{\otimes}A_2) \cong L^2(A_1)\otimes L^2(A_2)$.
This follows straightforwardly from the construction of the standard implementation in \cite{Vae01}. 

It follows from the previous theorem that we can identify 
\[
L^2(\Ind_{\X}(A_1\bar\otimes A_2))\cong L^2(A_1\ov\boxtimes A_2).
\] 
Theorem \ref{TheoIndL} identifies this space with $\Ind_{\X}(L^2(A_1)\otimes L^2(A_2))$ via a unitary $I$. Next, recall (Proposition \ref{PropIdentStandard}) that we have a unitary identification 
\[
\theta\mc{O}^*\colon L^2(A_1)\otimes L^2(A_2)\rightarrow \Ind_{\X}(L^2(A_1)\otimes L^2(A_2)),
\] 
where we have used simplified notation (neglecting subscripts). As will be clear from the proof of the next theorem, it is useful to compose these identifications with the unitary $
(\hat{\pi}_{U_{\gamma_1}}\otimes \hat{\pi}_{U_{\gamma_2}})(\wh{\mc{X}}_u)^*$ acting on $L^2(A_1)\otimes L^2(A_2)$. We denote the resulting unitary by $I_\boxtimes$.

\begin{Theorem}\label{TheoInductionBraided}
Let $(A_i,\gamma_i)$ be left $\G_i$-W$^*$-algebras. Consider the tensor product $\G$-W$^*$-algebra $ (A_1\bar{\otimes}A_2,\gamma_{\otimes})$. Then we can identify 
\begin{equation}\label{EqIdentificationMain}
I_{\boxtimes}= I^*\theta\mc{O}^*
(\hat{\pi}_{U_{\gamma_1}}\otimes \hat{\pi}_{U_{\gamma_2}})(\wh{\mc{X}}_u)^*
\colon L^2(A_1)\otimes L^2(A_2)\cong L^2(A_1\ov{\boxtimes} A_2)
\end{equation}
in such a way that:
\begin{enumerate}
\item The standard representation of $A_1\ov{\boxtimes} A_2$ becomes the $*$-representation \eqref{EqStandardRepBraid}.
\item The standard implementation of $\Ind_{\X}(\gamma_{\otimes})$ becomes the unitary
\[
(\hat{\pi}_{U_{\gamma_1}}\otimes \hat{\pi}_{U_{\gamma_2}})( \wh{\mc{X}}_u)_{[34]}
(\hat{\pi}_{1,\red}\otimes \hat{\pi}_{U_{\gamma_2 }})(\wh{\mc{X}}_u)_{[14]}
U_{\gamma_1 [13]}
U_{\gamma_2 [24]}
(u_{\G_1}\hat{\pi}_{1,\red}(-)u_{\G_1}\otimes \hat{\pi}_{U_{\gamma_2}})
(\wh{\mc{X}}_{u})_{[14]}.
\]
\item The modular conjugation $J_{A_1\ov\boxtimes A_2}$ becomes the anti-unitary
\begin{equation}\label{EqMainFormModConj}
(\hat{\pi}_{U_{\gamma_1}}\otimes \hat{\pi}_{U_{\gamma_2}})(\wh{\mcX}_u)(J_{A_1}\otimes J_{A_2}) = (J_{A_1}\otimes J_{A_2})(\hat{\pi}_{U_{\gamma_1}}\otimes \hat{\pi}_{U_{\gamma_2}})(\wh{\mcX}_u^*).
\end{equation}
\end{enumerate}
\end{Theorem}

\begin{proof}
As recalled above, by Proposition \ref{PropIdentStandard} and Theorem \ref{TheoUnivLift}, we have an identification 
\begin{multline}\label{eq51}
\theta\mc{O}^* \colon L^2(A_1)\otimes L^2(A_2) \rightarrow \Ind_{\X}(L^2(A_1)\otimes L^2(A_2)),\\
 \xi_1\otimes \xi_2 \mapsto U_{\gamma_1[13]}^*U_{\gamma_2[24]}^*(\hat{\pi}_{1,\red}\otimes \hat{\pi}_{U_{\gamma_2}})(\wh{\mc{X}}_u)_{[14]}^*(1\otimes 1 \otimes \xi_1\otimes \xi_2). 
\end{multline}

By Theorem \ref{TheoIndL}, we can identify (via the unitary $I$) the natural $*$-representation of $\Ind_{\X}(A_1\bar{\otimes} A_2)$ on the right hand side with the standard representation. Next, we  have the identification 
\[
\Ind_{\X}(A_1\bar{\otimes} A_2) \cong A_1\ov{\boxtimes} A_2
\] 
from Theorem \ref{CorIdent}, which allows us to realise the standard representation of $A_1\ov\boxtimes A_2$ on $L^2(A_1)\otimes L^2(A_2)$. A lengthy but straightforward computation using the pentagon equation, shows that it is given explicitely by
\begin{multline*}
A_1\overline{\boxtimes} A_2 \rightarrow \mcB(L^2(A_1)\otimes L^2(A_2)),\\
\pi_{\boxtimes,1}(a_1)\pi_{\boxtimes,2}(a_2) \mapsto (\hat{\pi}_{U_{\gamma_1}}\otimes \hat{\pi}_{U_{\gamma_2}})(\wh{\mcX}_u) ^*\pi_{\boxtimes,1}(a_1)\pi_{\boxtimes,2}(a_2) (\hat{\pi}_{U_{\gamma_1}}\otimes \hat{\pi}_{U_{\gamma_2}})(\wh{\mcX}_u).
\end{multline*}
Finally the identification \eqref{EqIdentificationMain} is then taken to be $I^*\theta\mc{O}^*$, precomposed with the unitary $(\hat{\pi}_{U_{\gamma_1}}\otimes \hat{\pi}_{U_{\gamma_2}})(\wh{\mc{X}}_u)^*\in \mc{B}(L^2(A_1)\otimes L^2(A_2))$:
\begin{multline}\label{eq52}
I_\boxtimes=I^*\theta\mc{O}^*
(\hat{\pi}_{U_{\gamma_1}}\otimes \hat{\pi}_{U_{\gamma_2}})(\wh{\mc{X}}_u)^*\colon 
L^2(A_1)\otimes L^2(A_2) \rightarrow 
L^2(A_1\ov\boxtimes A_2),\\
 \xi_1\otimes \xi_2 \mapsto 
I^* U_{\gamma_1[13]}^*U_{\gamma_2[24]}^*(\hat{\pi}_{1,\red}\otimes \hat{\pi}_{U_{\gamma_2}})(\wh{\mc{X}}_u)_{[14]}^*
(\hat{\pi}_{U_{\gamma_1}}\otimes \hat{\pi}_{U_{\gamma_2}})(\wh{\mc{X}}_u)^*_{[34]} 
 (1\otimes 1 \otimes 
 \xi_1\otimes \xi_2).
\end{multline}

The formula for the standard implementation now follows from Proposition \ref{PropInducedCorep},  Lemma \ref{LemConjCocy} and Lemma \ref{LemProjCobPair}.

To conclude, the first identity in \eqref{EqMainFormModConj} follows immediately from Theorem \ref{TheoModConjGen} and Lemma \ref{LemProjCobPair}, while the alternative formula holds since the modular conjugation is self-adjoint (alternatively, it can be obtained using \cite{Vae01}*{Proposition 3.7}).
\end{proof}

At the end of this section, we use the established theorems to deduce two results concerning faithfulness of states on $A_1\ov\boxtimes A_2$, transported from $A_1\bar\otimes A_2$. First we present a more general result.

\begin{Prop}\label{prop2}
Let $(A_i,\gamma_i)$ be left $\G_i$-W$^*$-algebras, and consider the tensor product $\G$-W$^*$-algebra $ (A_1\bar{\otimes}A_2,\gamma_{\otimes})$. Assume that $\omega\in (A_1\bar\otimes A_2)_*^+$ is a normal state, invariant under the $\G$-action $\gamma_\otimes$, and let $\xi\in L^2(A_1)\otimes L^2( A_2)$ be the unit vector in the standard positive cone of $A_1\bar\otimes A_2$ implementing $\omega$. Consider the normal state $\nu=\omega_{\xi}\in (A_1\ov\boxtimes A_2)_*^+$, where we let $A_1\ov\boxtimes A_2$ act on $L^2(A_1)\otimes L^2(A_2)$ via the defining representation.
\begin{enumerate}
\item We have $\nu(\pi_{\boxtimes,1}(a_1)\pi_{\boxtimes,2}(a_2))=\omega(a_1\otimes a_2)$ for $a_1\in A_1,a_2\in A_2$.
\item If $\omega$ is faithful, then $\nu$ is also faithful.
\end{enumerate}
\end{Prop}

\begin{proof}
Recall first that, since $\omega$ is invariant under the action of $\G$,  we have $U_{\gamma_1 [13]} U_{\gamma_2 [24]} (\zeta\otimes \xi)=\zeta\otimes \xi$ for $\zeta\in L^2(\G)$ \cite{DCKr24}*{Lemma 7.5}. It follows that 
\[
(\hat{\pi}_{U_{\gamma_1}}(x)\otimes \hat{\pi}_{U_{\gamma_2}}(y))\xi=\eps^u_{\hat{\G}}(x\otimes y)\xi,\qquad \forall x\otimes y\in C^*(\G).
\] 
Theorem \ref{TheoInductionBraided} allows us to treat $L^2(A_1)\otimes L^2(A_2)$ as the standard Hilbert space for $A_1\ov\boxtimes A_2$ with modular conjugation $I_\boxtimes^* J_{A_1\ov\boxtimes A_2} I_\boxtimes$. For $a_1\in A_1,a_2\in A_2$ we have
\begin{multline*}
\nu(\pi_{\boxtimes,1}(a_1) \pi_{\boxtimes,2}(a_2))=
\bigl\langle 
(\hat{\pi}_{U_{\gamma_1}}\otimes \hat{\pi}_{U_{\gamma_2}})(\wh{\mc{X}}_u)^*\xi \,,\,
(a_1\otimes 1)
(\hat{\pi}_{U_{\gamma_1}}\otimes \hat{\pi}_{U_{\gamma_2}})(\wh{\mc{X}}_u)^* (1\otimes a_2)
\xi\bigr\rangle \\=
\langle \xi , (a_1\otimes a_2)\xi\rangle=
\omega(a_1\otimes a_2),
\end{multline*}
using that $(\eps^u_{\hat{\G}_1}\otimes \id)\wh{\mc{X}}_u=1$ in the penultimate step. The vector $\xi$ is invariant under $I_\boxtimes^* J_{A_1\ov\boxtimes A_2} J_\boxtimes$, as follows from the concrete formula for $J_{A_1\ov\boxtimes A_2}$ in Theorem \ref{TheoInductionBraided}. Using this and the above calculation, we get
\begin{equation}\begin{split}\label{eq60}
&\quad\;
(I_\boxtimes^* J_{A_1\ov\boxtimes A_2} I_\boxtimes)
\pi_{\boxtimes,1}(a_1) \pi_{\boxtimes,2}(a_2)
(I_\boxtimes^* J_{A_1\ov\boxtimes A_2} I_\boxtimes)\xi=
(I_\boxtimes^* J_{A_1\ov\boxtimes A_2} I_\boxtimes)
(\hat{\pi}_{U_{\gamma_1}}\otimes \hat{\pi}_{U_{\gamma_2}})(\wh{\mc{X}}_u) (a_1\otimes a_2)
\xi.
\end{split}\end{equation}
If $\omega$ is faithful, then $\xi$ is cyclic for $(A_1\bar\otimes A_2)'$, and as $(J_{A_1}\otimes J_{A_2})\xi=\xi$,  we find that $\xi$ is cyclic for $A_1\bar\otimes A_2$. Equation \eqref{eq60} implies that $\xi$ is cyclic (hence also separating) for $(A_1\ov\boxtimes A_2)'$. Consequently $\nu$ is faithful.
\end{proof}

Next we specify to braided tensor products of states, see \cite{DCKr24}*{Proposition 5.1}.

\begin{Cor}
Let $(A_i,\gamma_i)$ be left $\G_i$-W$^*$-algebras. Consider the tensor product $\G$-W$^*$-algebra $ (A_1\bar{\otimes}A_2,\!\gamma_{\otimes})$. Let $\omega_i\in (A_i)_*^+$ be faithful, normal states, invariant under the $\G_i$-action, and consider the braided tensor product state
\[
\omega_1\boxtimes \omega_2\colon A_1\ov\boxtimes A_2\ni \pi_{\boxtimes,1}(a_1)\pi_{\boxtimes,2}(a_2)\mapsto \omega_1(a_1) \omega_2(a_2)\in \mathbb{C}.
\]
Then $\omega_1\boxtimes \omega_2$ is a faithful, normal state.
\end{Cor}

\begin{proof}
Consider first $\omega_1\otimes \omega_2\in (A_1\bar\otimes A_2)_*^+$. It is well known that $\omega_1\otimes \omega_2$ is a faithful, normal state, and since each $\omega_i$ is invariant under the action of $\G_i$, we have that $\omega_1\otimes \omega_2$ is invariant under the action of $\G$. Denoting $\xi_i\in L^2(A_i)$ vectors in the standard positive cones implementing $\omega_i$, we have that $\xi_1\otimes \xi_2$ is the vector in the standard positive cone of $L^2(A_1\bar{\otimes} A_2)$ implementing $\omega_1\otimes \omega_2$. Now the claim follows from \cite{DCKr24}*{Proposition 5.1} and Proposition \ref{prop2}.
\end{proof}

\subsection{Braided tensor product for quasi-triangular quantum groups}

We finally come back to proving the theorem stated in the introduction. Assume that $\Ll = (L,\Delta)$ is a LCQG, and put $\G_1 = \G_2 = \Ll$. Let $\wh{\mcX} \in \hat{L}\bar{\otimes}\hat{L}$ be a quasi-triangular structure, so $\wh{\mcX}$ is a unitary skew bicharacter and moreover 
\begin{equation}\label{EqQTStruct}
\wh{\mcX}\hat{\Delta}(x)\wh{\mcX}^* = \hat{\Delta}^{\opp}(x),\qquad \forall x\in \hat{L}. 
\end{equation}
Let $\hat{\G} =\hat{\G}_1\times \hat{\G}_2$, $\hat{\Omega}=\wh{\mc{X}}_{[32]}$ and let $\hat{\Hh} = D_{\wh{\mcX}}(\hat{\G}_1,\hat{\G}_2)$ be the generalized Drinfeld double associated to $\wh{\mcX}$.

\begin{Lem}
We have $\Ll$ as a closed quantum subgroup of $\Hh$ via
\[
\iota_{\Ll}\colon W^*(\Ll) \rightarrow W^*(\Hh) = W^*(\Ll)\bar{\otimes} W^*(\Ll),\qquad x\mapsto \hat{\Delta}(x). 
\]
\end{Lem} 
\begin{proof}
This is a normal embedding preserving the coproduct, since \eqref{EqQTStruct} implies
\[
\hat{\Omega}\hat{\Delta}_{\otimes}(\hat{\Delta}(x))\hat{\Omega}^* = (\hat{\Delta}\otimes \hat{\Delta})\hat{\Delta}(x)
\]
for $x\in \hat{L}$.
\end{proof}

It follows that if we have an $\Hh$-W$^*$-algebra $A$, we can restrict the action to obtain an $\Ll$-W$^*$-algebra $A$ (see \cite{DCKr24}*{Section 2.2}).

We now recall from \cite{DC09}*{Lemma 6.5.4} that, in the general situation of a closed quantum subgroup $\Ll \subseteq \Hh$ and an $\Hh$-W$^*$-algebra $(A,\gamma)$, its restriction $(A,\gamma_{\mid \Ll})$ satisfies 
\[
(U_{\gamma})_{\mid \Ll} = U_{\gamma_{\mid \Ll}},
\]
where on the left we write the restriction of a unitary representation to a closed quantum subgroup (see \cite{DCKr24}*{Section 2.2}). The following theorem then follows by combining this observation with Theorem \ref{TheoInductionBraided}.

\begin{Theorem}\label{TheoMainResultDone}
Let $(\Ll,\wh{\mcX})$ be a quasi-triangular quantum group, and let $(A_1,\gamma_1),(A_2,\gamma_2)$ be two $\Ll$-W$^*$-algebras. Then the $\Hh$-action on $A_1\ov{\boxtimes}A_2$ restricts to the $\Ll$-action given by \eqref{EqIdentification2}, and there is a canonical unitary intertwiner of unitary $\Ll$-representations 
\begin{equation}\label{EqFormulaImplBraid}
I_\boxtimes \colon L^2(A_1)\otimes L^2(A_2)\cong L^2(A_1\ov{\boxtimes} A_2),
\qquad U_{\gamma_1} \TensRep U_{\gamma_2} \cong  U_{\gamma_1\bowtie \gamma_2}
.
\end{equation}
This unitary intertwines the standard representation of $A_1\ov\boxtimes A_2$ with the identity representation \eqref{EqIdentification}, and satisfies 
\begin{equation}\label{eq53}
J_{A_1\barboxtimes A_2}\cong 
(\hat{\pi}_{U_{\gamma_1}}\otimes 
\hat{\pi}_{U_{\gamma_2}})(\wh{\mc{X}}_u)
(J_{A_1}\otimes J_{A_2}) .
\end{equation}
\end{Theorem} 

\begin{proof}
For the unitary $L^2(A_1)\otimes L^2(A_2)\rightarrow L^2(A_1\ov\boxtimes A_2)$ we take $I_\boxtimes =I^*\theta\mc{O}^*
(\hat{\pi}_{U_{\gamma_1}}\otimes \hat{\pi}_{U_{\gamma_2}})(\wh{\mc{X}}_u)^*
$, as in Theorem \ref{TheoInductionBraided}. It was already established that it intertwines the identity and standard representation of $A_1\ov\boxtimes A_2$, and satisfies \eqref{eq53}.

Next, we claim that the $\HH$-action $\Ind_{\X}(\gamma_\otimes)$ on $\Ind_{\X}(A_1\bar\otimes A_2)$ restricts to the $\Ll$-action $\gamma_1\bowtie \gamma_2$ on $A_1\ov\boxtimes A_2$ under the isomorphism
\[
A_1\ov\boxtimes A_2\cong \Ind_{\X}(A_1\bar\otimes A_2)
\]
(Theorem \ref{CorIdent}). Denoting this isomorphism by $\Upsilon$, we want to show
\begin{equation}\label{eq57}
\Ind_{\X}(\gamma_\otimes)_{| \Ll}\Upsilon=(\id\otimes \Upsilon)(\gamma_1\bowtie \gamma_2).
\end{equation}
Since $\Ll$ is a closed quantum subgroup of $\HH$, there is a left action
\[
\Delta_{\Ll,\HH}\colon \Ll\curvearrowright L^{\infty}(\HH)\colon \Delta_{\Ll,\HH}(x)=(\id\otimes \iota_{\Ll})(\ww_{\Ll})^* (1\otimes x) (\id\otimes \iota_{\Ll})(\ww_{\Ll}),\qquad \forall x\in L^{\infty}(\HH)
\]
 and the restricted action $\Ind_{\X}(\gamma_\otimes)_{|\Ll}$ is characterised by
\[
(\Delta_{\Ll,\Hh}\otimes \id^{\otimes 4})\Ind_{\X}(\gamma_\otimes)=
(\id\otimes \Ind_{\X}(\gamma_\otimes)) \Ind_{\X}(\gamma_\otimes)_{| \Ll}.
\]
Take $a_1\in A_1,a_2\in A_2$. Using \eqref{eq54} we calculate
\begin{equation}\begin{split}\label{eq56}
&\quad\;
(\id\otimes \Ind_{\X}(\gamma_\otimes))
\Ind_{\X}(\gamma_\otimes)_{|\mathbb{L}}
\bigl(\Upsilon(\pi_{\boxtimes,1}(a_1)) \bigr)=
(\Delta_{\mathbb{L},\HH}\otimes \id^{\otimes 4})
\Ind_{\X}(\gamma_\otimes)
\bigl(\Upsilon(\pi_{\boxtimes,1}(a_1)) \bigr)\\
&=
(\Delta_{\mathbb{L},\HH}\otimes \id^{\otimes 4})
\bigl(
(\id\otimes \gamma_1)\gamma_1(a_1)_{[135]}\bigr)=
\ww_{\mathbb{L} [12]}^*\ww_{\mathbb{L} [13]}^*
(\id\otimes \gamma_1)\gamma_1(a_1)_{[246]}
\ww_{\mathbb{L} [13]}\ww_{\mathbb{L} [12]}\\
&=
(\Delta_1\otimes \gamma_1)\gamma_1(a_1)_{[1246]}=
(\id\otimes \id\otimes \gamma_1)(\id\otimes \gamma_1)
\gamma_1(a_1)_{[1246]}\\
&=
(\id\otimes \Ind_{\X}(\gamma_\otimes)\Upsilon)
(\id\otimes \pi_{\boxtimes,1})\gamma_1(a_1)=
(\id\otimes \Ind_{\X}(\gamma_\otimes)\Upsilon)
(\gamma_1\bowtie \gamma_2)
\bigl(\pi_{\boxtimes,1}(a_1) \bigr).
\end{split}\end{equation}
Next we look at the second generator and calculate, using \eqref{eq55},
\[\begin{split}
&\quad\;
(\id\otimes \Ind_{\X}(\gamma_\otimes))
\Ind_{\X}(\gamma_\otimes)_{|\mathbb{L}}
\bigl(\Upsilon(\pi_{\boxtimes,2}(a_2)) \bigr)=
(\Delta_{\mathbb{L},\HH}\otimes \id^{\otimes 4})
\Ind_{\X}(\gamma_\otimes)
\bigl(\Upsilon(\pi_{\boxtimes,2}(a_2)) \bigr)\\
&=
(\Delta_{\mathbb{L},\HH}\otimes \id^{\otimes 4})\bigl(
\wh{\mc{X}}^*_{[34]}u_{\Ll [1]}
\wh{\mc{X}}^*_{[14]}u_{\Ll [1]}
\ww^*_{\Ll [24]}\wh{\mc{X}}^*_{[14]}
\gamma_2(a_2)_{[46]}
\wh{\mc{X}}_{[14]}
\ww_{\Ll [24]}
u_{\Ll [1]}
\wh{\mc{X}}_{[14]}u_{\Ll [1]}
\wh{\mc{X}}_{[34]}
\bigr)\\
&=
\ww_{\mathbb{L} [12]}^*\ww_{\mathbb{L} [13]}^*\wh{\mc{X}}^*_{[45]}u_{\Ll [2]}
\wh{\mc{X}}^*_{[25]}u_{\Ll [2]}
\ww^*_{\Ll [35]}\wh{\mc{X}}^*_{[25]}
\gamma_2(a_2)_{[57]}
\wh{\mc{X}}_{[25]}
\ww_{\Ll [35]}
u_{\Ll [2]}
\wh{\mc{X}}_{[25]}u_{\Ll [2]}
\wh{\mc{X}}_{[45]}
\ww_{\mathbb{L} [13]}\ww_{\mathbb{L} [12]}\\
&=
\wh{\mc{X}}^*_{[45]}
u_{\Ll [2]}
\wh{\mc{X}}^*_{[25]}u_{\Ll [2]}
\ww_{\mathbb{L} [12]}^*
\ww^*_{\mathbb{L} [35]}\ww^*_{\mathbb{L} [15]}
\wh{\mc{X}}^*_{[25]}
\gamma_2(a_2)_{[57]}
\wh{\mc{X}}_{[25]}
\ww_{\mathbb{L} [15]}
\ww_{\mathbb{L} [35]}
\ww_{\mathbb{L} [12]}
u_{\Ll [2]}
\wh{\mc{X}}_{[25]}u_{\Ll [2]}
\wh{\mc{X}}_{[45]}.
\end{split}\]
Using the assumption that $\wh{\mc{X}}$ is an $\operatorname{R}$-matrix, we have $
\ww^*_{\mathbb{L} [12]} \ww^*_{\mathbb{L} [15]}
\wh{\mc{X}}^*_{[25]}=
\wh{\mc{X}}_{[25]}^*
\ww_{\mathbb{L} [15]}^*
\ww_{\mathbb{L} [12]}^*$, hence
\[\begin{split}
&\quad\;
(\id\otimes \Ind_{\X}(\gamma_\otimes))
\Ind_{\X}(\gamma_\otimes)_{|\mathbb{L}}
\bigl(\Upsilon(\pi_{\boxtimes,2}(a_2)) \bigr)\\
&=
\wh{\mc{X}}^*_{[45]}
u_{\Ll [2]}
\wh{\mc{X}}^*_{[25]}u_{\Ll [2]}
\ww^*_{\mathbb{L} [35]}
\wh{\mc{X}}^*_{[25]}
\ww^*_{\mathbb{L} [15]}
\ww_{\mathbb{L} [12]}^*
\gamma_2(a_2)_{[57]}
\ww_{\mathbb{L} [12]}
\ww_{\mathbb{L} [15]}
\wh{\mc{X}}_{[25]}
\ww_{\mathbb{L} [35]}
u_{\Ll [2]}
\wh{\mc{X}}_{[25]}u_{\Ll [2]}
\wh{\mc{X}}_{[45]}.
\end{split}\]
On the other hand
\[\begin{split}
&\quad\;
(\id\otimes \Ind_{\X}(\gamma_\otimes)\Upsilon)
(\gamma_1\bowtie \gamma_2)
\bigl(\pi_{\boxtimes,2}(a_2) \bigr)=
(\id\otimes \Ind_{\X}(\gamma_\otimes)\Upsilon)
(\id\otimes \pi_{\boxtimes,2})\gamma_2(a_2)\\
&=
\wh{\mc{X}}^*_{[45]}u_{\Ll [2]}
\wh{\mc{X}}^*_{[25]}u_{\Ll [2]}
\ww^*_{\Ll[35]}\wh{\mc{X}}^*_{[25]}
(\id\otimes \gamma_2)\gamma_2(a_2)_{[157]}
\wh{\mc{X}}_{[25]}
\ww_{\Ll[35]}
u_{\Ll [2]}
\wh{\mc{X}}_{[25]}u_{\Ll [2]}
\wh{\mc{X}}_{[45]}\\
&=
\wh{\mc{X}}^*_{[45]}u_{\Ll [2]}
\wh{\mc{X}}^*_{[25]}u_{\Ll [2]}
\ww^*_{\Ll[35]}\wh{\mc{X}}^*_{[25]}
\ww^*_{\mathbb{L} [15]}
\gamma_2(a_2)_{[57]}
\ww_{\mathbb{L} [15]}
\wh{\mc{X}}_{[25]}
\ww_{\Ll[35]}
u_{\Ll [2]}
\wh{\mc{X}}_{[25]}u_{\Ll [2]}
\wh{\mc{X}}_{[45]}.
\end{split}\]
Since $\id\otimes \Ind_{\X}(\gamma_\otimes)$ is injective, this proves \eqref{eq57}.

 Finally, we need to show \eqref{EqFormulaImplBraid}. Since the $\HH$-action $\Ind_{\X}(\gamma_\otimes)$ restricts to the $\Ll$-action $\gamma_1\bowtie \gamma_2$ (up to the isomorphism $\Upsilon$), we can conclude that $(U_{\Ind_{\X}(\gamma_\otimes)})_{|\Ll} \cong U_{\gamma_1\bowtie \gamma_2}$. The restricted representation is characterised by
 \begin{equation}\label{eq58}
(\id\otimes\iota_{\Ll})( \ww^*_{\mathbb{L}})_{ [123]}
U_{\Ind_{\X}(\gamma_\otimes) [2345]}
(\id\otimes\iota_{\Ll})( \ww_{\mathbb{L}})_{ [123]}
=
(U_{\Ind_{\X}(\gamma_\otimes) })_{|\mathbb{L}[145]}
U_{\Ind_{\X}(\gamma_\otimes) [2345]}.
 \end{equation}

Using the formula for $U_{\Ind_{\X}(\gamma_\otimes)}$ from Theorem \ref{TheoInductionBraided},  transported to $L^2(A_1)\otimes L^2(A_2)$,  we calculate
\[\begin{split}
&\quad\;
(\id\otimes\iota_{\Ll})( \ww^*_{\mathbb{L}})_{ [123]}
I_{\boxtimes [45]}^*
U_{\Ind_{\X}(\gamma_\otimes) [2345]}
I_{\boxtimes [45]}
(\id\otimes\iota_{\Ll})( \ww_{\mathbb{L}})_{ [123]}\\
&=
\ww_{\mathbb{L} [12]}^*
\ww_{\mathbb{L} [13]}^*
(\hat{\pi}_{U_{\gamma_1}}\otimes \hat{\pi}_{U_{\gamma_2}})( \wh{\mc{X}}_u)_{[45]}
(\hat{\pi}_{1,\red}\otimes \hat{\pi}_{U_{\gamma_2 }})(\wh{\mc{X}}_u)_{[25]}
U_{\gamma_1 [24]}
U_{\gamma_2 [35]}
\\
&\quad\quad\quad\quad\quad\quad\quad\quad\quad
\quad\quad\quad\quad\quad\quad\quad\quad\quad
\quad\quad
u_{\Ll [2]}
(\hat{\pi}_{1,\red}\otimes \hat{\pi}_{U_{\gamma_2}})
(\wh{\mc{X}}_{u})_{[25]}
u_{\Ll [2]}
\ww_{\mathbb{L} [13]}
\ww_{\mathbb{L} [12]}\\
&=
\ww_{\mathbb{L} [12]}^*
(\hat{\pi}_{U_{\gamma_1}}\otimes \hat{\pi}_{U_{\gamma_2}})( \wh{\mc{X}}_u)_{[45]}
(\hat{\pi}_{1,\red}\otimes \hat{\pi}_{U_{\gamma_2 }})(\wh{\mc{X}}_u)_{[25]}
U_{\gamma_2 [15]}
\ww_{\mathbb{L} [12]}
U_{\gamma_1 [14]}
U_{\gamma_1 [24]}
U_{\gamma_2 [35]}\\
&\quad\quad\quad\quad\quad\quad\quad\quad\quad
\quad\quad\quad\quad\quad\quad\quad\quad\quad
\quad\quad
u_{\Ll [2]}
(\hat{\pi}_{1,\red}\otimes \hat{\pi}_{U_{\gamma_2}})
(\wh{\mc{X}}_{u})_{[25]}
u_{\Ll [2]}.
\end{split}\]
 
 Observe that (see \cite{DCKr24}*{Remark 2.2}) 
\[\begin{split}
&\quad\;
\ww_{\mathbb{L} [12]}^*
(\hat{\pi}_{U_{\gamma_1}}\otimes \hat{\pi}_{U_{\gamma_2}})( \wh{\mc{X}}_u)_{[45]}
(\hat{\pi}_{1,\red}\otimes \hat{\pi}_{U_{\gamma_2 }})(\wh{\mc{X}}_u)_{[25]}
U_{\gamma_2 [15]}
\ww_{\mathbb{L} [12]}
U_{\gamma_1 [14]}
\\
&=
(\id\otimes \hat{\pi}_{1,\red}\otimes \hat{\pi}_{U_{\gamma_1}}\otimes \hat{\pi}_{U_{\gamma_2}})
\bigl(
\wW^*_{\mathbb{L}[12]}
\wh{\mc{X}}_{u [34]}
\wh{\mc{X}}_{u [24]}
\wW_{\mathbb{L} [14]} 
\wW_{\mathbb{L} [12]}
\wW_{\mathbb{L} [13]}
\bigr)_{[1245]}\\
&=
(\id\otimes \hat{\pi}_{1,\red}\otimes \hat{\pi}_{U_{\gamma_1}}\otimes \hat{\pi}_{U_{\gamma_2}})
\bigl(
\wW^*_{\mathbb{L}[12]}
\wh{\mc{X}}_{u [34]}
\wW_{\mathbb{L}  [12]} 
\wW_{\mathbb{L}  [14]} 
\wW_{\mathbb{L} [13]}
\wh{\mc{X}}_{u [24]}
\bigr)_{[1245]}\\
&=
(\id\otimes \hat{\pi}_{1,\red}\otimes \hat{\pi}_{U_{\gamma_1}}\otimes \hat{\pi}_{U_{\gamma_2}})
\bigl(
\wW_{\mathbb{L} [13]}
\wW_{\mathbb{L}  [14]} 
\wh{\mc{X}}_{u [34]}
\wh{\mc{X}}_{u [24]}
\bigr)_{[1245]}\\
&=
U_{\gamma_1 [14]}
U_{\gamma_2 [15]}
(\hat{\pi}_{U_{\gamma_1}}\otimes \hat{\pi}_{U_{\gamma_2}})(\wh{\mc{X}}_u)_{[45]}
 (\hat{\pi}_{1,\red}\otimes \hat{\pi}_{U_{\gamma_2}})(\wh{\mc{X}}_u)_{[25]}.
\end{split}\]
 Using this, we continue
\[\begin{split}
&\quad\;
(\id\otimes\iota_{\Ll})( \ww^*_{\mathbb{L}})_{ [123]}
I_{\boxtimes [45]}^*
U_{\Ind_{\X}(\gamma_\otimes) [2345]}
I_{\boxtimes [45]}
(\id\otimes\iota_{\Ll})( \ww_{\mathbb{L}})_{ [123]}\\
&=
U_{\gamma_1 [14]}
U_{\gamma_2 [15]}
(\hat{\pi}_{U_{\gamma_1}}\otimes \hat{\pi}_{U_{\gamma_2}})(\wh{\mc{X}}_u)_{[45]}
 (\hat{\pi}_{1,\red}\otimes \hat{\pi}_{U_{\gamma_2}})(\wh{\mc{X}}_u)_{[25]}
U_{\gamma_1 [24]}
U_{\gamma_2 [35]}
u_{\Ll [2]}
(\hat{\pi}_{1,\red}\otimes \hat{\pi}_{U_{\gamma_2}})
(\wh{\mc{X}}_{u})_{[25]}
u_{\Ll [2]}\\
&=
(U_{\gamma_1}\otop U_{\gamma_2})_{[145]}
I_{\boxtimes [45]}^*
U_{\Ind_{\X}(\gamma_\otimes) [2345]}
I_{\boxtimes [45]}.
\end{split}\]
Equation \eqref{eq58} ends the proof. 
\end{proof}

\subsection{Application to generalized Drinfeld double}\label{SecAppGDD}

In this subsection we apply our results to obtain (to our knowledge -- new) information concerning the topological structure of (generalized) Drinfeld doubles. Let us stress that we do not make any regularity assumptions.

We keep the ``dual'' notation, to stay in agreement with the rest of the article. As in Section \ref{SecLQGFAB}, let $\G_1=(M_1,\Delta_1),\G_2=(M_2,\Delta_2)$ be LCQGs with a unitary skew bicharacter $\wh{\mc{X}}\in \hat{M}_1\bar{\otimes} \hat{M}_2$ and unitary $2$-cocycle $\hat{\Omega}=\wh{\mc{X}}_{[32]}$. Then we can consider (see Section \ref{SubSecCocBich}) the  generalized Drinfeld double
\[
D_{\wh{\mc{X}}}(\hat{\G}_1,\hat{\G}_2)=
(\hat{M}_1\bar\otimes \hat{M}_2, 
\wh{\mc{X}}_{[32]} \hat{\Delta}_{\otimes} (-)
\wh{\mc{X}}_{[32]}^*).
\]
In what follows, we will use the comultiplication $\hat{\Delta}_{\otimes,u}$ on $C^*(\G_1\times \G_2)=C_0^u(\hat{\G}_1)\otimes_{\max}C_0^u(\hat{\G}_2)$, given by $\hat{\Delta}_{\otimes,u}(x\otimes y)=\hat{\Delta}_{1,u}(x)_{[13]}\hat{\Delta}_{2,u}(y)_{[24]}$, see Lemma \ref{lemma2}.

\begin{Prop}\label{prop3}
There is an isomorphism $C_0^u( D_{\wh{\mc{X}}}(\hat{\G}_1,\hat{\G}_2) )\cong C^u_0(\hat{\G}_1)\otimes_{\max} C^u_0(\hat{\G}_2)$, under which:
\begin{enumerate}
\item the comultiplication on $C_0^u( D_{\wh{\mc{X}}}(\hat{\G}_1,\hat{\G}_2) )$ corresponds to $\wh{\mc{X}}_{u[32]} \hat{\Delta}_{\otimes,u}(-) \wh{\mc{X}}_{u[32]}^*$,
\item the reducing map of $D_{\wh{\mc{X}}}(\hat{\G}_1,\hat{\G}_2)$ corresponds to $(\hat{\pi}_{1,\red}\otimes\hat{\pi}_{2,\red})\hat{\theta}$, where $\hat{\theta}$ is the tensor product reducing map $C^u_0(\hat{\G}_1)\otimes_{\max} C^u_0(\hat{\G}_2)\rightarrow C^u_0(\hat{\G}_1)\otimes C^u_0(\hat{\G}_2)$,
\item the counit of $D_{\wh{\mc{X}}}(\hat{\G}_1,\hat{\G}_2)$ corresponds to $\eps_{\hat{\G}_1}\times \eps_{\hat{\G}_2}$.
\end{enumerate}
Furthermore, $C_0 (D_{\wh{\mc{X}}}(\hat{\G}_1,\hat{\G}_2))=C_{0}(\hat{\G}_1)\otimes C_{0}(\hat{\G}_2)$.
\end{Prop}

\begin{Rem}\noindent
\begin{itemize}
\item The equality $C_0 (D_{\wh{\mc{X}}}(\hat{\G}_1,\hat{\G}_2))=C_{0}(\hat{\G}_1)\otimes C_{0}(\hat{\G}_2)$ at the reduced level is known from  \cite{MNW03}, see in particular Section $8$ and Theorem $8.7$ therein. See also \cite{Roy15}*{Theorem 5.3} for a more general result in the setting of multiplicative unitaries.
\item One can express the unitary antipode on $C_0^u(D_{\wh{\mc{X}}}(\hat{\G}_1,\hat{\G}_2))$ using the universal analog of \eqref{eq12} and the element $X_{\hat{\Omega},u}$ (see Lemma \ref{LemUniqueUnitCob} and Lemma \ref{LemProjCobPair}).
\end{itemize}
\end{Rem}

\begin{proof}[Proof of Proposition \ref{prop3}]
Theorem \ref{TheoUnivLift} states that $\hat{\Omega}_u=\wh{\mc{X}}_{u[32]}$ is a universal lift of $\hat{\Omega}=\hat{\mc{X}}_{[32]}$ (in the sense of Definition \ref{DefUnivLift2Coc}). This immediately gives us an isomorphism $C_0^u( D_{\wh{\mc{X}}}(\hat{\G}_1,\hat{\G}_2) )\cong C^u_0(\hat{\G}_1)\otimes_{\max} C^u_0(\hat{\G}_2)$ for which points $(2),(3)$ hold, and it follows that $C_0 (D_{\wh{\mc{X}}}(\hat{\G}_1,\hat{\G}_2))=C_{0}(\hat{\G}_1)\otimes C_{0}(\hat{\G}_2)$. Next, denote by $\msG$ the linking quantum groupoid constructed from $\hat{\Omega}$ as in Section \ref{SecCocyc} and let $\hat{\Delta}_{u,ij}$ be the ingredients of the comultiplication on $C^*(\msG)$. In particular, we have $D_{\wh{\mc{X}}}(\hat{\G}_1,\hat{\G}_2)$ and $\hat{\G}_1\times \hat{\G}_2$ as two diagonal corners. Point $(1)$ follows from the calculation
\[
\Delta_{D_{\wh{\mc{X}}}(\hat{\G}_1,\hat{\G}_2),u}(z)=
\hat{\Delta}_{11,u}(z)=
\hat{\Delta}_{12,u}(1)
\hat{\Delta}_{22,u}(z)
\hat{\Delta}_{12,u}(1)^*=
\hat{\Omega}_u\hat{\Delta}_{\otimes,u}(z)
\hat{\Omega}_{u}^*=
\wh{\mc{X}}_{[32]}\hat{\Delta}_{\otimes,u}(z)
\wh{\mc{X}}_{[32]}^*,
\]
where $z\in C_0^u(\hat{\G}_1)\otimes_{\max} C_0^u(\hat{\G}_2)$ and where $\Delta_{D_{\wh{\mc{X}}}(\hat{\G}_1,\hat{\G}_2),u}$ is the universal comultiplication on $D_{\wh{\mc{X}}}(\hat{\G}_1,\hat{\G}_2)$, transported to $C_0^u(\hat{\G}_1)\otimes_{\max} C_0^u(\hat{\G}_2)$.
\end{proof}

\subsection{Proof of Theorem \ref{TheoUnivLift}}\label{SecProofTech}

To prove Theorem \ref{TheoUnivLift}, we start with some general preparations.

\begin{Lem}\label{lemma3}
Let $\Hh$ be a locally compact quantum group, with right half-lifted multiplicative unitary $\wW \in \Mult(C_0(\Hh)\otimes C^*(\Hh))$. Then the dual coproduct lifts uniquely to a non-degenerate $*$-homomorphism
\[
\hat{\Delta}_u^{\max}\colon C^*(\Hh) \rightarrow \Mult(C^*(\Hh)\otimes_{\max}C^*(\Hh))
\]
such that 
\begin{equation}\label{EqFormulaMaxCo}
\hat{\Delta}_u^{\max}((\omega \otimes \id)\wW) = (\omega \otimes \id\otimes \id)(\wW_{[13]}\wW_{[12]}), \qquad \forall \omega \in L^1(\Hh).
\end{equation}
\end{Lem} 
Here we consider the product of $\wW_{[13]}$ and $\wW_{[12]}$ in $\Mult(C_0(\Hh) \otimes (C^*(\Hh)\otimes_{\max}C^*(\Hh)))$. 
\begin{proof}
We recall that by its universal property, non-degenerate $*$-representations $\hat{\pi}$ of $C^*(\Hh)\otimes_{\max}C^*(\Hh)$ on a Hilbert space $\Hsp$ are in one-to-one correspondence with non-degenerate $*$-representations $\hat{\pi}_1,\hat{\pi}_2$ of $C^*(\Hh)$ on $\Hsp$ which pointwise commute. This correspondence is given by
\[
\hat{\pi} = \hat{\pi}_1 \times \hat{\pi}_2,\qquad (\hat{\pi}_1\times\hat{\pi}_2)(x\otimes y) = \hat{\pi}_1(x)\hat{\pi}_2(y),\qquad \forall x,y\in C^*(\Hh).
\]
Such $*$-representations $\hat{\pi}_1,\hat{\pi}_2$ are implemented by commuting unitary representations $U_{\hat{\pi}_1},U_{\hat{\pi}_2}$ of $\Hh$ on $\Hsp_{\hat{\pi}}$, 
\[
U_{\hat{\pi}_i} \in L^{\infty}(\Hh) \bar{\otimes} \mcB(\Hsp_{\hat{\pi}}),\qquad U_{\hat{\pi}_1[13]}U_{\hat{\pi}_2[23]} = U_{\hat{\pi}_2[23]}U_{\hat{\pi}_1[13]}.
\]
By the latter commutation relation, we have that 
\[
U_{\hat{\pi}_2\hat{\pi}_1} := U_{\hat{\pi}_2}U_{\hat{\pi}_1}\in L^{\infty}(\Hh)\bar{\otimes} \mcB(\Hsp_{\hat{\pi}})
\]
is a unitary representation of $\Hh$ too, and is hence associated to a non-degenerate $*$-representation \cite{Kus01}*{Proposition 6.5}
\[
\hat{\pi}_2\hat{\pi}_1\colon C^*(\Hh) \rightarrow \mcB(\Hsp_{\hat{\pi}}). 
\]
As 
\[
U_{\hat{\pi}_i} \in \Mult(C_0(\Hh)\otimes \hat{\pi}(C^*(\Hh)\otimes_{\max}C^*(\Hh))),
\]
we deduce that $\hat{\pi}_2\hat{\pi}_1$ is a non-degenerate $*$-homomorphism into $\Mult(\hat{\pi}(C^*(\Hh)\otimes_{\max}C^*(\Hh)))$. Taking for $\hat{\pi}$ a universal non-degenerate $*$-representation of $C^*(\Hh)\otimes_{\max}C^*(\Hh)$ with the associated $\hat{\pi}_1,\hat{\pi}_2$, we can then define $\hat{\Delta}_u^{\max}$ as the unique non-degenerate $*$-homomorphism $C^*(\Hh) \rightarrow \Mult(C^*(\Hh)\otimes_{\max}C^*(\Hh))$  such that 
\[
\hat{\pi}_2\hat{\pi}_1 = \hat{\pi} \circ \hat{\Delta}_u^{\max}.
\]
We then have by construction that 
\begin{equation}\begin{split}\label{eq13}
&\quad\;
(\id\otimes \hat{\pi} \hat{\Delta}_u^{\max})(\wW) = (\id\otimes \hat{\pi}_2\hat{\pi}_1)\wW 
= U_{\hat{\pi}_2}U_{\hat{\pi}_1} = (\id\otimes \hat{\pi}_2)(\wW)(\id\otimes \hat{\pi}_1)(\wW)\\
& = (\id\otimes (\hat{\pi}_1\times\hat{\pi}_2))(\wW_{[13]}\wW_{[12]})
= (\id\otimes \hat{\pi})(\wW_{[13]}\wW_{[12]}),
\end{split}\end{equation}
so that \eqref{EqFormulaMaxCo} is satisfied. In particular, $\hat{\Delta}_u^{\max}$ does not depend on the choice of $\hat{\pi}$. As 
\[
C^*(\Hh) = [(\omega \otimes \id)\wW\mid \omega \in L^1(\Hh)],
\] 
it is clear that this latter formula completely determines $\hat{\Delta}_u^{\max}$. 
\end{proof}

It is easy to see from the above construction, or directly from \eqref{EqFormulaMaxCo}, that $\hat{\Delta}_u^{\max}$ is still coassociative in the appropriate way:
\[
(\id\otimes_{\max} \hat{\Delta}_u^{\max})\hat{\Delta}_u^{\max} = (\hat{\Delta}_u^{\max}\otimes_{\max} \id)\hat{\Delta}_u^{\max}. 
\] 

In the following, we will also use this result for C$^*$-algebras associated to linking quantum groupoids (the proof is completely the same). 

The following lemma is just an easy observation on linking C$^*$-algebras.

\begin{Lem}\label{LemMorita}
Let $A = \begin{pmatrix} A_{11} & A_{12} \\ A_{21} & A_{22}\end{pmatrix}$ be a linking C$^*$-algebra, i.e.\ $[A_{ij}A_{jk}] = A_{ik}$ for all indices $i,j,k\in \{1,2\}$. Assume that $B$ is another linking C$^*$-algebra, and let 
\[
\pi\colon A \rightarrow \Mult(B)
\]
be a non-degenerate $*$-homomorphism preserving the projection $e_1 = \begin{pmatrix} 1 & 0 \\ 0 & 0 \end{pmatrix}\in \Mult(A)$ and hence also $e_2 = 1-e_1$. Write 
\[
\pi_{ij} = \pi_{\mid A_{ij}}\colon A_{ij} \rightarrow \Mult(B_{ij}) := e_i \Mult(B) e_j.
\]
Then if 
\[
A_{22} \cong_{\pi_{22}} B_{22},
\]
we have that $\pi$ is an isomorphism $A\cong_{\pi}B$.  
\end{Lem} 

\begin{proof}
We have 
\[
\pi_{12}(A_{12})B_{22} = \pi_{12}(A_{12})\pi_{22}(A_{22}) = \pi_{12}(A_{12}A_{22}),
\]
so we see that in fact $\pi_{12}(A_{12}) \subseteq B_{12}$. It follows easily from this that $\pi(A) \subseteq B$. Moreover, since 
\[
\pi_{12}(x)^*\pi_{12}(x) = \pi_{22}(x^*x),\qquad \forall x\in A_{12},
\]
we see that $\pi_{12}$ is injective, and then is $\pi_{21}$. To see that also $\pi_{11}$ is injective, take $x\in \Ker(\pi_{11})$ and notice that then $\pi_{12}(xA_{12})=\pi_{11}(x)\pi_{12}(A_{12})=\{0\}$, hence $xA_{12} = \{0\}$. Consequently $xA_{11}=\{0\}$ and $x =0$. In particular, $\pi$ is injective. 

Finally, to see that $\pi$ is surjective, it is sufficient to show that $\pi_{12}$ is surjective. Take a bounded net $(x_i)_{i\in I}$ in $A_{11}$ that strictly converges to $1$. We may assume each $x_i$ is a finite sum $\sum_{j=1}^{N_i} x_{ij}x_{ij}^*$ with $x_{ij} \in A_{12}$. Then as $\pi$ is non-degenerate, also $\bigl(\sum_{j=1}^{N_i} \pi_{12}(x_{ij})\pi_{12}(x_{ij})^*\bigr)_{i\in I}$ converges strictly to $1_{11}\in \Mult(B_{11})$, and hence
\[
y  = \lim_{i\in I} \sum_{j=1}^{N_i} \pi_{12}(x_{ij}) (\pi_{12}(x_{ij})^*y) \in [\pi_{12}(A_{12})B_{22}] \subseteq \pi_{12}(A_{12})
\]
for any $y\in B_{12}$. This shows $\pi_{12}$ is surjective. Finally,
\[
B_{11}=[B_{12} B_{21}]=[\pi_{12}(A_{12})\pi_{21}(A_{21})]= \pi_{11}(A_{11}).
\]
Note that we have used that $\pi_{11}$ and $\pi_{12}$ have closed image. 
\end{proof}

Let us now return to the setting of Theorem \ref{TheoUnivLift}.

We denote by $\msG$ the linking quantum groupoid associated to $\G = \G_1\times \G_2$ with respect to its dual unitary $2$-cocycle $\hat{\Omega} = \wh{\mcX}_{[32]}$ (see Section \ref{SubSecCocBich}). Recall that the universal linking W$^*$-algebra associated to $\msG$ is
\[
W^*(\msG) = \begin{pmatrix} W^*(\G_{11}) & W^*(\G_{12}) \\
W^*(\G_{21}) & W^*(\G_{22})\end{pmatrix}.
\]
We can $*$-represent all the components of $L^{\infty}(\msG)$ on the same Hilbert space, 
\[
L^{\infty}(\G_{ij}) \subseteq \mcB(L^2(\G_1)\otimes L^2(\G_2)). 
\]

Now for $\G_i$, we consider the trivial linking quantum groupoid $\msG_i = \G_i^{(2)}$, which is just the trivial linking W$^*$-algebra $\begin{pmatrix} \hat{M}_i & \hat{M}_i\\ \hat{M}_i& \hat{M}_i\end{pmatrix}$ with $\hat{\Delta}_i$ at all components. Its dual is the direct sum von Neumann algebra
\[
L^{\infty}(\G_i^{(2)}) = \bigoplus_{k,l=1}^2 L^{\infty}(\G_i)
\]
with $\Delta_{ab}^c = \Delta_i$ at all components. We call $\G_i^{(2)}$ the \emph{amplification} of $\G_i$.

The following lemma will show that $\hat{\G}_{i}^{(2)}$ is a closed linking quantum subgroupoid of $\hat{\msG}$. Note that we already used this observation in the proof of Lemma \ref{LemProjCobPair}, but this relied on the fact that we knew the cocycle twist of $C^*(\G)$  provided a linking quantum groupoid, making use of Theorem \ref{TheoUnivLift}.

\begin{Lem}\label{LemmaInclusionSubgroups}
For $i\in\{1,2\}$ we have a normal, unital embedding
\[
\pi^{(i)}\colon L^{\infty}(\G_i^{(2)}) \rightarrow L^{\infty}(\msG)
\]
which splits into $\pi^{(i)}_{kl}\colon L^{\infty}(\G_i) \rightarrow L^{\infty}(\G_{kl})$, given by 
\begin{equation}\label{EqInclusions}
\pi^{(1)}_{kl}(x_1) = x_1\otimes 1,\qquad \pi^{(2)}_{kl}(x_2) = \wh{\mcX}_{kl}^*(1\otimes x_2)\wh{\mcX}_{kl},\qquad \forall x_i \in L^{\infty}(\G_i),
\end{equation}
where 
\[
\wh{\mcX}_{22} = 1\otimes 1,\quad \wh{\mcX}_{12} = \wh{\mcX},\quad \wh{\mcX}_{21} = (\hat{J}_1\otimes \hat{J}_2)\wh{\mcX}(\hat{J}_1\otimes \hat{J}_2),\quad \wh{\mcX}_{11}= \wh{\mcX}_{12}\wh{\mcX}_{21}. 
\]
The maps $\pi^{(i)},\pi^{(i)}_{kl}$ respect coproducts, in the sense that
\[
\Delta^a_{kl}\pi^{(i)}_{kl}=
(\pi^{(i)}_{ka}\otimes \pi^{(i)}_{al})\Delta_i,\quad 
\Delta_{L^{\infty}(\msG)}\pi^{(i)}=
(\pi^{(i)}\otimes \pi^{(i)})\Delta_{L^{\infty}(\G_i^{(2)})}
\]
for $i,a,k,l\in\{1,2\}$. Furthermore
\[
L^{\infty}(\G_{kl}) = \overline{\operatorname{span}}^{\,\sigma\textrm{-weak}}\,\{\pi^{(1)}_{kl}(x_1)\pi^{(2)}_{kl}(x_2) \mid x_i \in L^{\infty}(\G_i)\}.
\]
\end{Lem} 

\begin{proof}
Recall that $J_N=(J_1\otimes J_2)\wh{\mc{X}}$ (Lemma \ref{lemma1}). Simplifying the expressions \eqref{EqMultUnitCocy2} in this specific case, one can compute that the components of $\hat{\ww}$, the left regular multiplicative partial isometry for $W^*(\msG)$, are given by
\begin{equation}
\label{eq16}
\hat{\ww}_{kl}^a = \hat{\ww}_{1[13]}\wh{\mcX}_{kl[34]}^*\hat{\ww}_{2[24]}\wh{\mc{X}}_{kl[34]}=
(\id\otimes \pi^{(1)}_{kl})(\hat{\ww}_{1})_{[134]}
(\id\otimes \pi^{(2)}_{kl})(\hat{\ww}_{2})_{[234]}
\end{equation}
for $a,k,l\in\{1,2\}$ (note that the formula for $\hat{\ww}^a_{11}$ gives the form of the multiplicative unitary of the generalized quantum double as given in  \cite{BV05}*{Theorem 5.3}). As the $L^{\infty}(\G_{kl})$ are obtained by slicing the first leg of $\hat{\ww}_{kl}$, the last statement of the lemma follows immediately from \eqref{eq16}. Applying $\Delta_{kl}^a$ to legs $3,4$ of \eqref{eq16} and using \eqref{eq17} gives us
\begin{equation}\begin{split}\label{eq18}
&\quad\;
(\id\otimes \Delta^a_{kl}\pi^{(1)}_{kl})(\hat{\ww}_{1})_{[13456]}
(\id\otimes \Delta^a_{kl}\pi^{(2)}_{kl})(\hat{\ww}_{2})_{[23456]}
=
(\id\otimes\id\otimes \Delta^a_{kl})(\hat{\ww}^a_{kl})=
\hat{\ww}^a_{al [1256]} \hat{\ww}^a_{ka [1234]}\\
&=
(\id\otimes \pi^{(1)}_{al})(\hat{\ww}_{1})_{[156]}
(\id\otimes \pi^{(2)}_{al})(\hat{\ww}_{2})_{[256]}
(\id\otimes \pi^{(1)}_{ka})(\hat{\ww}_{1})_{[134]}
(\id\otimes \pi^{(2)}_{ka})(\hat{\ww}_{2})_{[234]}\\
&=
(\id\otimes \pi^{(1)}_{al})(\hat{\ww}_{1})_{[156]}
(\id\otimes \pi^{(1)}_{ka})(\hat{\ww}_{1})_{[134]}
(\id\otimes \pi^{(2)}_{al})(\hat{\ww}_{2})_{[256]}
(\id\otimes \pi^{(2)}_{ka})(\hat{\ww}_{2})_{[234]}\\
&=
(\id\otimes \pi^{(1)}_{ka}\otimes\pi^{(1)}_{al})(\id\otimes \Delta_1)(\hat{\ww}_{1})_{[13456]}
(\id\otimes \pi^{(2)}_{ka}\otimes \pi^{(2)}_{al})
(\id\otimes \Delta_2)(\hat{\ww}_{2})_{[23456]},
\end{split}\end{equation}
where $\Delta_1,\Delta_2$ are the comultiplications of respectively $\G_1$ and $\G_2$. Since there is a net of normal functionals $\omega_\lambda\in L^1(\hat{\G}_2)$ such that $(\omega_\lambda\otimes\id)(\hat{\ww}_2)\to 1$ weak$^*$ in $L^{\infty}(\G_2)$, equation \eqref{eq18} gives us $\Delta^a_{kl}\pi^{(i)}_{kl}=(\pi^{(i)}_{ka}\otimes \pi^{(i)}_{al})\Delta_i$ for $i=1$. The case $i=2$ follows in a similar way. Then the identity 
\[
\Delta_{L^{\infty}(\msG)}\pi^{(i)}=
(\pi^{(i)}\otimes \pi^{(i)})\Delta_{L^{\infty}(\G_i^{(2)})}
\] 
is immediate.
\end{proof}

Lemma \ref{lemma3} in the case of $\Hh=\G_1\times \G_2$ gives us a lift of the comultiplication $\hat{\Delta}_{\otimes,u}$ on $C^*(\G_1\times \G_2)$ to a non-denegerate $*$-homomorphism
\[
\hat{\Delta}_{\otimes,u}^{\max}\colon C^*(\G_1 \times \G_2)\rightarrow \Mult( 
C^*(\G_1 \times \G_2)\otimes_{\max}C^*(\G_1 \times \G_2)).
\]

On the other hand, we can also consider the composition
\[\begin{split}
&\quad\;
\hat{\Delta}^{\max}_{u,\otimes}\colon 
C^*(\G_1 \times \G_2)=C^*(\G_1)\otimes_{\max} C^*(\G_2)\\
&\xrightarrow[]{\hat{\Delta}^{\max}_{1,u}\otimes_{\max} \hat{\Delta}^{\max}_{2,u}}
\Mult(C^*(\G_1)\otimes_{\max} C^*(\G_1))
\otimes_{\max} \Mult( C^*(\G_2)\otimes_{\max} C^*(\G_2))\\
&\longrightarrow
\Mult( C^*(\G_1)\otimes_{\max} C^*(\G_1)\otimes_{\max} C^*(\G_2)\otimes_{\max} C^*(\G_2))\\
&\cong
\Mult( C^*(\G_1)\otimes_{\max} C^*(\G_2)\otimes_{\max} C^*(\G_1)\otimes_{\max} C^*(\G_2)),
\end{split}\]
where the above arrow without label indicates the canonical non-degenerate $*$-homomorphism 
\[
\Mult(A)\otimes_{\max} \Mult(B)\rightarrow \Mult(A\otimes_{\max} B),
\] 
and where the isomorphism in the last line is the extension of the flip map 
\[
a\otimes b\otimes c \otimes d\mapsto a\otimes c\otimes b\otimes d.
\] 
Recalling that $\wW_{\G_1\times \G_2}=\wW_{\G_1 [13]} \wW_{\G_2 [24]}$ (Lemma \ref{lemma2}), we immediately conclude:

\begin{Lem}\label{lemma4}
The non-degenerate $*$-homomorphisms $\hat{\Delta}^{\max}_{\otimes,u}$ and $\hat{\Delta}^{\max}_{u,\otimes }$ are equal.
\end{Lem}

\begin{Lem}\label{LemTrivialQG}
The linking quantum groupoid $\msG$, built using the unitary $2$-cocycle $\hat{\Omega}$ for $(\hat{M}_1\bar{\otimes}\hat{M}_2,\hat{\Delta}_{\otimes})$, has trivialisable universal linking C$^*$-algebra. 
\end{Lem} 

\begin{proof}
Let us write the universal linking C$^*$-algebra associated to $\msG$ as 
\[
C^*(\msG) = \begin{pmatrix} C^*(\Hh) & C^*(\X) \\
 C^*(\Y)  & C^*(\G)
\end{pmatrix}
\]
(that $C^*(\msG)$ has this form follows directly from the definition of the universal quantum groupoid C$^*$-algebra associated to a linking quantum groupoid, see e.g.\ \cite{DC09}*{Theorem 7.6.4}). A linking quantum groupoid version of Lemma \ref{lemma3} gives a lift of $\Delta_{C^*(\msG)}$ to a $*$-homomorphism
\[
\Delta_{C^*(\msG)}^{\max}\colon C^*(\msG)\rightarrow 
\Mult(C^*(\msG)\otimes_{\max}C^*(\msG))
\]
which respects matrix decompositions. We can identify 
\[
C^*(\G) = C^*(\G_1)\otimes_{\max} C^*(\G_2)
\]
as C$^*$-algebras (Lemma \ref{lemma2}).  But since
\[
L^{\infty}(\G_i^{(2)}) \subseteq L^{\infty}(\msG)
\]
via an inclusion map which respects coproducts (Lemma \ref{LemmaInclusionSubgroups}), we also have surjective, coproduct-preserving C$^*$-algebra morphisms
\[
\hat{\pi}_i\colon C^*(\msG) \twoheadrightarrow \begin{pmatrix} C^*(\G_i) & C^*(\G_i) \\ C^*(\G_i)& C^*(\G_i)\end{pmatrix}
\]
which respect the matrix structure. Using that the $22$-parts of the half-lifted Kac-Takesaki operators for $\hat\msG$ and $\widehat{\G_i^{(2)}}$ agree with the corresponding operators for $\hat\G$ and $\hat\G_i$, we see that $\hat{\pi}_{1,22}=\id\times \hat{\eps}_2$ and $\hat{\pi}_{2,22}=\hat{\eps}_{1}\times\id$. We obtain the non-degenerate C$^*$-algebra morphism
\begin{equation}\label{eq19}
\hat{\pi} = (\hat{\pi}_1\otimes_{\max}\hat{\pi}_2)\Delta_{C^*(\msG)}^{\max}\colon C^*(\msG) \rightarrow \Mult\begin{pmatrix} C^*(\G_1) \otimes_{\max} C^*(\G_2) & C^*(\G_1) \otimes_{\max} C^*(\G_2)\\ C^*(\G_1) \otimes_{\max} C^*(\G_2)& C^*(\G_1) \otimes_{\max} C^*(\G_2)\end{pmatrix}.
\end{equation}
Note that the codomain of $\hat{\pi}$ is correct, because $\hat{\pi}_i$ and $\Delta^{\max}_{C^*(\msG)}$ preserve the matrix decomposition. Using Lemma \ref{lemma4} we conclude that $\hat{\pi}$ restricts to the identity map on the lower right hand corner. Then it follows from Lemma \ref{LemMorita} that $\hat{\pi}$ must be an isomorphism 
\begin{equation}\label{EqIsoMatrix}
C^*(\msG) \cong_{\hat{\pi}} \begin{pmatrix} C^*(\G_1) \otimes_{\max} C^*(\G_2) & C^*(\G_1) \otimes_{\max} C^*(\G_2)\\ C^*(\G_1) \otimes_{\max} C^*(\G_2)& C^*(\G_1) \otimes_{\max} C^*(\G_2)\end{pmatrix}.
\end{equation}
\end{proof}

To prove Theorem \ref{TheoUnivLift}, it remains to show that the particular element $\hat{\Omega}_u$ is a universal lift of $\hat{\Omega}$. This will again require some preparations. 

Note that in the proof of Theorem \ref{LemTrivialQG}, we obtained in particular an isomorphism of right Hilbert C$^*$-modules over $C^*(\G_1)\otimes_{\max} C^*(\G_2)$: 
\begin{equation}\label{EqHatPi12}
\hat{\pi}_{12} \colon C^*(\G_{12}) \cong C^*(\G_1)\otimes_{\max} C^*(\G_2). 
\end{equation}

\begin{Rem}
For $k\in\{1,2\}$, let
\[
\vp^{(k)}=(\vp_{ij}^{(k)})_{i,j=1}^{2}\colon A^{(k)}=\begin{pmatrix}
A_{11}^{(k)} & A_{12}^{(k)} \\ A_{21}^{(k)} & A_{22}^{(k)}
\end{pmatrix}\rightarrow \Mult(B^{(k)})=\Mult \biggl(\! \begin{pmatrix} B_{11}^{(k)} & B_{12}^{(k)} \\ B_{21}^{(k)} & B_{22}^{(k)} \end{pmatrix} \!\biggr)
\]
be a non-degenerate $*$-homomorphism between C$^*$-algebras which preserves the matrix structure. Then the tensor product 
\[
\vp^{(1)}\otimes_{\max}\vp^{(2)}\colon A^{(1)}\otimes_{\max} A^{(2)}\rightarrow \Mult(B^{(1)}\otimes_{\max} B^{(2)})
\] 
can be (co-)restricted to a map 
\[
\vp^{(1)}_{12}\otimes_{\max} \vp^{(2)}_{12}\colon A^{(1)}_{12}\otimes_{\max} A^{(2)}_{12}\rightarrow \Mult( B^{(1)}_{12}\otimes_{\max} B^{(2)}_{12}),
\] 
where $A^{(1)}_{12}\otimes_{\max} A^{(2)}_{12}$ is a corner in $A^{(1)}\otimes_{\max} A^{(2)}$ (and similarly for $B$). In particular, we will consider the isomorphism
\[
\hat{\pi}_{12} = (\hat{\pi}_{1,12}\otimes_{\max}\hat{\pi}_{2,12})\hat{\Delta}_{u,12}^{\max}.
\]
Note that if $A^{(k)}_{ij}=C^{(k)}$ for some C$^*$-algebra $C^{(k)}$ and all $i,j\in \{1,2\}$ (so that $A^{(k)}=C^{(k)}\otimes M_2$), then because the matrix algebras are nuclear, $A^{(1)}_{12}\otimes_{\max} A^{(2)}_{12}$ agrees with the usual maximal tensor product of C$^*$-algebras $C^{(1)}\otimes_{\max} C^{(2)}$.
\end{Rem}

We can use the map $\hat{\pi}_{12}$ from \eqref{EqHatPi12} to transport $\hat{\Delta}_{u,12}$ to a map 
\[
\widetilde{\Delta}_{u,12} := (\hat{\pi}_{12}\otimes \hat{\pi}_{12})\hat{\Delta}_{u,12}\hat{\pi}_{12}^{-1}\colon C^*(\G_1)\otimes_{\max} C^*(\G_2)\rightarrow \Mult((C^*(\G_1)\otimes_{\max} C^*(\G_2))\otimes (C^*(\G_1)\otimes_{\max} C^*(\G_2)))
\]
and similarly $\hat{\Delta}^{\max}_{u,12}$ to
\[\begin{split}
\widetilde{\Delta}_{u,12}^{\max} := (\hat{\pi}_{12}\otimes_{\max} \hat{\pi}_{12})\hat{\Delta}_{u,12}^{\max}\hat{\pi}_{12}^{-1}&\colon C^*(\G_1)\otimes_{\max} C^*(\G_2)\\
&\rightarrow \Mult((C^*(\G_1)\otimes_{\max} C^*(\G_2))\otimes_{\max} (C^*(\G_1)\otimes_{\max} C^*(\G_2))).
\end{split}\]
On the other hand, we can also consider the quotient maps 
\[
\widetilde{\pi}_i = \hat{\pi}_{i,12}\circ \hat{\pi}^{-1}_{12}\colon C^*(\G_1)\otimes_{\max} C^*(\G_2) \rightarrow C^*(\G_i). 
\]

\begin{Lem}\label{LemTrivY}
We have
\[
\widetilde{\pi}_1 = \id_{C^*(\G_1)}\times \hat{\varepsilon}_2,\quad \widetilde{\pi}_2 = \hat{\varepsilon}_1\times \id_{C^*(\G_2)},\quad 
(\widetilde{\pi}_1\otimes_{\max} \widetilde{\pi}_2)\widetilde{\Delta}_{u,12}^{\max} = \id,\quad
(\widetilde{\pi}_1\otimes \widetilde{\pi}_2)\widetilde{\Delta}_{u,12} = \hat{\theta},  
\]
where $\hat{\theta}\colon C^*(\G_1)\otimes_{\max} C^*(\G_2)\rightarrow C^*(\G_1)\otimes C^*(\G_2)$ is the canonical quotient map.
\end{Lem} 

\begin{proof}
We have 
\[
\widetilde{\pi}_1\hat{\pi}_{12}=
\hat{\pi}_{1,12}
=
(\hat{\pi}_{1,12}\otimes_{\max} \hat{\varepsilon}_{12}) \hat{\Delta}_{u,12}^{\max}
=
 (\hat{\pi}_{1,12}\otimes_{\max} \hat{\varepsilon}_2\hat{\pi}_{2,12}) \hat{\Delta}_{u,12}^{\max}=
 (\id\times\hat{\eps}_2)\hat{\pi}_{12}
\]
where we used that universal linking quantum groupoids admit counits \eqref{EqUnivCounit}, which factorize through the counit of any quotient linking quantum groupoid. A similar computation works for $\widetilde{\pi}_2$. The next identity is just a formal computation:
\[
(\widetilde{\pi}_1\otimes_{\max} \widetilde{\pi}_2)\widetilde{\Delta}_{u,12}^{\max}=
(\hat{\pi}_{1,12}\hat{\pi}_{12}^{-1}\otimes_{\max}
\hat{\pi}_{2,12}\hat{\pi}_{12}^{-1})
(\hat{\pi}_{12}\otimes_{\max} \hat{\pi}_{12})\hat{\Delta}_{u,12}^{\max}\hat{\pi}_{12}^{-1}=
(\hat{\pi}_{1,12}\otimes_{\max}
\hat{\pi}_{2,12})
\hat{\Delta}_{u,12}^{\max}\hat{\pi}_{12}^{-1}=\id.
\]
The last equality follows by composing both sides with $\hat{\theta}$.
\end{proof}

\begin{Lem}\label{lemma5}
For $i,j\in \{1,2\}$ we have $(\hat{\eps}_1\times \hat{\eps}_2)\hat{\pi}_{ij}=\hat{\eps}_{ij}$.
\end{Lem}

\begin{proof}
The claim is trivial for $i=j=2$. Consider now $i=1,j=2$. Denote by $\hat{\eps}_{\G_1^{(2)},ij}$ the counit of $\hat{\G}_1^{(2)}$. Using Lemma \ref{LemTrivY} we find
\[
\hat{\eps}_1\times \hat{\eps}_2=\hat{\eps}_1\widetilde{\pi}_1=
\hat{\eps}_1 \hat{\pi}_{1,12} \hat{\pi}^{-1}_{12}=
\hat{\eps}_{\G_{1}^{(2)},12} \hat{\pi}_{1,12} \hat{\pi}^{-1}_{12}=
\hat{\eps}_{12}\hat{\pi}^{-1}_{12}
\]
which proves $(\hat{\eps}_1\times \hat{\eps}_2)\hat{\pi}_{12}=\hat{\eps}_{12}$. The case $i=2,j=1$ follows by taking adjoints at  the matrix level. Finally, for $x\in C^*(\G_{11})$ of the form $x=yz$ with $y\in C^*(\G_{12}),z\in C^*(\G_{21})$ we have
\[
\hat{\eps}_{11}(x)=
\hat{\eps}_{12}(y)\hat{\eps}_{21}(z)=
(\hat{\eps}_1\times \hat{\eps}_2)(\hat{\pi}_{12}(y)
\hat{\pi}_{21}(z))=
(\hat{\eps}_1\times \hat{\eps})\hat{\pi}_{11}(x),
\]
and the claim follows as $C^*(\G_{12})C^*(\G_{21})$ is linearly dense in $C^*(\G_{11})$.
\end{proof}

Write now 
\[
\widetilde{\Omega}_u = \widetilde{\Delta}_{u,12}(1\otimes 1) \in \Mult((C^*(\G_1)\otimes_{\max} C^*(\G_2))\otimes (C^*(\G_1)\otimes_{\max} C^*(\G_2))).
\]

\begin{Lem}\label{LemUnitSkewBich}
There exists a unitary skew bicharacter $\widetilde{\mcX}_u \in \Mult(C^*(\G_1) \otimes C^*(\G_2))$ such that 
\[
\widetilde{\Omega}_u = \widetilde{\mcX}_{u[32]}. 
\]
\end{Lem} 

\begin{proof}
First note that we can also write 
\[
\widetilde{\Delta}^{\max}_{u,12}(z)=(\hat{\pi}_{12}\otimes_{\max} \hat{\pi}_{12})\hat{\Delta}_{u,12}^{\max}(\hat{\pi}_{12}^{-1}(z))  = \widetilde{\Omega}_u^{\max}\hat{\Delta}_{u,\otimes}^{\max}(z),\qquad \forall z \in C^*(\G),
\]
for the unitary $2$-cocycle
\[
\widetilde{\Omega}^{\max}_u = \widetilde{\Delta}_{u,12}^{\max}(1)\in \Mult(( C^*(\G_1)\otimes_{\max} C^*(\G_2))\otimes_{\max} ( C^*(\G_1)\otimes_{\max} C^*(\G_2)))
\]
lifting $\widetilde{\Omega}_u$. Below, we will do computations with $\widetilde{\Omega}_u^{\max}$, where we still will use the `placement' indices, even though we are now interpreting them inside maximal tensor products. The $2$-cocycle condition of $\widetilde{\Omega}_u^{\max}$ is with respect to $\hat{\Delta}_{\otimes,u}^{\max}$:
\begin{equation}\label{EqCocycIdWrittenOut}
\widetilde{\Omega}_{u[1234]}^{\max}(\hat{\Delta}_{\otimes,u}^{\max}\otimes_{\max} \id \otimes_{\max} \id )(\widetilde{\Omega}_u^{\max}) = \widetilde{\Omega}_{u[3456]}^{\max}(\id \otimes_{\max} \id \otimes_{\max} \hat{\Delta}_{\otimes,u}^{\max})(\widetilde{\Omega}_u^{\max}).
\end{equation}

Put now
\[
A = (\id \otimes_{\max} \hat{\varepsilon}_2 \otimes_{\max} \id \otimes_{\max} \id)\widetilde{\Omega}_u^{\max},\qquad B = (\hat{\varepsilon}_1 \otimes_{\max} \id\otimes_{\max} \id \otimes_{\max} \id )\widetilde{\Omega}_u^{\max}. 
\]
Applying $\hat{\varepsilon}_2$ on the second component of \eqref{EqCocycIdWrittenOut} leads to 
\begin{equation}\label{EqCrucialId}
A_{[123]}(\hat{\Delta}_{1,u}^{\max}\otimes_{\max}\id
\otimes_{\max}\id \otimes_{\max} \id)(\widetilde{\Omega}_u^{\max})= \widetilde{\Omega}_{u [2345]}^{\max} (\id 
\otimes_{\max} \hat{\Delta}_{\otimes,u}^{\max})(A).
\end{equation}
But since 
\[\begin{split}
(\id \otimes_{\max} \hat{\varepsilon}_1 \otimes_{\max} \id)A = (\id \otimes_{\max} \hat{\varepsilon}_2\otimes_{\max} \hat{\varepsilon}_1 \otimes_{\max} \id)(\widetilde{\Omega}_u^{\max}) = 
(\widetilde{\pi}_1\otimes_{\max}\widetilde{\pi}_2)
\widetilde{\Delta}_{u,12}^{\max}(1)=1
\end{split}\]
by Lemma \ref{LemTrivY}, applying $(\id\otimes_{\max}\id\otimes_{\max} \hat{\varepsilon}_2\otimes_{\max} \hat{\varepsilon}_1\otimes_{\max} \id)$ to \eqref{EqCrucialId} leads to
\begin{equation}\label{eq20}
(\id\otimes_{\max}\id\otimes_{\max} \hat{\eps}_2)(A)_{[12]}=A.
\end{equation}
However, using Lemma \ref{LemTrivY} and the property that $\hat{\pi}_1$ commutes with comultiplications
\begin{equation}\begin{split}\label{eq21}
&\quad\;
(\id\otimes_{\max}\id\otimes_{\max} \hat{\eps}_2)(A)=
(\id\otimes_{\max}\hat{\eps}_2\otimes_{\max}\id\otimes_{\max} \hat{\eps}_2)\widetilde{\Omega}^{\max}_u\\
&=
(\id\otimes_{\max}\hat{\eps}_2\otimes_{\max}\id\otimes_{\max} \hat{\eps}_2)
(\hat{\pi}_{12}\otimes_{\max}\hat{\pi}_{12})
\hat{\Delta}^{\max}_{u,12}(\hat{\pi}^{-1}_{12}(1))
=
(\tilde{\pi}_1\hat{\pi}_{12}\otimes_{\max}\tilde{\pi}_1\hat{\pi}_{12})
\hat{\Delta}^{\max}_{u,12}(\hat{\pi}^{-1}_{12}(1))\\
&=
(\hat{\pi}_{1,12}\otimes_{\max}\hat{\pi}_{1,12})
\hat{\Delta}^{\max}_{u,12}(\hat{\pi}^{-1}_{12}(1))=
\hat{\Delta}^{\max}_{1,u}(\hat{\pi}_{1,12}\hat{\pi}^{-1}_{12}(1))=
\hat{\Delta}^{\max}_{1,u}(\tilde{\pi}_1(1))=
\hat{\Delta}^{\max}_{1,u}((\id\times \hat{\eps}_2)(1))=1,
\end{split}\end{equation}
which together with \eqref{eq20} gives $A=1$.
Applying $\id\otimes_{\max} \hat{\varepsilon}_1\otimes_{\max} \id^{\otimes_{\max} 3}$ on \eqref{EqCrucialId}, we find 
\[
\widetilde{\Omega}_{u}^{\max} = B_{[234]}.
\]
Put now 
\[
Z = (\id\otimes_{\max}\id\otimes_{\max} \hat{\varepsilon}_2)(B) = 
 (\hat{\varepsilon}_1
 \otimes_{\max}\id
 \otimes_{\max}\id
 \otimes_{\max} \hat{\varepsilon}_2)(\widetilde{\Omega}_u^{\max}) \in \Mult(C^*(\G_2)\otimes_{\max}C^*(\G_1)). 
\]
Then applying $\hat{\eps}_1\otimes_{\max}\id\otimes_{\max}\id\otimes_{\max} \hat{\varepsilon}_2 \otimes_{\max} \hat{\varepsilon}_1\otimes_{\max} \id$ to \eqref{EqCocycIdWrittenOut}, we find that
\[
(\id\otimes_{\max}\id\otimes_{\max}\hat{\eps}_2)(B)_{[12]} 
(\id\otimes_{\max}\hat{\eps}_1\otimes_{\max}\id )(B)_{[23]}=
B.
\]
As in \eqref{eq21} we find that $(\id\otimes_{\max}\hat{\eps}_1\otimes_{\max}\id)(B)=1$ and consequently
\[
\widetilde{\Omega}_u^{\max}= B_{[234]} = Z_{[23]}. 
\]
Considering then once again the $2$-cocycle identity \eqref{EqCocycIdWrittenOut}, applying resp.~$\id^{\otimes_{\max} 2}\otimes_{\max} \hat{\varepsilon}_1\otimes_{\max} \id^{\otimes_{\max} 3}$ or $\id^{\otimes_{\max} 3}\otimes_{\max} \hat{\varepsilon}_2\otimes_{\max} \id^{\otimes_{\max} 2}$ leads to 
\[
(\hat{\Delta}_{2,u}^{\max}\otimes_{\max}\id)Z = Z_{[23]}Z_{[13]},\qquad (\id\otimes_{\max} \hat{\Delta}_{1,u}^{\max})Z = Z_{[12]}Z_{[13]}. 
\]
If we now write $\widetilde{\mcX}_u$ for the image of $Z_{[21]}$ under the quotient map 
\[
C^*(\G_1)\otimes_{\max}C^*(\G_2)\rightarrow C^*(\G_1)\otimes C^*(\G_2),
\]
we find the statement of the lemma.
\end{proof}

To proceed, we will need to modify the isomorphism $\hat{\pi}$. In order to do this, first we establish a general lemma concerning cocycles.

\begin{Lem}\label{lemma6}
Let $\Hh_1,\Hh_2, \Hh=\Hh_1\times \Hh_2$ be locally compact quantum groups and $\mf{X},\widetilde{\mf{X}}\in L^{\infty}(\Hh_1)\bar\otimes L^{\infty}(\Hh_2)$ skew bicharacters. Denote the coproduct of $\Hh$ by $\Delta_\otimes$. Assume that cocyles $\mf{X}_{[32]},\widetilde{\mf{X}}_{[32]}\in L^{\infty}(\Hh)\bar\otimes L^{\infty}(\Hh)$ are coboundary equivalent, i.e.
\begin{equation}\label{eq32}
\mf{X}_{[32]}\Delta_{\otimes }(Z)=
(Z\otimes Z)\widetilde{\mf{X}}_{[32]}.
\end{equation}
for a unitary $Z\in L^{\infty}(\Hh)$. Then, there exist group-like unitaries $v_1\in L^{\infty}(\Hh_1),v_2\in L^{\infty}(\Hh_2)$ such that 
\[
Z=v_1\otimes v_2,\qquad \mf{X}=(v_1\otimes v_2)\widetilde{\mf{X}} (v_1^*\otimes v_2^*).
\]
\end{Lem}

\begin{proof}
We will write $\Delta_i$ for the coproduct of $\Hh_i\,(i\in\{1,2\})$. Applying the flip at legs $2,3$ of \eqref{eq32} gives
\begin{equation}\label{eq22}
\mf{X}_{[23]}(\Delta_{1}\otimes\Delta_{2 })(Z)=
Z_{[13]} Z_{[24]} \widetilde{\mf{X}}_{[23]}.
\end{equation}
If we apply $\Delta_1$ to the first leg of \eqref{eq22}, we get
\begin{equation}\label{eq23}
\mf{X}_{[34]} (\Delta_1^{(2)}\otimes \Delta_2)(Z)=
(\Delta_1\otimes \id)(Z)_{[124]} Z_{[35]}
\widetilde{\mf{X}}_{[34]},
\end{equation}
while if we apply $\Delta_1$ to the second leg of \eqref{eq22}, we obtain
\begin{equation}\label{eq24}
\mf{X}_{[24]} \mf{X}_{[34]}
(\Delta_1^{(2)}\otimes \Delta_2)(Z)=
Z_{[14]} (\Delta_1\otimes\id)(Z)_{[235]}
\widetilde{\mf{X}}_{[24]}
\widetilde{\mf{X}}_{[34]}.
\end{equation}
Combining \eqref{eq23} and \eqref{eq24} and cancelling $\widetilde{\mf{X}}_{[34]}$ gives
\begin{equation}\label{eq25}
\mf{X}_{[24]} 
(\Delta_1\otimes \id)(Z)_{[124]} Z_{[35]}
=
Z_{[14]} (\Delta_1\otimes\id)(Z)_{[235]}
\widetilde{\mf{X}}_{[24]}.
\end{equation}
Similarly, applying $\Delta_2$ to the third and fourth leg of \eqref{eq22} leads to
\begin{equation}\label{eq26}
\mf{X}_{[24]}
Z_{[13]} (\id\otimes\Delta_2)(Z)_{[245]}
=
(\id\otimes\Delta_2)(Z)_{[134]}
Z_{[25]}
\widetilde{\mf{X}}_{[24]}.
\end{equation}
From equation \eqref{eq26} we deduce
\[
 (\id\otimes \Delta_2)(Z)_{[134]}=
Z_{[13]} \mf{X}_{[24]} (\id\otimes\Delta_2)(Z)_{[245]} \widetilde{\mf{X}}_{[24]}^* Z_{[25]}^*.
\]
Apply the flip $a\otimes b\otimes c\otimes d\otimes e\mapsto a\otimes c\otimes d \otimes b \otimes e$ to get
\[
 (\id\otimes \Delta_2)(Z)_{[123]}=
 Z_{[12]}
\mf{X}_{[43]} (\id\otimes\Delta_2)(Z)_{[435]} \widetilde{\mf{X}}_{[43]}^* Z_{[45]}^*.
\]
It follows that
\begin{equation}\label{eq28}
(\id\otimes \Delta_2)(Z)=Z\otimes v_2,
\end{equation}
where $v_2\in L^{\infty}(\Hh_2)$ is a unitary defined via $v_{2 [123]}=\mf{X}_{[21]} (\id\otimes \Delta_2)(Z)_{[213]} \widetilde{\mf{X}}^*_{[21]}Z^*_{[23]}$. Applying $\Delta_2$ to the third leg of \eqref{eq28} gives
\[
Z\otimes \Delta_2(v_2)=
(\id\otimes \Delta_2^{(2)})(Z)=
(\id\otimes \Delta_2)(Z)\otimes v_2=
Z\otimes v_2\otimes v_2,
\]
hence $v_2$ is a group-like unitary. 

Observe now that 
\[
(\id\otimes \Delta_2)(Z (1\otimes v_2^*))=
(Z\otimes v_2)(1\otimes v_2^*\otimes v_2^*)=
Z(1\otimes v_2^*)\otimes 1.
\]
Hence it follows from the ``invariants are constant'' technique \cite{Daw12}*{Theorem 2.1} that $Z(1\otimes v_2^*)=v_1\otimes 1$ for a unitary $v_1\in L^{\infty}(\Hh_1)$. Consequently \begin{equation}\label{eq30}
Z=v_1\otimes v_2.
\end{equation}
We can rewrite \eqref{eq25} as
\[
\mf{X}_{[24]} \Delta_1(v_1)_{[12]} v_{2 [4]}
v_{1 [3]} v_{2[5]} = 
v_{1[1]} v_{2[4]}
\Delta_1(v_1)_{[23]} v_{2 [5]}
\widetilde{\mf{X}}_{[24]}.
\]
Consequently, similarly as above we can introduce a unitary $\tilde{v}_1\in L^{\infty}(\Hh_1)$ such that 
\begin{equation}\label{eq29}
\Delta_1(v_1)=v_1\otimes \tilde{v}_1.
\end{equation}
As with $v_2$ and \eqref{eq28}, we deduce from here that $\tilde{v}_1$ is a group-like unitary. Using again equation \eqref{eq22} and equation \eqref{eq29} gives
\[
\mf{X}_{[23]} v_{1[1]} \tilde{v}_{1 [2]}
v_{2[3]} v_{2[4]}=
v_{1[1]}  v_{2 [3]}
v_{1[2]} v_{2[4]}\widetilde{\mf{X}}_{[23]}
\]
which simplifies to
\begin{equation}\label{eq31}
\mf{X}  (\tilde{v}_1\otimes v_{2}) =
(v_{1}\otimes  v_{2 })
 \widetilde{\mf{X}}.
 \end{equation}
Using that $\tilde{v}_1$ is a group-like, applying $\Delta_1$ to the first leg of \eqref{eq31} gives
\[
(\Delta_1(v_1)\otimes v_2)\widetilde{\mf{X}}_{[13]}
\widetilde{\mf{X}}_{[23]}=
\mf{X}_{[13]} \mf{X}_{[23]}
(\tilde{v}_1\otimes \tilde{v}_1\otimes v_2)=
\mf{X}_{[13]} 
(\tilde{v}_1\otimes v_1\otimes v_2)
\widetilde{\mf{X}}_{[23]}
=
(v_1\otimes v_1\otimes v_2)
\widetilde{\mf{X}}_{[13]} 
\widetilde{\mf{X}}_{[23]},
\]
and we conclude that $v_1$ is group-like as well. Equation \eqref{eq29} implies that $\tilde{v}_1=v_1$ and \eqref{eq31} ends the proof.
\end{proof}

Let us go back to our situation. Let $\widetilde{\mcX}_u\in \Mult(C^*(\G_1)\otimes C^*(\G_2))$ be the skew bicharacter from Lemma \ref{LemUnitSkewBich}, and let $\widetilde{\mcX}$ be its image in $\Mult(C^*_{\red}(\G_1)\otimes C^*_{\red}(\G_2))$. 
\begin{Lem}\label{lemma7}
The cocycles $\hat\Omega=\wh{\mc{X}}_{[32]}$ and $\widetilde{\Omega}=\widetilde{\mc{X}}_{[32]}$ are coboundary equivalent. More precisely, we have
\[
\hat{\Omega}\hat{\Delta}_{\otimes}(Z)=(Z\otimes Z) \widetilde{\Omega} 
\]
for $Z=\hat{\pi}_{\red,12}\hat{\pi}^{-1}_{12}(1\otimes 1)\in \hat{M}_1 \bar\otimes \hat{M}_2$.
\end{Lem}

\begin{proof}
Observe that $Z^* \hat{\pi}_{\red,12}(\hat{\pi}_{12}^{-1}(x))=\hat{\pi}_{\red}(x)$ for $x\in C^*(\G)$. Using this, we have
\[\begin{split}
&\quad\;
\widetilde{\Omega}=
(\hat{\pi}_{\red}\otimes \hat{\pi}_{\red})(\widetilde{\mc{X}}_{u [32]})=
(\hat{\pi}_{\red}\otimes \hat{\pi}_{\red})\widetilde{\Delta}_{u,12}(1\otimes 1)=
(Z^*\otimes Z^*)
(\hat{\pi}_{\red,12}\hat{\pi}^{-1}_{12}\otimes
 \hat{\pi}_{\red,12} \hat{\pi}^{-1}_{12})\widetilde{\Delta}_{u,12} (1\otimes 1)\\
&=
(Z^*\otimes Z^*)
(\hat{\pi}_{\red,12}\otimes 
 \hat{\pi}_{\red,12})
 \hat{\Delta}_{u,12}  \hat{\pi}^{-1}_{12}(1\otimes 1)
 =
(Z^*\otimes Z^*)
 \hat{\Delta}_{12}(\hat{\pi}_{\red,12}  \hat{\pi}^{-1}_{12}(1\otimes 1))\\
 &=
(Z^*\otimes Z^*)\hat{\Omega} 
 \hat{\Delta}_{\otimes }(\hat{\pi}_{\red,12}  \hat{\pi}^{-1}_{12}(1\otimes 1))=
(Z^*\otimes Z^*)\hat{\Omega}
 \hat{\Delta}_{\otimes }(Z).
\end{split}\]
\end{proof}

If we apply Lemma \ref{lemma6} and Lemma  \ref{lemma7} to the bicharacters $\wh{\mc{X}},\widetilde{\mc{X}}$, we obtain group-like unitaries $v_1\in \hat{M}_1,v_2\in \hat{M}_2$ such that $Z=v_1\otimes v_2$ and $\wh{\mc{X}}=(v_1\otimes v_2)\widetilde{\mc{X}}(v_1^*\otimes v_2^*)$. These unitaries $v_1,v_2$ can be seen as one-dimensional unitary representations, hence they lift uniquely to group-like unitaries $v_1^u\in \Mult(C^*(\G_1)), v_2^u\in \Mult(C^*(\G_2))$. It then easily follows that
\[
\wh{\mc{X}}_u=
(v_1^u\otimes v_2^u)\widetilde{\mc{X}}_{u }(v_1^{u *}
\otimes v_2^{u *}).
\]

Next, we correct the isomorphism $\hat{\pi}\colon C^*(\msG)\rightarrow \begin{pmatrix}
C^*(\G) & C^*(\G) \\ C^*(\G) & C^*(\G)
\end{pmatrix}$ by twisting it with an inner automorphism:
\begin{equation}\label{eq33}
\underline{\hat{\pi}}=\oon{Ad}\Bigl( \begin{pmatrix}
v_1^{u } \otimes v_2^{u } & 0 \\ 0 & 1\otimes 1
\end{pmatrix}\Bigr) \hat{\pi}\colon
C^*(\msG)\rightarrow \begin{pmatrix}
C^*(\G) & C^*(\G) \\ C^*(\G) & C^*(\G)
\end{pmatrix}.
\end{equation}
More explicitely,
\begin{equation}\label{eq34}
\underline{\hat{\pi}}
\Bigl(\,\begin{pmatrix}
a & b \\ c & d
\end{pmatrix}\,\Bigr)=
\begin{pmatrix}
\underline{\hat{\pi}}_{11}(a) & 
\underline{\hat{\pi}}_{12}(b)\\
\underline{\hat{\pi}}_{21}(c) & 
\underline{\hat{\pi}}_{22}(d)
\end{pmatrix}=
\begin{pmatrix}
(v^{u }_1\otimes v_2^{u })\hat{\pi}_{11}(a)
(v_1^u\otimes v_2^u)^* & 
(v^u_1\otimes v_2^u)\hat{\pi}_{12}(b)\\
\hat{\pi}_{21}(c)(v_1^u\otimes v_2^u)^* & 
\hat{\pi}_{22}(d)
\end{pmatrix}
\end{equation}

for $\begin{pmatrix}
a & b \\ c &d
\end{pmatrix}\in C^*(\msG)$. The following lemma now finishes the proof of Theorem \ref{TheoUnivLift}.

\begin{Lem}\label{lemma8}
Via the isomorphism $\underline{\hat{\pi}}$, the element $\hat{\Omega}_u$ is a universal lift of $\hat{\Omega}$ in the sense of Definition \ref{DefUnivLift2Coc}.
\end{Lem}

\begin{proof}
We first claim that $\msG$ has trivial universal linking C$^*$-algebra with isomorphism 
\[
\underline{\hat{\pi}}\colon C^*(\msG)\rightarrow \begin{pmatrix}
\hat{M}_u & \hat{M}_u \\
\hat{M}_u & \hat{M}_u
\end{pmatrix}
\] 
(see Definition \ref{def3}). It is clear that $\underline{\hat{\pi}}$ preserves the matrix structure, and is  the identity on the lower right hand corner. Next, we have for $x\in C^*(\G_{ij})$
\[
\eps_{\hat{M}_u} \underline{\hat{\pi}}_{ij}(x)=
\eps_{\hat{M}_u} \hat{\pi}_{ij}(x)=\hat{\eps}_{ij}(x)
\]
using that $\eps_{\hat{M}_u}(v_1^u\otimes v_2^u)=1$ in the first equality (because $v_i^u$ is a group-like unitary) and Lemma \ref{lemma5}. Finally, we need to check that $\underline{\hat{\pi}}$ respects reducing maps. Take $y\in C^*(\G_{12})$. We have
\[\begin{split}
&\quad\;
\hat{\pi}_{\red,12}(y)=
\hat{\pi}_{\red,12} \hat{\pi}^{-1}_{12}\hat{\pi}_{12}(y)=
\hat{\pi}_{\red,12} (\hat{\pi}^{-1}_{12}(1\otimes 1)
\hat{\pi}^{-1}_{22}\hat{\pi}_{12}(y))=
\hat{\pi}_{\red,12} (\hat{\pi}^{-1}_{12}(1\otimes 1)
\hat{\pi}_{12}(y))\\
&=
\hat{\pi}_{\red,12} \hat{\pi}^{-1}_{12}(1\otimes 1)
\hat{\pi}_{\red,22}\hat{\pi}_{12}(y)=
(v_1\otimes v_2) \hat{\pi}_{\red}\hat{\pi}_{12}(y)=
\hat{\pi}_{\red}( (v_1^u\otimes v_2^u) \hat{\pi}_{12}(y))=
\hat{\pi}_{\red} \underline{\hat{\pi}}_{12} (y)
\end{split}\]
using Lemma \ref{lemma7} and the explicit formula \eqref{eq34}. The claim now follows by a reasoning similar to the one presented in the proof of Lemma \ref{lemma5}.

To end the proof, we show that
\begin{equation}\label{eq35}
(\underline{\hat{\pi}}_{12}\otimes 
\underline{\hat{\pi}}_{12})\hat{\Delta}^u_{12}
\underline{\hat{\pi}}_{12}^{-1}(1\otimes 1)=\hat{\Omega}_u.
\end{equation}
Indeed, we have
\[\begin{split}
&\quad\;
(\underline{\hat{\pi}}_{12}\otimes 
\underline{\hat{\pi}}_{12})\hat{\Delta}^u_{12}
\underline{\hat{\pi}}_{12}^{-1}(1\otimes 1)=
(v_1^u\otimes v_2^u\otimes 
v_1^u\otimes v_2^u)(\hat{\pi}_{12}\otimes 
\hat{\pi}_{12})\hat{\Delta}^u_{12}
\hat{\pi}^{-1}_{12}(
v_1^{u *}\otimes v_2^{u *}
)\\
&=
(v_1^u\otimes v_2^u\otimes 
v_1^u\otimes v_2^u)
\widetilde{\Delta}_{u,12} (
v_1^{u *}\otimes v_2^{u *}
)=
(v_1^u\otimes v_2^u\otimes 
v_1^u\otimes v_2^u)
\widetilde{\Delta}_{u,12} (
1\otimes 1
)
\widetilde{\Delta}_{u,22} (
v_1^{u *}\otimes v_2^{u *}
)\\
&=
(v_1^u\otimes v_2^u\otimes 
v_1^u\otimes v_2^u)
\widetilde{\Omega}_u
\hat{\Delta}_{u,22}(v_1^{u *}\otimes v_2^{u *})=
(v_1^u\otimes v_2^u\otimes 
v_1^{u }\otimes v_2^{u })
\widetilde{\mc{X}}_{u [32]}
(v_1^{u *}\otimes v_2^{u *}\otimes 
v_1^{u *}\otimes v_2^{u *})\\
&=
\wh{\mc{X}}_{u [32]}=\hat{\Omega}_u.
\end{split}\]
\end{proof}

\section{Acknowledgements}
The work of JK was partially supported by FWO grant 1246624N.

\end{document}